\documentclass[a4paper,12pt]{amsart}
\usepackage{fullpage}

\usepackage{amsmath,amssymb,amsthm,enumitem,bm,pict2e,xcolor,tikz-cd,indentfirst,xparse,ytableau,stackengine,bbm,thmtools,tikz,wrapfig,mathtools,subcaption,makecell,needspace}
\usepackage[greek,english]{babel}
\usepackage[foot]{amsaddr}
\usepackage
{hyperref}

\usetikzlibrary{calc}

\DeclareMathSizes{10.95}{10}{7}{5}

\let\oldsubsection\subsection
\makeatletter
\renewcommand{\subsection}{\@startsection{subsection}{2}{\z@}%
  {-2.5ex\@plus -1ex \@minus -.2ex}%
  {0.5ex\@plus .2ex}%
  {\normalfont\bfseries}}
\makeatother

\hypersetup{
    colorlinks=true,
    linkcolor=blue,
    filecolor=magenta,      
    urlcolor=cyan,
}

\theoremstyle{definition}
\newtheorem{theorem}{Theorem}
\newtheorem*{theorem*}{Theorem}
\numberwithin{theorem}{section}
\newtheorem{defn}[theorem]{Definition}
\newtheorem*{defn*}{Definition}
\newtheorem{prop}[theorem]{Proposition}
\newtheorem{lemma}[theorem]{Lemma}
\newtheorem{remark}[theorem]{Remark}
\newtheorem{example}[theorem]{Example}
\newtheorem{cor}[theorem]{Corollary}

\setlist{leftmargin=2em}

\DeclareMathOperator{\sgn}{sgn}

\DeclareMathOperator{\grad}{grad}

\DeclareMathOperator{\grdim}{grdim}
\DeclareMathOperator{\Gr}{Gr}
\DeclareMathOperator{\init}{in}

\DeclareMathOperator{\Spec}{Spec}

\DeclareMathOperator{\Trop}{Trop}

\DeclareMathOperator{\Grob}{Gr\ddot{o}b}

\DeclareMathOperator{\Path}{path}
\DeclareMathOperator{\vpath}{vpath}

\newcommand{\la}{\lambda}
\newcommand{\om}{\omega}

\newcommand{\cM}{\mathcal M}

\newcommand{\cR}{\mathcal R}

\newcommand{\cA}{\mathcal A}

\newcommand{\cT}{\mathcal T}
\newcommand{\cK}{\mathcal K}
\newcommand{\cW}{\mathcal W}

\newcommand{\bC}{\mathbb{C}}

\newcommand{\bP}{\mathbb{P}}
\newcommand{\bR}{\mathbb{R}}
\newcommand{\bZ}{\mathbb{Z}}
\newcommand{\bM}{\mathbb{M}}

\newcommand{\rA}{\mathrm A}
\newcommand{\rC}{\mathrm C}
\newcommand{\al}{\langle}
\newcommand{\ar}{\rangle}
\newcommand{\bs}{\backslash}
\newcommand{\wt}{\widetilde}
\newcommand{\wh}{\widehat}
\newcommand{\mr}{\mathring}
\newcommand{\ol}{\overline}
\newcommand{\ph}{\varphi}
\newcommand{\mon}{\mathrm{mon}}
\newcommand{\tor}{\mathrm{tor}}

\newcommand{\pp}{{{+}{+}}}

\newcommand{\dotprec}{\mathrel{\dot\prec}}

\newcommand{\vv}{\mathbf v}

\DeclareFontFamily{U}{mathx}{\hyphenchar\font45}
\DeclareFontShape{U}{mathx}{m}{n}{
      <5> <6> <7> <8> <9> <10>
      <10.95> <12> <14.4> <17.28> <20.74> <24.88>
      mathx10
      }{}
\DeclareSymbolFont{mathx}{U}{mathx}{m}{n}
\DeclareMathSymbol{\bigtimes}{1}{mathx}{"91}

\mathcode`l="8000
\begingroup
\makeatletter
\lccode`\~=`\l
\DeclareMathSymbol{\lsb@l}{\mathalpha}{letters}{`l}
\lowercase{\gdef~{\ifnum\the\mathgroup=\m@ne \ell \else \lsb@l \fi}}%
\endgroup

\mathchardef\newbracket=\mathcode`)
\mathcode`)="8000
\begingroup
\catcode`) \active
\gdef){\nolinebreak\newbracket}
\endgroup

\mathchardef\newcomma=\mathcode`,
\mathcode`,="8000
\begingroup
\catcode`, \active
\gdef,{\nolinebreak\newcomma}
\endgroup

\title{Tropical cluster varieties of type C}

\author{Igor Makhlin}
\address{Technische Universität Berlin, Fakultät II Mathematik und Naturwissenschaften, Institut für Mathematik, FG Diskrete Mathematik/Geometrie}
\email{iymakhlin@gmail.com}

\begin{document}

\maketitle

\begin{abstract}
We explicitly describe the tropicalization of a cluster variety of finite type C, realizing it as the space of axially symmetric phylogenetic trees. We also find all occurring sign patterns of coordinates, for both the cluster variety and the cluster configuration space. We show that each of the corresponding signed tropicalizations is, combinatorially, dual to either a cyclohedron or an associahedron. As additional results, we construct Gröbner and tropical bases for the defining ideals of both varieties, and classify the arising toric degenerations. 
\end{abstract}

\section*{Introduction}

Connections between the theory of cluster algebras and tropical geometry were established already in the foundational series of papers by Fomin, Zelevinsky and Berenstein:~\cite{FomZel2002,FomZel2003,BerFomZel2005,FomZel2007}. Subsequently, these connections were investigated by many authors, giving rise to at least two extensively studied notions of tropicalization for cluster algebras and cluster varieties. \textit{Positive tropicalizations} were defined by Speyer and Williams in~\cite{SpeWil2005}, where the authors formulated a long-lasting conjecture concerning positive tropicalizations of finite type cluster varieties, only resolved in~\cite{AHHL2021,JaLoSt2021}. The related \textit{Fock--Goncharov tropicalizations} are due to several works of Fock and Goncharov including~\cite{FocGon2006,FocGon2009}, as well as the paper~\cite{GHKK2018} by Gross, Hacking, Keel and Kontsevich. This construction plays a central role in the influential Fock--Goncharov conjectures, see~\cite{GHKK2018}. Further work studying these two notions of tropicalized cluster varieties or similar structures includes~\cite{Qin2017,RieWil2019,SheWen2020,BosMohNCh2021,Man2021,DFGK2021,Bos2022}. 

However, much less is known about \textit{tropicalizations} of cluster varieties as commonly understood in tropical algebraic geometry today, a concept originating in~\cite{BieGro1984,Stu2002,SpeStu2004,Mikh2005,EinKapLin2006}. This object can be thought of as containing those mentioned above; it is variously defined in terms of coordinatewise valuation maps, limits of amoebas or as a subfan of the Gröbner fan. In~\cite{SpeStu2004}, Speyer and Sturmfels gave a remarkable explicit description of the tropicalization of the Grassmannian $\Gr(2,n)$, realizing it as the space of phylogenetic trees introduced in~\cite{RobWhi1996,BilHolVog2001}. This offers a direct connection to cluster algebra theory since the affine cone over $\Gr(2,n)$ arises as a cluster variety: it is the spectrum of a cluster algebra of finite type A, the prototypical and best understood instance of a cluster algebra. To the author's knowledge, this remains the only infinite family of cluster varieties with explicitly described tropicalizations. The goal of this paper is to obtain another such family by describing the tropicalizations of cluster varieties of finite type C.

An important role in our construction is played by the results of~\cite{CoxMakh2025}, where Cox and the author constructed a polyhedral fan termed the \textit{space of axially symmetric phylogenetic trees}. It was conjectured that the defined fan arises as the tropicalization of the \textit{binary geometry} or \textit{cluster configuration space} associated with a type C cluster algebra, see~\cite{AHHLT2023,AHHL2021}. In this paper, our approach is to work directly with the type C cluster variety\footnote{
Following one of the common conventions, we define the cluster variety as the affine spectrum of the cluster algebra. This terminology is used in the highly relevant source~\cite{AHHL2021} and by a variety of other authors, e.g.,~\cite{Mul2012,LamSpe2022,CouDuc2020}. The resulting affine scheme is, in general, different from the quasi-affine union of cluster tori, known as the \textit{cluster $\cA$-variety} or the \textit{cluster manifold}.
}, specifically, the variety described by Fomin and Zelevinsky in~\cite{FomZel2003}. The latter corresponds to a particularly natural and well-behaved choice of coefficients for the cluster algebra, which associates cluster variables to diagonals of a regular $2n$-gon and frozen variables to its sides. The resulting variety can be viewed as the type C counterpart of the affine cone over $\Gr(2,n)$. As our main result, we realize its tropicalization as a variation of the fan defined in~\cite{CoxMakh2025}. This is subsequently applied to cluster configuration spaces, resolving the conjectures stated in the latter paper. In turn, these results on cluster configuration spaces allow us to extend our description of tropicalized cluster varieties to all full rank geometric type cluster algebras of type C.

To present our results in more detail, we introduce some notation. For $n\ge 3$, set $N=\{1,\dots,n,\ol 1,\dots,\ol n\}$. The studied cluster algebra $\cA$ of type $\rC_{n-1}$ is generated by its set of cluster and frozen variables $\{\Delta_{a,b}\}_{(a,b)\in D}$, where $D$ is a certain subset of $N^2$. In particular, we have a surjection $\bC[x_{a,b}]_{(a,b)\in D}\to \cA$ with kernel $I$. The respective cluster variety is $X=\Spec\cA$ and we study  the \textit{tropical cluster variety} $\Trop X=\Trop I$, a polyhedral fan in $\bR^D$. 

An \textit{axially symmetric phylogenetic tree} (or \textit{ASPT}) is a tree with $2n$ leaves labeled by $N$ and no vertices of degree 2, that satifies a natural axial-symmetry condition. With every ASPT we associate a cone in $\bR^D$, together these cones form a polyhedral fan: the \textit{space of ASPTs}. 

Although the definition of the type C cluster algebra inherently relies on central symmetry, its tropicalization is instead described in terms of axial symmetry. Namely, the majority of this paper is dedicated to a proof of the following.

\begin{theorem}[Theorem~\ref{thm:main}]\label{thm:A}
The tropical cluster variety $\Trop X$ is the space of ASPTs.
\end{theorem}

We now explain the two main obstacles encountered in that regard. Theorem~\ref{thm:A} states that for $w\in\bR^D$, the initial ideal $\init_w I$ is monomial-free if and only if $w$ lies in the space of ASPTs. In the $\Gr(2,n)$ case treated in~\cite{SpeStu2004}, the ``if'' part of a similar equivalence holds because, for every maximal cone, the respective initial ideal can be explicitly realized as a toric ideal. The ``only if'' part is verified for $\Gr(2,n)$ by checking that, for weights outside of the tropicalization, at least one of the quadratic Plücker relations has a monomial initial form, i.e., these relations form a tropical basis. Unfortunately, neither of these two favorable circumstances carries over to type C: the initial ideals do not seem to admit a simple explicit description, and the quadratic generators do not form a tropical basis.

Instead, our approach is as follows. For the first implication, we construct a toric ideal that \textit{contains} the initial ideal corresponding to a maximal cone, see Lemma~\ref{lem:initinker}. In general, the initial ideal itself will not be toric or even prime; the toric degenerations are classified separately by Theorem~\ref{thm:toricdegens}. For the converse implication, we define certain cubic relations, see~\eqref{eq:sijk}, which need to be included in the tropical basis of $I$. The tropical basis claim then follows from an axial symmetry criterion for phylogenetic trees, given by Lemma~\ref{lem:keylemma}. Both of the mentioned Lemmas~\ref{lem:keylemma} and~\ref{lem:initinker} are proved via heavy combinatorial analysis of phylogenetic trees. These two proofs are postponed until the end of the paper and constitute its most difficult part.

After Theorem~\ref{thm:A} is established, we turn to the cluster configuration space $\cM$, which is a certain GIT quotient of the very affine part of $X$. We deduce from Theorem~\ref{thm:A} that $\Trop\cM$ is the lineality-space quotient of the space of ASPTs. We then show that the tropicalized cluster variety given by any full rank geometric type cluster algebra of type $\rC_{n-1}$ is, modulo lineality, also linearly equivalent to the space of ASPTs, see~Corollary~\ref{cor:fullrank}. 

Finally, with every sign pattern of coordinates that occurs in the real part of a variety, one associates the respective \textit{signed tropicalization}: a certain subfan of the tropicalization. The case of all coordinates positive corresponds to the well-studied \textit{positive tropicalization}~\cite{SpeWil2005}, while other signed tropicalizations serve as a natural generalization. We study sign patterns of coordinates occurring in the real parts of $\cM$ and $X$, as well as the signed tropicalizations corresponding to these sign patterns. We confirm another conjecture of~\cite{CoxMakh2025}, as well as a conjecture made in~\cite{AHHL2021} regarding the number of occurring sign patterns. Results concerning the cluster configuration space $\cM$ are summarized as follows, where \textit{dihedral orderings} of a set are its cyclic orderings considered modulo reversal.

\begin{theorem}[Theorems~\ref{thm:signpatterns} and~\ref{thm:signedtrops}]
There are $2^{n-2}(n+1)(n-1)!$ sign patterns that occur in $\cM$, each defining a single connected component of $\cM(\bR)$. These sign patterns are enumerated by those dihedral orderings of $N$ which are either centrally symmetric or axially symmetric. For the former dihedral orderings, the respective signed tropicalization is combinatorially equivalent to the dual fan of a cyclohedron, while for the latter it is combinatorially equivalent to the dual fan of an associahedron.
\end{theorem}
A version of this theorem for the cluster variety $X$ is provided by Corollary~\ref{cor:Xsigns}.

\oldsubsection*{Further questions}

To conclude the introduction, we sketch several questions arising naturally from our results. These questions are motivated by the rich properties of the space of phylogenetic trees $\Trop\Gr(2,n)$ and broadly fit into the problem of extending these properties to type C. Since the details are beyond the scope of this paper, the interested reader is referred to the cited sources for all relevant definitions and results.
\begin{itemize}
\item Is there a polytope $P$ such that $\Trop X$ is a subfan of the secondary fan of $P$ and every ASPT is the dual graph to the respective subdivision of $P$? If so, is there a relationship to valuated matroids? In type A, one considers the hypersimplex $\Delta_{2,n}$ and the Dressian subfan which parametrizes matroid subdivisions, see~\cite{Kap1993}.
\item The type A counterpart of $\cM$ is the moduli space $\cM_{0,n}$, which admits the Deligne--Mumford compactification $\ol\cM_{0,n}$. One way of obtaining the latter is via \textit{tropical compactification}: as the closure of $\cM_{0,n}$ inside the toric variety defined by the space of phylogenetic trees, see~\cite{Tev2007}. Does a similar construction exist in type C and provide a compactification $\ol{\cM}$ that is stratified by the face poset of $\Trop X$?
\item How does one interpret $\Trop X$ as a tropical moduli space? For example, every point of $\Trop X$ is a weighted phylogenetic tree and hence defines a tropical plane by~\cite[\S~6]{SpeStu2004}. What is special about tropical planes of this form?
\item Another role of the space of phylogenetic trees in matroid theory is as the Bergman fan of a complete graph, see~\cite{ArdKli2006}. What is a similar realization of the space of ASPTs?
\end{itemize}

\oldsubsection*{Acknowledgements} The author is thankful to Lara Bossinger, Shelby Cox, Nathan Ilten, Thomas Lam and Matt Larson for valuable discussions, and to the anonymous referees for their helpful comments. This work was partially supported by the SFB-TRR 195 ``Symbolic Tools in Mathematics and their Application''.

\section{Preliminaries: initial ideals and tropicalizations}\label{sec:prelim}

We start by recalling standard definitions and results from Gr\"obner theory and tropical geometry. We refer to~\cite[Chapters 2--3]{MacStu2015} for a comprehensive treatment of these topics.

For a finite set $D$, consider the polynomial ring $R = \bC[x_d]_{d\in D}$. A real weight $w \in \bR^D$ can be viewed as an $\bR$-grading on $R$ that takes the value $w_d$ on $x_d$. The \textit{initial form} $\init_w r$ of a polynomial $r\in R$ is its highest nonzero $w$-homogeneous component, i.e., that of maximal grading. For an ideal $I \subset R$, its \textit{initial ideal} $\init_w I$ is the ideal spanned by all initial forms $\init_w r$ with $r\in I$. We also recall that a \textit{Gröbner basis} of $I$ with respect to $w$ is a generating set $G$ of $I$ such that $\{\init_w r\}_{r\in G}$ generates $\init_w I$.

In algebraic geometry, initial ideals play an important role as a standard tool for obtaining flat degenerations. In the above setting, there exists a flat family of schemes over $\bC$ with its fiber over zero isomorphic to $\Spec(R/\init_w I)$ and all other fibers isomorphic to $\Spec(R/I)$. One says that the former scheme is a \textit{Gröbner degeneration} or an \textit{initial degeneration} of the latter.

When the ideal $I$ is homogeneous (with respect to total degree), this construction is particularly well-behaved, preserving many important properties, including Krull dimension and graded dimension: 
\begin{equation}\label{eq:presdims}
\dim(R/I)=\dim(R/\init_w I),\quad\grdim I=\grdim(\init_w I).    
\end{equation} 
Here, $\grdim I\in(\bZ_{\ge0})^{\bZ_{\ge0}}$ is the sequence whose $i$th element is equal to the dimension of the subspace of degree $i$ homogeneous polynomials in $I$. We also write $\grdim I\le\grdim J$ to denote that $(\grdim I)_i\le(\grdim J)_i$ for all $i$.

For any weight $w\in\bR^D$, the set of weights $w'$ such that $\init_{w'} I=\init_w I$ is a cone in $\bR^D$. Such cones are known as \textit{Gr\"obner cones} of $I$; in general, they need not be convex. However, the Gr\"obner cones of a homogeneous ideal $I$ are relatively open convex polyhedral cones that form a complete fan in $\bR^D$. This is the \textit{Gr\"obner fan} of $I$, denoted by $\Grob I$, see~\cite{MorRob1988} (we view fans as collections of relatively open cones).

Next, if the ideal $\init_w I$ is monomial-free, then the Gr\"obner cone of $I$ that contains $w$ is called a \textit{tropical cone} of $I$. For any $I$, all tropical cones are relatively open, convex and polyhedral, and also form a polyhedral fan.
\begin{defn}[\cite{SpeStu2004}]
For an ideal $I\subset R$, the polyhedral fan formed by all tropical cones of $I$ is the \textit{tropicalization} of $I$, denoted by $\Trop I$.
\end{defn}

In particular, the support $|\Trop I|$ consists of all $w$ such that $\init_w I$ is monomial-free. When $I$ is the defining ideal of an affine scheme $X$, the tropicalization $\Trop I$ is also often denoted by $\Trop X$. This is, however, a slight abuse of notation since $\Trop I$ is not determined by $\Spec(R/I)$.
Nonetheless, whenever $\Trop I$ is nonempty, it has dimension $\dim(R/I)=\dim\Spec(R/I)$, it also has pure dimension if $I$ is prime. 

We will also use the notion of tropical bases.

\begin{defn}
A \textit{tropical basis} of $I$ is a generating set $T$ of $I$ such that, for any $w\in\bR^D$, one has $w\in|\Trop I|$ if and only if $\init_w r$ is not a monomial for every $r\in T$.    
\end{defn}

Next, the positive tropicalization arises as a subfan of the tropicalization. Let $\bR_{>0}[x_d]_{d\in D}$ denote the semiring of nonzero polynomials with no negative coefficients. 
\begin{defn}[\cite{SpeWil2005}]
For an ideal $I\subset R$, its \textit{positive tropicalization} $\Trop_{>0} I$ is the subfan of $\Trop I$ supported on those $w\in|\Trop I|$ for which $\init_w I\cap\bR_{>0}[x_d]_{d\in D}=\varnothing$.
\end{defn}

In other words, $w\in|\Trop_{> 0} I|$ if and only if every element of $\init_w I$ has both positive and negative coefficients. Let $X \subset \bC^D$ be the zero set of $I$ and let $X_{>0} = X \cap \bR_{>0}^D$ be its totally positive part. Clearly, if $X_{>0}$ is nonempty, then $\Trop_{> 0} I$ is also nonempty. The converse is not true in general.

A natural generalization of positive tropicalization is given by signed tropicalizations. Below, a \textit{sign pattern} is a vector with coordinates in $\{1,-1\}$.
\begin{defn}\label{def:signedtrop}
For a sign pattern $\nu\in\{1,-1\}^D$, consider the automorphism $\epsilon_\nu$ of $R$ taking $x_d$ to $\nu_d x_d$. The \textit{signed tropicalization} $\Trop_\nu I$ is the positive tropicalization $\Trop_{> 0} \epsilon_\nu(I)$.
\end{defn}

Note that $\Trop\epsilon_\nu(I)=\Trop I$, hence $\Trop_\nu I$ is also a subfan of $\Trop I$. Set 
\begin{equation}\label{eq:Xnu}
X_\nu=\{p\in X(\bR)=X\cap\bR^D\mid\sgn p_d=\nu_d\text{ for all }d\in D\}.    
\end{equation}
We say that $\nu$ \textit{occurs} in $X$ if $X_\nu$ is nonempty. One sees that $\Trop_\nu I$ is nonempty if the sign pattern $\nu$ occurs in $X$.

Finally, the above definitions work equally well for ideals in the ring of Laurent polynomials $\mr R=\bC[x_d^{\pm1}]_{d\in D}$. Given an ideal $\mr I\subset\mr R$, one may similarly define the initial ideal $\init_w\mr I$ with respect to $w\in\bR^D$. This leads to the notion of Gröbner cones and tropical cones of $\mr I$, which allows us to define $\Trop\mr I$ as a fan in $\bR^D$ (with $w\in|\Trop\mr I|$ if and only if $\init_w\mr I\neq \mr R$). Again, $\Trop\mr I$ is often denoted by $\Trop\mr X$, where $\mr X=\Spec(\mr R/\mr I)$ is the respective very affine scheme. The definitions of Gröbner bases, tropical bases, positive tropicalizations and signed tropicalizations are also extended straightforwardly. Note that if $I=\mr I\cap R$, then $\Trop\mr I=\Trop I$, $\Trop_{>0}\mr I=\Trop_{>0} I$ and $\Trop_\nu\mr I=\Trop_\nu I$. Furthermore, in Subsection~\ref{sec:TropM} we will consider the case of an ideal in the ring of polynomials over a ring of Laurent series (i.e., only some of the variables are invertible). This case is completely analogous to the cases of polynomials or Laurent polynomials.

\section{Type C cluster algebras}\label{sec:clusteralgs}

Throughout the paper, we work with the (complex form of the) \textit{special cluster algebra of type C} discussed by Fomin and Zelevinsky in~\cite[\S~12.3]{FomZel2003}. In Subsection~\ref{sec:TropM}, we will see that any other full rank geometric type cluster algebra of the same finite type provides the same tropicalization modulo lineality space. However, the above choice will allow for particularly natural forms of our results.

In this section, we prove several preparatory structural results about type C cluster algebras. We first obtain a presentation of the algebra in terms of generators and relations. Then we construct a monomial initial ideal of the ideal of relations; this also provides a basis in the cluster algebra different from the cluster-monomial basis. Finally, we discuss a family of automorphisms that is parametrized by the type BC Coxeter group.\footnote{
The author was not able to find any of these results in the literature despite consulting with experts. References would be highly appreciated if they do exist.}

\subsection{Generators and relations}\label{sec:genrel}

We fix an integer $n\ge 3$ and denote the cluster algebra of type $\rC_{n-1}$ by $\cA$. Rather than giving the original definition of $\cA$ in terms of mutations, we recall a realization of $\cA$ as of a subalgebra in a polynomial ring, following~\cite{FomZel2003}.

Consider the $2n$-element set $N=\{1,\dots,n,\ol 1,\dots,\ol n\}$. For $i\in[1,n]$ we use the convention $\ol{\ol i}=i$ and $|\ol i|=|i|=i$. Define the $n^2$-element set $D\subset N^2$ as 
\[D=\{(i,j)\mid 1\le i<j\le n\}\sqcup\{(i,\ol j)\mid 1\le i\le j\le n\}.\]
We note that $D$ can be interpreted as follows. The set of pairs $(a,b)\in N^2$ such that $a\neq b$ is equipped with the equivalence relation 
\[(a,b)\sim(b,a)\sim(\ol a,\ol b)\sim(\ol b,\ol a);\] 
$D$ is a set of representatives for this relation. Accordingly, we introduce the following convention used throughout the paper. Whenever $(o_{a,b})_{(a,b)\in D}$ is a $D$-indexed family, we extend the notation $o_{a,b}$ by symmetry to all pairs $a\neq b$ in $N$, with the understanding that $o_{a,b}=o_{b,a}=o_{\ol a,\ol b}=o_{\ol b,\ol a}$. This convention applies, for example, to the coordinates of a point $(w_{a,b})_{(a,b)\in D}\in\bR^D$, and to the generators $\Delta_{a,b}$ of $\cA$ considered next.

As a $\bC$-algebra, $\cA$ is generated by $n^2$ elements $\Delta_{a,b}$, $(a,b)\in D$, which form the set of all cluster variables and frozen variables. Let $S$ denote the polynomial ring $\bC[z_{t,i}]_{t\in\{1,2\},\, i\in[1,n]}$ in $2n$ variables. We use the following as a definition of $\cA$.

\begin{prop}[{\cite[Proposition 12.13]{FomZel2003}}]\label{prop:AinS1}
There is an injective homomorphism $\iota$ from $\cA$ to $S$ defined on the generators by
\begin{align*}
\iota(\Delta_{i,j}) & = z_{1,i}z_{2,j}-z_{1,j}z_{2,i},\quad 1\le i<j\le n,\\
\iota(\Delta_{i,\ol j}) & = z_{1,i}z_{1,j}+z_{2,i}z_{2,j},\quad 1\le i\le j\le n.
\end{align*}
\end{prop}

Set $1\prec\dots\prec n\prec\ol 1\prec\dots\prec\ol n$. Consider a matrix $Z$, with columns labeled by $N$ according to $\prec$ and rows labeled by $\{1,2\}$, such that $Z_{1,i}=Z_{2,\ol i}=z_{1,i}$ and $Z_{2,i}=-Z_{1,\ol i}=z_{2,i}$ for all $i\in[1,n]$.
We observe that the embedding $\iota$ takes $\Delta_{a,b}$ to the minor of $Z$ spanned by columns $a$ and $b$.


To realize $\cA$ via generators and relations, we first consider the cluster algebra $\wh \cA$ of type $\rA_{2n-3}$ (see~\cite[Example 12.6]{FomZel2003}). It is generated by elements $\wh\Delta_{a,b}$ with $a,b\in N$ and $a\prec b$. Since $\wh\cA$ is the Pl\"ucker algebra of the Grassmannian $\Gr(2,2n)$, it admits an embedding into $\wh S=\bC[z_{t,a}]_{t\in\{1,2\},\,a\in N}$, which takes $\wh\Delta_{a,b}$ to $z_{1,a}z_{2,b}-z_{1,b}z_{2,a}$. Consider the surjection from $\wh S$ to $S$ which, for $i\in[1,n]$, preserves $z_{1,i}$ and $z_{2,i}$, while $z_{1,\ol i}\mapsto -z_{2,i}$ and $z_{2,\ol i}\mapsto z_{1,i}$. In view of Proposition~\ref{prop:AinS1}, this restricts to a surjection $q\colon\wh\cA\to\cA$ given by $q(\wh\Delta_{a,b})=\Delta_{a,b}$.

The defining relations of $\wh\cA$ are the three-term Pl\"ucker relations
\[\wh\Delta_{a,b}\wh\Delta_{c,d}+\wh\Delta_{a,d}\wh\Delta_{b,c}=\wh\Delta_{a,c}\wh\Delta_{b,d}\] 
for all quadruples $a\prec b\prec c\prec d$ in $N$. In view of the surjection $q$, the relations
\[\Delta_{a,b}\Delta_{c,d}+\Delta_{a,d}\Delta_{b,c}=\Delta_{a,c}\Delta_{b,d}\] 
with $a\prec b\prec c\prec d$ hold in $\cA$. We claim that this is a complete set of relations for $\cA$.

Set $R=\bC[x_{a,b}]_{(a,b)\in D}$. We have a surjection $p\colon R\to\cA$ taking $x_{a,b}$ to $\Delta_{a,b}$, denote its kernel by $I$. For $a\prec b\prec c\prec d$ we set
\[r_{a,b,c,d} = x_{a,b}x_{c,d} + x_{a,d}x_{b,c} - x_{a,c}x_{b,d}\in R.\]
If $(a',b',c',d')$ is a permutation of $(a,b,c,d)$, we also denote $r_{a',b',c',d'}=r_{a,b,c,d}$.

\begin{prop}\label{prop:Igens}
$I$ is generated by the expressions $r_{a,b,c,d}$ for all $a\prec b\prec c\prec d$.
\end{prop}

\begin{remark}
It is clear that the $2n\choose 4$ quadratic generators provided by Proposition~\ref{prop:Igens} are not pairwise distinct. Specifically, $r_{a,b,c,d}=r_{\ol a,\ol b,\ol c,\ol d}$. This leaves us with a total of $\frac{n(n-1)(2n^2-4n+3)}6$ distinct relations which are seen to be linearly independent. Among them are the $\frac{n(n-1)(n^2-2n+3)}6$ exchange relations of type $\rC_{n-1}$: those $r_{a,b,c,d}$ for which either $d\prec\ol a$, or $c=\ol a$ and $d=\ol b$, see~\cite[\S~12.3]{FomZel2003}.
\end{remark}

To prove Proposition~\ref{prop:Igens}, we introduce some further notation. Consider the polynomial ring $\wh R=\bC[\wh x_{a,b} \mid a\prec b]$ in $2n\choose 2$ variables. We have a surjection $\wh p\colon\wh R\to \wh\cA$ taking $\wh x_{a,b}$ to $\wh\Delta_{a,b}$, its kernel $\wh I$ is the Pl\"ucker ideal. We also have a surjection $q_x\colon\wh R\to R$ taking $\wh x_{a,b}$ to $x_{a,b}$. By construction, $p\circ q_x=q\circ \wh p$, while Proposition~\ref{prop:Igens} claims that $q_x(\wh I)=I$.

Denote by $\wh I_\mon$ the monomial ideal in $\wh R$ generated by products $\wh x_{a,c}\wh x_{b,d}$ such that $a\prec b\prec c\prec d$. As is well known (see~\cite{SpeStu2004}), $\wh I_\mon$ is an initial ideal of $\wh I$. 
Its type C counterpart is the monomial ideal $I_\mon=q_x(\wh I_\mon)$ in $ R$. As shown in \cite[Corollary~5.3.2]{IltNCTref2025}, $I_\mon$ is an initial ideal of $I$. 
In particular, the graded dimensions of $I$ and $I_\mon$ coincide.

\begin{proof}[Proof of Proposition~\ref{prop:Igens}]
We have already seen that all expressions $r_{a,b,c,d}$ lie in $I$. Thus, the ideal $J$ generated by these expressions is contained in $I$.

For $a,b\in N$ we set $|a-b|=k+1$ where $k$ is the number of elements lying strictly between $a$ and $b$ in the order $\prec$. Consider $w\in\bR^D$ such that $w_{a,b}=\ln(|a-b|)$. 
The ideal $I_\mon$ is generated by products $x_{a,c}x_{b,d}$ such that $a\prec b\prec c\prec d$. However, for such a quadruple, our choice of $w$ provides
\begin{equation*}
\init_w r_{a,b,c,d}=-x_{a,c}x_{b,d}.    
\end{equation*} 
Consequently, $I_\mon\subset \init_w J$, where $J=\al r_{a,b,c,d}\ar_{a\prec b\prec c\prec d}$ as before. We obtain
\[\grdim J=\grdim(\init_w J)\ge\grdim(I_\mon)=\grdim I.\]
Since $J\subset I$, we must have $J=I$.
\end{proof}


We denote $X=\Spec\cA$, this is the \textit{type C cluster variety}: an irreducible affine variety of dimension $2n-1$. The main goal of this paper is to study the tropicalization $\Trop X=\Trop I$. This is a polyhedral fan in $\bR^D$ which we term the \textit{type C tropical cluster variety}.

\begin{remark}
In~\cite{FomZel2003}, the variety $X$ is realized as a certain space of $SO_2(\bC)$-invariants in $\mathrm{Mat}_{2,n}(\bC)$. Proposition~\ref{prop:Igens} provides a different realization: $X$ is the section of the affine cone over $\Gr(2,2n)$ by a linear subspace.
\end{remark}

\subsection{Another monomial ideal}\label{sec:monideal}

Consider an alternative ordering of the set $N$:
\[1\dotprec \dots\dotprec  n\dotprec \ol n\dotprec \dots\dotprec \ol 1.\]
In other words, $a\dotprec  b$ if and only if $s(a)\prec s(b)$, where $s$ is the permutation of $N$ such that $s(i)=i$ and $s(\ol i)=\ol{n+1-i}$ for $i\in[1,n]$.

We define the ideal $\wh I'_\mon\subset\wh R$ generated by products $\wh x_{a,c}\wh x_{b,d}$ for which $a\dotprec b\dotprec c\dotprec d$, where we use the convention $\wh x_{a,b}=-\wh x_{b,a}$ (note that $x_{a,b}\neq q_x(\wh x_{a,b})$ for $a\succ b$). The ideal $\wh I'_\mon$ is obtained from $\wh I_\mon$ by the natural action of the permutation $s$ given by $s(\wh x_{a,b})=\wh x_{s(a),s(b)}$. In particular, since $s(\wh I)=\wh I$, we see that $\wh I'_\mon$ is also an initial ideal of $\wh I$.
We now introduce the ideal $I'_\mon=q_x(\wh I'_\mon)$ in $R$. 
This ideal is generated by products $x_{a,c}x_{b,d}$ for which $a\dotprec b\dotprec c\dotprec d$. 

While $\wh I'_\mon$ differs from $\wh I_\mon$ only by a permutation of the variables, the same cannot be said about $I'_\mon$ and $I_\mon$. Their structures are rather different, for example, $I_\mon$ is square-free (as a monomial ideal) but $I'_\mon$ contains $x_{1,\ol 2}^2=q_x(\wh x_{1,\ol 2}\wh x_{2,\ol 1})$. Nonetheless,
 we prove the following.

\begin{lemma}\label{lem:ImonI}
$I'_\mon$ is an initial ideal of $I$. Consequently, products $\Delta_{a_1,b_1}\dots\Delta_{a_m,b_m}$ such that $x_{a_1,b_1}\dots x_{a_m,b_m}\notin\wh I'_\mon$ form a $\bC$-basis of $\cA$.
\end{lemma}

\begin{remark}
The products $\Delta_{a_1,b_1}\dots\Delta_{a_m,b_m}$ such that $x_{a_1,b_1}\dots x_{a_m,b_m}\notin I_\mon$ are known as \textit{cluster monomials} and form a different basis of $\cA$, see~\cite[Theorem 4.3.1]{IltNCTref2025}.
\end{remark}

To prove the lemma, we construct a toric initial ideal of $I$ that itself has $I_\mon$ as an initial ideal. Consider the polynomial ring $U=\bC[u_i,v_j]_{i,j\in[1,n]}$. We define a homomorphism $\pi\colon R\to U$ by
\begin{align*}
\pi(x_{i,j})&=v_iv_j\prod_{k=i}^{j-1} u_k,& &1\le i<j\le n,\\
\pi(x_{i,\ol j})&=v_iv_j\prod_{k=i}^n u_k\prod_{k=j}^n u_k,& &1\le i\le j\le n.
\end{align*}
This homomorphism can be interpreted as follows. Consider the path graph $G$ in Figure~\ref{fig:graphG}. Its vertex set is $N$, ordered according to $\dotprec$. For $i\in[1,n-1]$, the edge $\{i,i+1\}$ is denoted by $e_i$ and the edge $\{\ol i,\ol{i+1}\}$ by $e_{\ol i}$, while the edge $\{n,\ol n\}$ is denoted by $e_n$. For $a,b\in N$, let $\Path(a,b)$ denote the set of edges lying on the path between vertices $a$ and $b$. Then $\pi(x_{a,b})$ is the product of $v_a$, $v_{|b|}$ and $\prod_{e_c\in\Path(a,b)}u_{|c|}$, with an additional factor of $u_n$ if $e_n\in\Path(a,b)$. The kernel of $\pi$ is a toric ideal that we denote by $I_\tor$.

\begin{figure}[h!tbp]

\begin{tikzpicture}[x=15mm, y=10mm]
\node[circle, fill=black, inner sep=2] (a) at (0,0) {};
\node[circle, fill=black, inner sep=2] (b) at (1,0) {};
\node[circle, fill=black, inner sep=2] (c) at (2.5,0) {};
\node[circle, fill=black, inner sep=2] (d) at (3.5,0) {};
\node[circle, fill=black, inner sep=2] (e) at (5,0) {};
\node[circle, fill=black, inner sep=2] (f) at (6,0) {};

\node at (0,0.3) {$1$};
\node at (0.5,0.3) {$e_1$};
\node at (1,0.3) {$2$};
\node at (2.5,0.3) {$n$};
\node at (3,0.3) {$e_n$};
\node at (3.5,0.3) {$\ol n$};
\node at (5,0.32) {$\ol 2$};
\node at (5.5,0.3) {$e_{\ol 1}$};
\node at (6,0.31) {$\ol 1$};

\draw (a) -- (b)
      (c) -- (d)
      (e) -- (f);
\draw[dashed] 
      (b) -- (c)
      (d) -- (e);
\end{tikzpicture} 

\caption{The graph $G$.}\label{fig:graphG}

\end{figure}
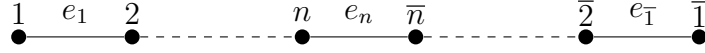

Before proceeding with the proof, we provide a table for the reader's convenience, listing the four ideals we consider in $R$. The indicated generating set of $I_\tor$ will be obtained in the proof of Lemma~\ref{lem:ItorI}.
\begin{table}[h!tbp]
\centering
\begin{tabular}{c|c|c|c|c}
Ideal  &  $I$ & $I_\mon$ & $I'_\mon$ & $I_\tor$\\ 
\hline
Generators  & \makecell{$r_{a,b,c,d}$,\\ \scalebox{1}{$a\prec b\prec c\prec d$}} & \makecell{$x_{a,c}x_{b,d}$,\\ \scalebox{1}{$a\prec b\prec c\prec d$}} & \makecell{$x_{a,c}x_{b,d}$,\\ \scalebox{1}{$a\dotprec b\dotprec c\dotprec d$}} & \makecell{$x_{a,c}x_{b,d}-x_{a,d}x_{b,c}$,\\ \scalebox{1}{$a\dotprec b\dotprec c\dotprec d$}}  
\end{tabular}
\vspace{-3mm}
\end{table}

\begin{lemma}\label{lem:ImonItor}
$I'_\mon$ is an initial ideal of $I_\tor$.    
\end{lemma}
\begin{proof}
For every quadruple $a\dotprec b\dotprec c\dotprec d$, the binomial $x_{a,c}x_{b,d}-x_{a,d}x_{b,c}$ lies in $I_\tor$. Furthermore, if we define $w'\in\bR^D$ by $w'_{a,b}=\ln{|\Path(a,b)|}$, then 
\[\init_{w'}(x_{a,c}x_{b,d}-x_{a,d}x_{b,c})=x_{a,c}x_{b,d}.\]
This shows that $I'_\mon\subset\init_{w'} I_\tor$.

By~\eqref{eq:presdims}, for the reverse inclusion it suffices to check that $\grdim I'_\mon\ge\grdim I_\tor$. For $m\ge 0$, let $\cM_m$ denote the set of degree $m$ monomials in $R$. By construction, $(\grdim I_\tor)_m$ is equal to $|\cM_m| - |\pi(\cM_m)|$. Meanwhile, $(\grdim I'_\mon)_m=|\cM_m|-|\cM_m\bs I'_\mon|$. Thus, we are to prove that $|\pi(\cM_m)|\ge|\cM_m\bs I'_\mon|$, which we do by showing that $\pi$ is injective on $\cM_m\bs I'_\mon$, i.e., that a monomial $M\notin I'_\mon$ is determined uniquely by $\pi(M)$.

Consider 
$M =x_{a_1,b_1}\dots x_{a_m,b_m}\notin I'_\mon$,
where all $(a_i,b_i)\in D$. Let $A(M)$ denote the multiset of subsegments of $[1,n+1]$ obtained as the multiset union of $\{[a_i,b_i]\}_{b_i\in[1,n]}$ and $\bigcup_{b_i\notin[1,n]}\{[a_i,n+1],[\ol{b_i},n+1]\}$, where all braces and unions are again multiset operations.
The multiset $A(M)$ determines $M$. Indeed, let $x_{i_1,\ol{j_1}}x_{i_2,\ol{j_2}}$ divide $M$, where $1\le i_t\le j_t\le n$. Since $M\notin I'_\mon$, the segments $[i_1,j_1]$ and $[i_2,j_2]$ are either disjoint or intersect in a single point that is an endpoint of both: in all other cases one finds $a\dotprec b\dotprec c\dotprec d$ such that $q_x(\wh x_{a,c}\wh x_{b,d})=x_{i_1,\ol{j_1}}x_{i_2,\ol{j_2}}$. Hence, if the elements of $A(M)$ ending in $n+1$ are $[i_1,n+1],\dots,[i_{2l},n+1]$ with $i_1\le\dots\le i_{2l}$, then
\[M=\prod_{j=1}^l x_{i_{2j-1},\ol{i_{2j}}}\cdot \prod_{[a_i,b_i]\in A(M)\,|\, b_i\in[1,n]} x_{a_i,b_i}.\]

Thus, we show that $A(M)$ can be recovered from $\pi(M)$. For $i\in[1,n+1]$, let $s_i$ and $t_i$ denote the number of elements in $A(M)$ that start and end in $i$ respectively. To recover $A(M)$ it suffices to determine $s_i$ and $t_i$ for all $i$. Indeed, $M\notin I'_\mon$ implies that $[a_1,b_1],[a_2,b_2]\in A(M)$ cannot satisfy $a_1<a_2<b_1<b_2$. Hence, to determine $A(M)$ from the multisets of starting and ending points, we may go through all starting points in decreasing order, successively pairing each starting point $a$ with the smallest unpaired ending point that is greater than $a$. In other words, if we place an opening bracket in every starting point and a closing bracket in every ending point, with the closing brackets preceding the opening ones within each position, we obtain a regular bracket sequence.

For $i\in[1,n]$, let $f_i$ and $g_i$ denote the exponents of, respectively, $u_i$ and $v_i$ in $\pi(M)$. Note that $f_i$ is the number of segments in $A(M)$ that contain $[i,i+1]$, while $g_i$ is the number of those that have an endpoint in $i$. Setting $f_0=0$, we obtain for $i\in[1,n]$:
\[
\begin{cases}
s_i+t_i = g_i,\\
s_i-t_i = f_i-f_{i-1}.
\end{cases}
\]
This allows us to determine $s_i$ and $t_i$ from $\pi(M)$. Finally, $s_{n+1}=0$ and $t_{n+1}=f_n$.
\end{proof}

\begin{lemma}\label{lem:ItorI}
$I_\tor$ is an initial ideal of $I$.    
\end{lemma}
\begin{proof}
As seen in the proof of Lemma~\ref{lem:ImonItor}, the initial forms $\init_{w'}(x_{a,c}x_{b,d}-x_{a,d}x_{b,c})$ with $a\dotprec b\dotprec c\dotprec d$ generate $\init_{w'} I_\tor=I'_\mon$. This means that such binomials $x_{a,c}x_{b,d}-x_{a,d}x_{b,c}$ form a Gr\"obner basis of $I_\tor$ with respect to $w'$, in particular, generate $I_\tor$. We first check that every such binomial lies in $\init_{w''}I$, where $w''\in\bR^D$ is defined by $w''_{a,b}=|\Path(a,b)|$. 

Indeed, if $a\dotprec b\dotprec c\dotprec d$, then 
\[\init_{w''} r_{a,b,c,d}=\pm x_{a,c}x_{b,d} \pm x_{a,d}x_{b,c}.\]
It remains to check that the two signs in the above binomial are distinct. By comparing the orders $\dotprec $ and $\prec$, we see that there are four possibilities: $a\prec b\prec c\prec d$, $a\prec b\prec d\prec c$, $a\prec d\prec c\prec b$ and $d\prec c\prec b\prec a$. The desired condition holds in all four cases.

Thus, $I_\tor\subset\init_{w''} I$. Next, note that for any variable $x_{a,b}$ we have $\deg_{u_1}\pi(x_{a,b})=\deg_{v_1}\pi(x_{a,b})$. Hence, we have a hyperplane in $\bR^{2n}$ that contains the exponent vectors of all monomials of the form $\pi(x_{a,b})$. This means that the toric ring $\pi(R)\cong R/I_\tor$ has Krull dimension at most $2n-1$. Since $I_\tor$ is prime, the quotient of $R$ by any strictly larger ideal must have Krull dimension less than $2n-1$. However, 
\[\dim (R/\init_{w''} I)=\dim (R/I)=\dim\cA=2n-1,\] and we conclude that $\init_{w''} I=I_\tor$.
\end{proof}

\begin{proof}[Proof of Lemma~\ref{lem:ImonI}]
The first claim follows directly from Lemmas~\ref{lem:ImonItor} and~\ref{lem:ItorI} by transitivity. The second claim follows from the first as a general property of monomial initial ideals of homogeneous ideals.    
\end{proof}

\subsection{Automorphisms and group actions}

Let $W$ denote the group of those permutations $\tau$ of the set $N$ that satisfy $\tau(\ol a)=\ol{\tau(a)}$---this is a Coxeter group of type $\mathrm{BC}_n$. For $\tau\in W$ we define the automorphism $h_\tau$ of $R$ such that for $(a,b)\in D$ one has
\begin{equation*}\label{eq:htau}
h_\tau\colon x_{a,b}\mapsto\varepsilon_1(b,\tau)\varepsilon_2(a,b,\tau) x_{\tau(a),\tau(b)},    
\end{equation*}
where $\varepsilon_1(b,\tau),\,\varepsilon_2(a,b,\tau)\in\{1,-1\}$ are determined as follows:
\begin{itemize}
\item $\varepsilon_1(b,\tau)=-1$ if and only if $b\notin[1,n]$ and $\tau(b)\in[1,n]$;
\item $\varepsilon_2(a,b,\tau)=-1$ if and only if $\tau(b)\prec\tau(a)$.
\end{itemize}

The reasoning behind this definition is as follows. From Proposition~\ref{prop:AinS1} and its accompanying discussion, recall the ring $S$, the embedding $\iota\colon\cA\to S$, and the matrix $Z$. Consider the automorphism $g_\tau$ of $S$ such that $g_\tau(z_{t,i})=Z_{t,\tau(i)}$ for $t\in\{1,2\}$ and $i\in[1,n]$. In other words, the submatrix of $g_\tau(Z)$ spanned by columns $i$ and $\ol i$ is $\begin{psmallmatrix} z_{1,\tau(i)}&-z_{2,\tau(i)}\\ z_{2,\tau(i)}&z_{1,\tau(i)} \end{psmallmatrix}$ if $\tau(i)\in[1,n]$ and $\begin{psmallmatrix} -z_{2,|\tau(i)|}&-z_{1,|\tau(i)|}\\ z_{1,|\tau(i)|}&-z_{2,|\tau(i)|} \end{psmallmatrix}$ otherwise.
Thus, the matrix $g_\tau(Z)$ is obtained from $Z$ by permuting the columns according to $\tau$ and then, for every $b\notin[1,n]$, multiplying column $b$ by $\varepsilon_1(b,\tau)$. Therefore, 
\begin{equation}\label{eq:gtauongens}
g_\tau(\iota(\Delta_{a,b}))=\pm\iota(\Delta_{\tau(a),\tau(b)}),
\end{equation}
implying that $g_\tau$ preserves the subalgebra $\iota(\cA)$ and $\iota^{-1}\circ g_\tau\circ\iota$ is an automorphism of $\cA$.
Now, for any distinct $c,d\in N$ one has
$\iota(\Delta_{c,d})=\varepsilon\begin{vsmallmatrix} Z_{1,c}&Z_{1,d}\\ Z_{2,c}&Z_{2,d} \end{vsmallmatrix}$, where $\varepsilon=1$ if $c\prec d$ and $\varepsilon=-1$ otherwise. Therefore, the sign in~\eqref{eq:gtauongens} evaluates precisely to $\varepsilon(a,b,\tau)$. We conclude that $\iota\circ p\circ h_\tau=g_\tau\circ\iota\circ p$, meaning that $h_\tau$ lifts the automorphism $\iota^{-1}\circ g_\tau\circ\iota$ to $R$. In particular, $\ker(p\circ h_\tau)=\ker p$, and we obtain

\begin{prop}\label{prop:htaupreservesI}
For every $\tau\in W$ one has $h_\tau(I)=I$.    
\end{prop}

It is easily checked that we do not, in general, have $h_\tau\circ h_\sigma=h_{\tau\sigma}$ or $g_\tau\circ g_\sigma=g_{\tau\sigma}$, which means that we do not obtain a $W$-action on $\cA$ and $X$, see also Remark~\ref{rem:groupaction}.

The group $W$ does, however, act on the set $D$ with $\tau\in W$ taking $(a,b)\in D$ to the unique element of 
\[\{(\tau(a),\tau(b)),(\tau(b),\tau(a)),(\tau(\ol a),\tau(\ol b)),(\tau(\ol b),\tau(\ol a))\}\cap D.\] This induces an action on $\bR^D$ given by $\tau(w)_{\tau(a),\tau(b)}=w_{a,b}$.
We also observe that for every monomial $M\in R$ and $w\in\bR^D$, the grading of $M$ with respect to $w$ is equal to the grading of $h_\tau(M)$ with respect to $\tau(w)$. In combination with Proposition~\ref{prop:htaupreservesI}, this provides

\begin{prop}\label{prop:htauinI}
For every $\tau\in W$ and $w\in\bR^D$, one has $h_\tau(\init_w I)=\init_{\tau(w)} I$.
\end{prop}

Since $h_\tau(\init_w I)$ is monomial-free (resp.\ monomial) if and only if $\init_w I$ is monomial-free (resp.\ monomial), we have the following consequence.
\begin{cor}\label{cor:fansWinv}
Both the tropical cluster variety $\Trop X$ and the Gr\"obner fan of $I$ are preserved by the action of $W$ on $\bR^D$.
\end{cor}

\begin{remark}\label{rem:groupaction}
Although the automorphisms $h_\tau$ will suffice for our study of the tropical cluster variety, the group action on $\cA$ generated by the induced automorphisms may be of interest. First, we describe the group of automorphisms of $S$ generated by the $g_\tau$. For each element $g$ of this group, the matrix $g(Z)$ is obtained from $Z$ by permuting the columns according to some $\tau\in W$ and then, for every $i\in[1,n]$,
\begin{itemize}
\item changing the sign of neither or both of the columns $\tau(i)$ and $\tau(\ol i)$ \textbf{if} $\tau(i)\in[1,n]$, or
\item changing the sign of exactly one of the columns $\tau( \ol i)$ and $\tau(i)$ \textbf{if} $\tau(i)\notin[1,n]$.
\end{itemize}
The latter $n$ binary choices are determined by a $\beta\in(\bZ/2)^n$ (with $\beta_i=0$ indicating the first choice in both cases), and we denote the element $g$ by $g_{\tau,\beta}$. In particular, $g_{\tau,(0,\dots,0)}=g_\tau$. The resulting group of $2^{2n}n!$ elements acts faithfully on $S$, however, $g_{\tau,\beta}$ is trivial on $\iota(S)$ if and only if $\beta$ is $(0,\dots,0)$ or $(1,\dots,1)$ and $\tau$ lies in the two-element center $Z(W)$. As a result, we obtain a faithful $\wt W$-action on $\cA$, where $\wt W$ has order $2^{2n-2}n!$ and fits into the exact sequence
\begin{equation}\label{eq:exactseq}
1\to(\bZ/2)^{n-1}\to\wt W\to W/Z(W)\to 1.    
\end{equation}
This exact sequence does not split already for $n=3$.
\end{remark}

\section{ASPTs and the main theorem}

This section largely contains a reinterpretation of the combinatorial constructions in~\cite{CoxMakh2025}, with two key new additions: Lemma~\ref{lem:keylemma} and the main Theorem~\ref{thm:main}.

\subsection{Axially symmetric phylogenetic trees}

A \textit{phylogenetic tree} is a pair $(T,\vv)$ consisting of a tree $T$ with $2n$ leaves and no vertices of degree 2, together with a bijection $\vv$ from $N$ to the set of leaf vertices of $T$. Here, the tree $T$ is abstract, meaning that we consider phylogenetic trees only up to label-preserving isomorphism.


Let $P$ denote a regular $2n$-gon. We say that two diagonals of $P$ cross if they share exactly one interior point. We use the term \textit{subdivision of $P$} to refer to a polygonal subdivision of $P$ formed by a set $\Theta$ of pairwise non-crossing diagonals. We identify such a subdivision with the set $\Theta$. A \textit{labeling} of $P$ is a bijection from $N$ to the set of sides of $P$. Given a subdivision $\Theta$ of $P$ and a labeling $\ph$ of $P$, we obtain a phylogenetic tree $\cT_{\Theta,\ph}=(T,\vv)$ as the dual graph. Explicitly:
\begin{itemize}
\item the non-leaf vertices of $T$ correspond to the cells (polygons) of $\Theta$, with two such vertices adjacent if and only if the respective cells share a side;
\item for $a\in N$, the leaf $\vv(a)$ is adjacent to the vertex corresponding to the cell that contains the side $\ph(a)$.
\end{itemize}
It is clear that any phylogenetic tree has the form $\cT_{\Theta,\ph}$ for an appropriate choice of $\Theta$ and $\ph$. However, for a given phylogenetic tree this choice is not unique, even modulo symmetries of $P$.\footnote{For example, if $(T,\vv)=\cT_{\Theta,\ph}$ and the leaves $\vv(a)$ and $\vv(b)$ have a common neighbor, then $(T,\vv)=\cT_{\Theta,\ph'}$, where $\ph'$ is obtained from $\ph$ by swapping the images of $a$ and $b$.}

Now, we fix one of the $n$ longest diagonals of $P$ and denote it by $\delta_0$. We say that a subdivision of $P$ is \textit{axially symmetric} if with every diagonal it also contains its mirror image with respect to $\delta_0$. Note that such a subdivision can only contain diagonals that do not cross $\delta_0$ or are perpendicular to $\delta_0$. A labeling of $P$ is \textit{axially symmetric} if the sides $\ph(i)$ and $\ph(\ol i)$ are symmetric to each other with respect to $\delta_0$ for every $i\in [1,n]$. 

\begin{defn}
An \textit{axially symmetric phylogenetic tree} (or \textit{ASPT}) is a phylogenetic tree that has the form $\cT_{\Theta,\ph}$ for an axially symmetric subdivision $\Theta$ and an axially symmetric labeling $\ph$.
\end{defn}

We see that every ASPT $(T,\vv)$ is equipped with an involution $\sigma$ of $T$ that exchanges vertices $\vv(i)$ and $\vv(\ol i)$. Since a tree automorphism is determined by the images of the leaves, $\sigma$ is unique. We will refer to $\sigma$ as the \textit{symmetry} of $(T,\vv)$.

\begin{example}\label{ex:ASPTn=3}
For $n=3$, every ASPT has one of the 7 forms shown in Figure~\ref{fig:ASPTn=3}, where the set $\{a,b,c\}\subset N$ satisfies $\{|a|,|b|,|c|\}=\{1,2,3\}$, and we label each vertex $v$ by the element $\vv^{-1}(v)\in N$.

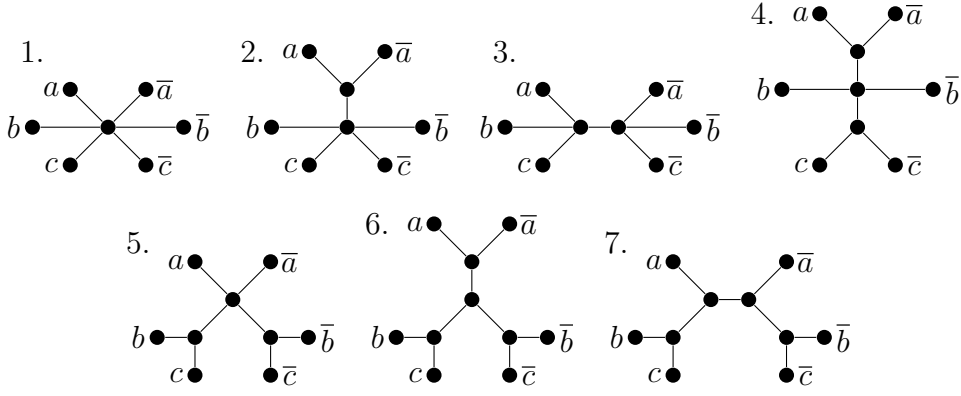
\begin{figure}[h!tbp]

\begin{tikzpicture}[x=5mm, y=5mm]
\node at (-2,2) {1.};

\node[circle, fill=black, inner sep=2] (p) at (0,0) {};

\node[circle, fill=black, inner sep=2] (a) at (-1,1) {};
\node[circle, fill=black, inner sep=2] (b) at (1,1) {};
\node[circle, fill=black, inner sep=2] (c) at (-2,0) {};
\node[circle, fill=black, inner sep=2] (d) at (2,0) {};
\node[circle, fill=black, inner sep=2] (e) at (-1,-1) {};
\node[circle, fill=black, inner sep=2] (f) at (1,-1) {};

\node at (-1.5,1) {$a$};
\node at (1.5,1) {$\ol a$};
\node at (-2.5,0) {$b$};
\node at (2.5,0) {$\ol b$};
\node at (-1.5,-1) {$c$};
\node at (1.5,-1) {$\ol c$};

\draw (p) -- (a)
      (p) -- (b)
      (p) -- (c)
      (p) -- (d)
      (p) -- (e)
      (p) -- (f);
\end{tikzpicture} 
\begin{tikzpicture}[x=5mm, y=5mm]
\node at (-2.5,2) {2.};

\node[circle, fill=black, inner sep=2] (p) at (0,1) {};
\node[circle, fill=black, inner sep=2] (q) at (0,0) {};

\node[circle, fill=black, inner sep=2] (a) at (-1,2) {};
\node[circle, fill=black, inner sep=2] (b) at (1,2) {};
\node[circle, fill=black, inner sep=2] (c) at (-2,0) {};
\node[circle, fill=black, inner sep=2] (d) at (2,0) {};
\node[circle, fill=black, inner sep=2] (e) at (-1,-1) {};
\node[circle, fill=black, inner sep=2] (f) at (1,-1) {};

\node at (-1.5,2) {$a$};
\node at (1.5,2) {$\ol a$};
\node at (-2.5,0) {$b$};
\node at (2.5,0) {$\ol b$};
\node at (-1.5,-1) {$c$};
\node at (1.5,-1) {$\ol c$};

\draw (p) -- (a)
      (p) -- (b)
      (q) -- (c)
      (q) -- (d)
      (q) -- (e)
      (q) -- (f)
      (p) -- (q);
\end{tikzpicture} 
\begin{tikzpicture}[x=5mm, y=5mm]
\node at (-2,2) {3.};

\node[circle, fill=black, inner sep=2] (p) at (0,0) {};
\node[circle, fill=black, inner sep=2] (q) at (1,0) {};

\node[circle, fill=black, inner sep=2] (a) at (-1,1) {};
\node[circle, fill=black, inner sep=2] (b) at (2,1) {};
\node[circle, fill=black, inner sep=2] (c) at (-2,0) {};
\node[circle, fill=black, inner sep=2] (d) at (3,0) {};
\node[circle, fill=black, inner sep=2] (e) at (-1,-1) {};
\node[circle, fill=black, inner sep=2] (f) at (2,-1) {};

\node at (-1.5,1) {$a$};
\node at (2.5,1) {$\ol a$};
\node at (-2.5,0) {$b$};
\node at (3.5,0) {$\ol b$};
\node at (-1.5,-1) {$c$};
\node at (2.5,-1) {$\ol c$};

\draw (p) -- (a)
      (q) -- (b)
      (p) -- (c)
      (q) -- (d)
      (p) -- (e)
      (q) -- (f)
      (p) -- (q);
\end{tikzpicture} 
\begin{tikzpicture}[x=5mm, y=5mm]
\node at (-2.5,2) {4.};

\node[circle, fill=black, inner sep=2] (p) at (0,1) {};
\node[circle, fill=black, inner sep=2] (q) at (0,0) {};
\node[circle, fill=black, inner sep=2] (r) at (0,-1) {};

\node[circle, fill=black, inner sep=2] (a) at (-1,2) {};
\node[circle, fill=black, inner sep=2] (b) at (1,2) {};
\node[circle, fill=black, inner sep=2] (c) at (-2,0) {};
\node[circle, fill=black, inner sep=2] (d) at (2,0) {};
\node[circle, fill=black, inner sep=2] (e) at (-1,-2) {};
\node[circle, fill=black, inner sep=2] (f) at (1,-2) {};

\node at (-1.5,2) {$a$};
\node at (1.5,2) {$\ol a$};
\node at (-2.5,0) {$b$};
\node at (2.5,0) {$\ol b$};
\node at (-1.5,-2) {$c$};
\node at (1.5,-2) {$\ol c$};

\draw (p) -- (a)
      (p) -- (b)
      (q) -- (c)
      (q) -- (d)
      (r) -- (e)
      (r) -- (f)
      (p) -- (q)
      (q) -- (r);
\end{tikzpicture}
\vspace{2mm}

\begin{tikzpicture}[x=5mm, y=5mm]
\node at (-2.5,1.5) {5.};

\node[circle, fill=black, inner sep=2] (p) at (0,0) {};
\node[circle, fill=black, inner sep=2] (q) at (-1,-1) {};
\node[circle, fill=black, inner sep=2] (r) at (1,-1) {};

\node[circle, fill=black, inner sep=2] (a) at (-1,1) {};
\node[circle, fill=black, inner sep=2] (b) at (1,1) {};
\node[circle, fill=black, inner sep=2] (c) at (-2,-1) {};
\node[circle, fill=black, inner sep=2] (d) at (2,-1) {};
\node[circle, fill=black, inner sep=2] (e) at (-1,-2) {};
\node[circle, fill=black, inner sep=2] (f) at (1,-2) {};

\node at (-1.5,1) {$a$};
\node at (1.5,1) {$\ol a$};
\node at (-2.5,-1) {$b$};
\node at (2.5,-1) {$\ol b$};
\node at (-1.5,-2) {$c$};
\node at (1.5,-2) {$\ol c$};

\draw (p) -- (a)
      (p) -- (b)
      (q) -- (c)
      (r) -- (d)
      (q) -- (e)
      (r) -- (f)
      (p) -- (q)
      (p) -- (r);
\end{tikzpicture}
\begin{tikzpicture}[x=5mm, y=5mm]
\node at (-2.5,2) {6.};

\node[circle, fill=black, inner sep=2] (p) at (0,1) {};
\node[circle, fill=black, inner sep=2] (q) at (0,0) {};
\node[circle, fill=black, inner sep=2] (r) at (-1,-1) {};
\node[circle, fill=black, inner sep=2] (s) at (1,-1) {};

\node[circle, fill=black, inner sep=2] (a) at (-1,2) {};
\node[circle, fill=black, inner sep=2] (b) at (1,2) {};
\node[circle, fill=black, inner sep=2] (c) at (-2,-1) {};
\node[circle, fill=black, inner sep=2] (d) at (2,-1) {};
\node[circle, fill=black, inner sep=2] (e) at (-1,-2) {};
\node[circle, fill=black, inner sep=2] (f) at (1,-2) {};

\node at (-1.5,2) {$a$};
\node at (1.5,2) {$\ol a$};
\node at (-2.5,-1) {$b$};
\node at (2.5,-1) {$\ol b$};
\node at (-1.5,-2) {$c$};
\node at (1.5,-2) {$\ol c$};

\draw (p) -- (a)
      (p) -- (b)
      (r) -- (c)
      (s) -- (d)
      (r) -- (e)
      (s) -- (f)
      (p) -- (q)
      (q) -- (r)
      (q) -- (s);
\end{tikzpicture} 
\begin{tikzpicture}[x=5mm, y=5mm]
\node at (-2.5,1.5) {7.};

\node[circle, fill=black, inner sep=2] (p) at (0,0) {};
\node[circle, fill=black, inner sep=2] (q) at (1,0) {};
\node[circle, fill=black, inner sep=2] (r) at (-1,-1) {};
\node[circle, fill=black, inner sep=2] (s) at (2,-1) {};

\node[circle, fill=black, inner sep=2] (a) at (-1,1) {};
\node[circle, fill=black, inner sep=2] (b) at (2,1) {};
\node[circle, fill=black, inner sep=2] (c) at (-2,-1) {};
\node[circle, fill=black, inner sep=2] (d) at (3,-1) {};
\node[circle, fill=black, inner sep=2] (e) at (-1,-2) {};
\node[circle, fill=black, inner sep=2] (f) at (2,-2) {};

\node at (-1.5,1) {$a$};
\node at (2.5,1) {$\ol a$};
\node at (-2.5,-1) {$b$};
\node at (3.5,-1) {$\ol b$};
\node at (-1.5,-2) {$c$};
\node at (2.5,-2) {$\ol c$};

\draw (p) -- (a)
      (q) -- (b)
      (r) -- (c)
      (s) -- (d)
      (r) -- (e)
      (s) -- (f)
      (p) -- (q)
      (p) -- (r)
      (q) -- (s);
\end{tikzpicture}

\caption{The 7 forms of ASPTs for $n=3$.}\label{fig:ASPTn=3}

\end{figure}
\end{example}

\begin{defn}
A \textit{weighted phylogenetic tree} $(T,\vv,l)$ is a phylogenetic tree $(T,\vv)$ together with a weight function $l$ from the edge set of $T$ to $\bR$ such that $l(e)>0$ for every non-leaf edge $e$ (but not necessarily for the leaf edges). We also say that $l$ is a \textit{weighting} of $(T,\vv)$.
\end{defn}

A weighted phylogenetic tree $(T,\vv,l)$ defines a ``distance'' function $d_{T,\vv,l}\colon N^2\to\bR$. For vertices $u$, $v$ of $T$, let $\Path(u,v)$ denote the set of edges lying in the non-self-intersecting path between $u$ and $v$. 
We set 
\[d_{T,\vv,l}(a,b)=\sum_{e\in\Path(\vv(a),\vv(b))}l(e).\] 
The function $d_{T,\vv,l}(a,b)$ is symmetric and satisfies $d_{T,\vv,l}(a,a)=0$. 

\begin{defn}
An \textit{axially symmetric weighted phylogenetic tree} (or \textit{ASWPT}) is a weighted phylogenetic tree $(T,\vv,l)$ such that $(T,\vv)$ is an ASPT and $l(\sigma(e))=l(e)$ for every edge $e$, where $\sigma$ is the symmetry of $(T,\vv)$.
\end{defn}

Evidently, an ASWPT $(T,\vv,l)$ has the property that $d_{T,\vv,l}(a,b)=d_{T,\vv,l}(\ol a,\ol b)$. 
We extend this property to a key criterion that will allow us to distinguish ASWPTs among weighted phylogenetic trees.

\begin{lemma}\label{lem:keylemma}
A weighted phylogenetic tree $(T,\vv,l)$ with the following two properties is an ASWPT.
\begin{enumerate}[label=(\roman*)]
\item $d_{T,\vv,l}(a,b)=d_{T,\vv,l}(\ol a,\ol b)$ for any $a,b\in N$.
\item For any $1\le i<j<k\le n$, the maximum occurs at least twice among the four quantities 
\begin{align*}
&d_{T,\vv,l}(i,j)+d_{T,\vv,l}(i,\ol j)+d_{T,\vv,l}(k,\ol k),\\
&d_{T,\vv,l}(i,k)+d_{T,\vv,l}(i,\ol k)+d_{T,\vv,l}(j,\ol j),\\
&d_{T,\vv,l}(j,k)+d_{T,\vv,l}(j,\ol k)+d_{T,\vv,l}(i,\ol i),\\
&d_{T,\vv,l}(i,\ol j)+d_{T,\vv,l}(i,k)+d_{T,\vv,l}(j,\ol k).
\end{align*}
\end{enumerate}
\end{lemma}

The proof of Lemma~\ref{lem:keylemma} is given separately in Section~\ref{sec:keylemmaproof}. This proof is somewhat technical, however, it is independent of the rest of the discussion.

\begin{remark}\label{rem:keylemma}
Although we state and apply Lemma~\ref{lem:keylemma} as a one-directional implication, it is not hard to check that it is indeed a criterion: properties (i) and (ii) are satisfied by ASWPTs. Furthermore, one may deduce another, possibly more transparent, criterion for axial symmetry. That criterion effectively says that a phylogenetic tree is axially symmetric if and only if its subtree spanned by $\vv(i)$, $\vv(j)$, $\vv(k)$, $\vv(\ol i)$, $\vv(\ol j)$, $\vv(\ol k)$ is axially symmetric for any $i,j,k\in[1,n]$; a precise statement is given by Corollary~\ref{cor:altcriterion}. Unfortunately, this alternative form does not appear much easier to prove.
\end{remark}

\begin{example}
While the meaning and necessity of property (i) is clear, property (ii) is less straightforward. Its role is somewhat clarified by the following example. Consider the phylogenetic tree $(T,\vv)$ shown in Figure~\ref{fig:nonASPT} and a weighting $l$ of $(T,\vv)$ such that $l(e_i)=l(e_{\ol i})$ for all $i\in[1,3]$, where $e_a$ is the edge incident to $\vv(a)$. It is clear that $(T,\vv)$ is not an ASPT and, therefore, $(T,\vv,l)$ is not an ASWPT. Nonetheless, $(T,\vv,l)$ satisfies property (i) and it is property (ii) that is violated.
\begin{figure}[h!tbp]
\begin{tikzpicture}[x=5mm, y=5mm]
\node[circle, fill=black, inner sep=2] (p) at (0,1) {};
\node[circle, fill=black, inner sep=2] (q) at (0,0) {};
\node[circle, fill=black, inner sep=2] (r) at (-1,-1) {};
\node[circle, fill=black, inner sep=2] (s) at (1,-1) {};

\node[circle, fill=black, inner sep=2] (a) at (-1,2) {};
\node[circle, fill=black, inner sep=2] (b) at (1,2) {};
\node[circle, fill=black, inner sep=2] (c) at (-2,-1) {};
\node[circle, fill=black, inner sep=2] (d) at (2,-1) {};
\node[circle, fill=black, inner sep=2] (e) at (-1,-2) {};
\node[circle, fill=black, inner sep=2] (f) at (1,-2) {};

\node at (-1.5,2) {$1$};
\node at (1.5,2) {$\ol 1$};
\node at (-2.5,-1) {$2$};
\node at (2.5,-1) {$3$};
\node at (-1.5,-2) {$\ol 2$};
\node at (1.5,-2) {$\ol 3$};

\draw (p) -- (a)
      (p) -- (b)
      (r) -- (c)
      (s) -- (d)
      (r) -- (e)
      (s) -- (f)
      (p) -- (q)
      (q) -- (r)
      (q) -- (s);
\end{tikzpicture} 
\caption{}\label{fig:nonASPT}
\end{figure}
\end{example}

\subsection{The space of ASPTs}

All subdivisions and labelings of $P$ considered in this subsection are assumed to be axially symmetric.

We associate a point $w(T,\vv,l)\in\bR^D$ with every ASWPT $(T,\vv,l)$ by setting 
\[w(T,\vv,l)_{a,b}=d_{T,\vv,l}(a,b).\] 
For an ASPT $(T,\vv)$ we consider the set
\[C_{T,\vv}=\{w(T,\vv,l)\,|\,l\text{ is a weighting of }(T,\vv)\}\subset\bR^D.\]

Before listing the key properties of the sets $C_{T,\vv}$, we discuss contractions. Consider an ASPT $(T,\vv)$ with symmetry $\sigma$ and a non-leaf edge $e$ of $T$. We may contract the (one- or two-element) set of edges $\{e,\sigma(e)\}$ in $T$, obtaining a tree $T/\{e,\sigma(e)\}$. The leaves of $T/\{e,\sigma(e)\}$ are in natural bijection with the leaves of $T$, allowing us to consider the phylogenetic tree $(T/\{e,\sigma(e)\},\vv)$. We say that $(T/\{e,\sigma(e)\},\vv)$ is obtained from $(T,\vv)$ by the \textit{symmetric contraction} of $e$.

Note that, in terms of subdivisions, a symmetric contraction is either the deletion of one axially symmetric diagonal or of two diagonals that are axially symmetric to each other. Explicitly, an ASPT $(T,\vv)$ is obtained from $\cT_{\Theta,\ph}$ by a symmetric contraction if and only if $(T,\vv)=\cT_{\Theta',\ph}$, where $\Theta'=\Theta\bs\{\delta_1,\delta_2\}$ and $\delta_1$ is the reflection of $\delta_2$ across $\delta_0$. Consequently, ASPTs that cannot be obtained from any other ASPT by a symmetric contraction are precisely those $\cT_{\Theta,\ph}$ for which $\Theta$ is maximal by inclusion among axially symmetric subdivisions. We will refer to such ASPTs as \textit{maximal}. 

If $\Theta$ is maximal and $\delta_0\in\Theta$, then, in view of maximality, each of the halves of $P$ on either side of $\delta_0$ is triangulated by the diagonals in $\Theta$. If $\delta_0\notin\Theta$, then every interior point $p$ of $\delta_0$ either lies in an element of $\Theta$ that is perpendicular to $\delta_0$, or is not contained in any element of $\Theta$. In the latter case consider the polygonal cell of the subdivision containing $p$. This cell is axially symmetric and, by the maximality of $\Theta$, must either be a trapezoid with bases perpendicular to $\delta_0$, or an isosceles triangle with base perpendicular to $\delta_0$. All cells that are not obtained in this way will be triangles, each contained in one of the two halves of $P$ formed by $\delta_0$. We obtain the following classification (see also Figure~\ref{fig:subdivs}):
\begin{prop}\label{prop:ASPTtypes}
Maximal axially symmetric subdivisions are of two types:
\begin{enumerate}[label=(\Roman*)]
\item $\delta_0\in\Theta$ and $\Theta$ is a triangulation,
\item $\delta_0\notin\Theta$, in which case $\Theta$ contains $m\in[1,n-1]$ diagonals that are perpendicular to $\delta_0$ and $m-1$ of its cells are trapezoids that are axially symmetric with respect to $\delta_0$, the remaining cells being triangles.
\end{enumerate}    
\end{prop}

\begin{figure}[h!tbp]
\centering
\begin{tikzpicture}[scale = 1.3]
\foreach \i in {1,...,8}{
\coordinate (P\i) at ({cos(90-45*(\i-1))},{sin(90-45*(\i-1))});
}

\draw (P1)--(P2)--(P3)--(P4)--(P5)--(P6)--(P7)--(P8)--cycle;
\draw (P1)--(P5) (P1)--(P4) (P1)--(P6) (P4)--(P2) (P6)--(P8);
\end{tikzpicture}
\hspace{15mm}
\begin{tikzpicture}[scale=1.3]
\foreach \i in {1,...,8}{
  \coordinate (P\i) at ({cos(90-45*(\i-1))},{sin(90-45*(\i-1))});
}

\draw (P1)--(P2)--(P3)--(P4)--(P5)--(P6)--(P7)--(P8)--cycle;
\draw (P1)--(P3) (P1)--(P7) (P3)--(P7) (P4)--(P6);
\end{tikzpicture}

\caption{Two maximal axially symmetric subdivisions for $n=4$: of type (I) on the left and of type (II) with $m=2$ on the right.}
\label{fig:subdivs}
\end{figure}
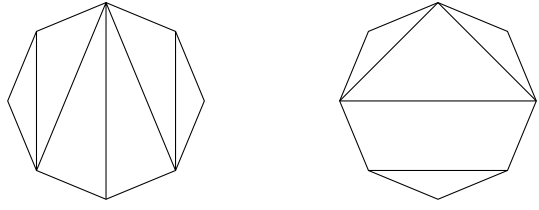

We will also say that a maximal ASPT $\cT_{\Theta,\ph}$ is of type (I) if $\Theta$ has type (I), and of type (II) if $\Theta$ has type (II). A maximal ASPT cannot be of both types: it is of type (II) if and only if it has a vertex that is fixed by the symmetry.

\begin{example}
Recall the setting of Example~\ref{ex:ASPTn=3}. ASPTs of the forms 4, 6 and 7 are maximal, while 1, 2, 3 and 5 can be obtained from others by symmetric contraction. Also, maximal ASPTs of form 7 have type (I), while maximal ASPTs of the forms 4 and 6 have type (II).
\end{example}

Next, for an ASPT $(T,\vv)$ with symmetry $\sigma$, we let $k(T,\vv)$ denote the number of $\sigma$-orbits in the edge set of $T$. 
Equivalently, if $(T,\vv)=\cT_{\Theta,\ph}$, then $k(T,\vv)-n$ is equal to the number of orbits with respect to axial symmetry across $\delta_0$ in the set of diagonals $\Theta$.

\begin{prop}\label{prop:maxk}
An ASPT $(T,\vv)$ is maximal if and only if $k(T,\vv)=2n-1$.    
\end{prop}
\begin{proof}
Suppose $(T,\vv)=\cT_{\Theta,\ph}$ is maximal. Then, $\Theta$ is a maximal axially symmetric subdivision. If $\Theta$ is of type (I), then $|\Theta|=2n-3$ and $k(T,\vv)=n+(2n-4)/2+1=2n-1$.

Let $\Theta$ be of type (II) and contain $m$ diagonals that are perpendicular to $\delta_0$. Since $\Theta$ has $m-1$ quadrilateral cells, it consists of $2n-3-(m-1)=2n-m-2$ diagonals, which form $k(T,\vv)=n+m+(2n-m-2-m)/2=2n-1$ axial-symmetry orbits.

Conversely, a non-maximal ASPT $(T,\vv)=\cT_{\Theta,\ph}$ must have $k(T,\vv)<n-1$ because there exists a maximal axially symmetric subdivision that contains $\Theta$.
\end{proof}

We now establish properties of the sets $C_{T,\vv}$ that allow us to define the space of ASPTs.

\begin{prop}\label{prop:spaceofASPTs}
\hfill
\begin{enumerate}[label=(\alph*)]
\item If ASWPTs $(T_1,\vv_1,l_1)$ and $(T_2,\vv_2,l_2)$ are distinct, then $w(T_1,\vv_1,l_1)\neq w(T_2,\vv_2,l_2)$. If ASPTs $(T_1,\vv_1)$ and $(T_2,\vv_2)$ are distinct, then $C_{T_1,\vv_1}$ and $C_{T_2,\vv_2}$ are disjoint.
\item For an ASPT $(T,\vv)$, the set $C_{T,\vv}$ is a relatively open polyhedral cone of dimension $k(T,\vv)$: the product of a simplicial cone of dimension $k(T,\vv)-n$ and an $n$-dimensional affine space.
\item The relatively open faces of the cone $C_{T,\vv}$ are precisely the cones of the form $C_{T',\vv'}$ such that $(T',\vv')$ is obtained from $(T,\vv)$ by a sequence of symmetric contractions.
\item The collection of cones $C_{T,\vv}$ with $(T,\vv)$ ranging over all ASPTs is a polyhedral fan in $\bR^D$ of pure dimension $2n-1$ with a lineality space of dimension $n$.
\end{enumerate}    
\end{prop}
\begin{proof}
For the first claim in part (a), let $l_i^C$, $i\in\{1,2\}$ denote the weighting obtained from $l_i$ by adding $C\in\bR$ to the weights of all leaf edges. Note that 
\[d_{T_i,\vv_i,l_i^C}(a,b)=d_{T_i,\vv_i,l_i}(a,b)+2C\] 
for distinct $a,b\in N$. Therefore, $w(T_1,\vv_1,l_1)\neq w(T_2,\vv_2,l_2)$ if and only if $w(T_1,\vv_1,l_1^C)\neq w(T_2,\vv_2,l_2^C)$ for some $C\in\bR$. For large enough $C$, both ASWPTs $(T_i,\vv_i,l_i^C)$ will be \textit{metric}, meaning that all of the weights will be positive. However, it is easily checked by induction that a metric tree with no vertices of degree 2 is determined by the pairwise distances between its leaves, see also \cite{Bun1974}. The second claim in (a) follows from the first.

For part (b), consider an edge $e$ of $T$ and let $l_e$ denote the function on the edge set of $T$ such that $l_e(e)=l_e(\sigma(e))=1$ (where $\sigma$ is the symmetry of $(T,\vv)$) and $l_e(e')=0$ for all other edges $e'$. Although $(T,\vv,l_e)$ is, in general, not a weighted phylogenetic tree, we may still define the function $d_{T,\vv,l_e}$ and the point $w(T,\vv,l_e)\in\bR^D$ as above. Let $E$ be a subset of the edge set containing exactly one representative of each $\sigma$-orbit. The weightings of $(T,\vv)$ that define ASWPTs are precisely linear combinations $\sum_{e\in E} c_el_e$ such that $c_e>0$ for all non-leaf edges $e\in E$. Consequently, $C_{T,\vv}$ consists of linear combinations $\sum_{e\in E} c_ew(T,\vv,l_e)$ with $c_e>0$ for all non-leaf edges $e\in E$. Furthermore, the vectors $w(T,\vv,l_e)$, $e\in E$ are linearly independent, otherwise we would have two distinct strictly positive weightings of $(T,\vv)$ that provide the same distance function. Part (b) follows. 

Suppose that for a subset $E'\subset E$ consisting of non-leaf edges, $(T',\vv')$ is obtained from $(T,\vv)$ by contracting the $\sigma$-invariant set of edges $\{e,\sigma(e)\}_{e\in E'}$. Then $C_{T',\vv'}$ consists precisely of linear combinations $\sum_{e\in E\bs E'} c_ew(T,\vv,l_e)$ such that $c_e>0$ for non-leaf edges $e$. Thus, $C_{T',\vv'}$ is a face of $C_{T,\vv}$ and, since $C_{T,\vv}$ has $2^{k(T,\vv)-n}$ faces, we obtain part (c).

Part (d) follows from the previous parts and Proposition~\ref{prop:maxk}.
\end{proof}

We refer to the polyhedral fan formed by the cones $C_{T,\vv}$ with $(T,\vv)$ ranging over all ASPTs as the \textit{space of ASPTs}. We can now state our first main result.
\begin{theorem}\label{thm:main}
The type C tropical cluster variety $\Trop X$ is the space of ASPTs. In particular, its support $|\Trop X|$ is the set of points $w(T,\vv,l)$ with $(T,\vv,l)$ ranging over all ASWPTs.
\end{theorem}

The theorem will be proved in Subsection~\ref{sec:mainproof}.

\begin{example}
For $n=3$ the space of ASPTs has dimension $5$ and a lineality space of dimension 3. Thus, topologically, it is the product of $\bR^3$ and a cone over a 1-dimensional simplicial complex. This simplicial complex is depicted in~\cite[Figure A]{CoxMakh2025}. It has 13 vertices (corresponding to 4-dimensional cones in the fan, enumerated by ASPTs of the forms 2, 3 and 5 in Example~\ref{ex:ASPTn=3}) and 21 edges (corresponding to maximal cones, enumerated by ASPTs of the forms 4, 6 and 7).
\end{example}

\begin{remark}
After this paper was written, the reference~\cite{CaLeYu2025} was brought to the author's attention. It studies a fan whose cones are indexed by a similar, but apparently smaller, family of trees. Moreover, small cases suggest that the fan in~\cite{CaLeYu2025} may be combinatorially equivalent to a subfan of the space of ASPTs. It should be very interesting to discern the exact relationship between the two constructions. 
\end{remark}

\section{Proof of the main theorem and toric degenerations}

Throughout this section we assume that all considered subdivisions and labelings of $P$ are axially symmetric, unless stated otherwise.

\subsection{Monomial-freeness and the containing toric ideal}

Theorem~\ref{thm:main} can be split into several claims as follows. 
\begin{enumerate}[label = (C\arabic*)]
\item For every ASWPT $(T,\vv,l)$, the initial ideal $\init_{w(T,\vv,l)} I$ is monomial-free.
\item For ASWPTs $(T_1,\vv_1,l_1)$ and $(T_2,\vv_2,l_2)$, the initial ideals $\init_{w(T_1,\vv_1,l_1)} I$ and $\init_{w(T_2,\vv_2,l_2)} I$ coincide if and only if $(T_1,\vv_1)=(T_2,\vv_2)$.
\item If $w\in\bR^D$ is not equal to $w(T,\vv,l)$ for any ASWPT $(T,\vv,l)$, then $\init_w I$ contains a monomial.
\end{enumerate}

We start with (C1), i.e., monomial-freeness. The difficulty here is that, unlike the type A case considered in~\cite{SpeStu2004}, the initial ideals corresponding to maximal cones of $\Trop I$ are not necessarily toric or even prime. 

We make two general observations that allow us to focus on ASPTs of a specific form, significantly simplifying our arguments. First, by Proposition~\ref{prop:spaceofASPTs}(c), the set of points $w(T,\vv,l)$ with $(T,\vv)$ maximal is dense in the set of all $w(T,\vv,l)$. Hence, it suffices to show that $\init_{w(T,\vv,l)} I$ is monomial-free when $(T,\vv)$ is a maximal ASPT.

Next, observe that the group $W$ acts on the set of ASPTs by $\tau\colon(T,\vv)\mapsto(T,\vv\circ\tau^{-1})$, and similarly for ASWPTs. For an ASWPT $(T,\vv,l)$ and $\tau\in W$, we have 
\begin{equation}\label{eq:taurespectsw}
\tau(w(T,\vv,l)) = w(T,\vv\circ\tau^{-1},l),
\end{equation}
i.e., $w$ is $W$-equivariant as a function from the set of ASWPTs to $\bR^D$. 
Furthermore, $W$ acts on the set of axially symmetric labelings of $P$ by $\tau\colon\ph\to\ph\circ\tau^{-1}$. For any $\tau\in W$, subdivision $\Theta$ and labeling $\ph$ of $P$, one has
\begin{equation}\label{eq:actonlabeling}
\tau(\cT_{\Theta,\ph})=\cT_{\Theta,\ph\circ\tau^{-1}}.    
\end{equation}

Let $(T,\vv)=\cT_{\Theta,\ph}$ be a maximal ASPT . Consider a labeling $\ph_0$ of $P$ such that the sides of $P$ on one side of $\delta_0$ are labeled by the elements $1,\dots,n$ in counterclockwise order, while those on the other side of $\delta_0$ are labeled by $\ol 1,\dots,\ol n$ in clockwise order. In view of~\eqref{eq:actonlabeling}, for any ASPT $(T,\vv)=\cT_{\Theta,\ph}$, we may find $\tau\in W$ such that $\tau((T,\vv))=\cT_{\Theta,\ph_0}$. By Corollary~\ref{cor:fansWinv} and \eqref{eq:taurespectsw}, $w(T,\vv,l)$ lies in $|\Trop I|$ if and only if $w(T,\vv\circ\tau^{-1},l)$ does. 

We have reduced the proof of monomial-freeness to the case of an ASWPT $(T,\vv,l)$ such that $(T,\vv)$ is maximal and has the form $\cT_{\Theta,\ph_0}$. Our approach is to realize $\init_{w(T,\vv,l)} I$ as a subideal of a toric ideal, i.e., the kernel of a monomial map.

We establish two properties of ASPTs of this form by analyzing their combinatorial structure. Recall from Proposition~\ref{prop:ASPTtypes} that $\Theta$ and $(T,\vv)$ have either type (I) or type (II). Let $\sigma$ be the symmetry of $(T,\vv)$. We also let $\vpath(u,v)$ denote the set of vertices lying on the path between vertices $u$ and $v$ of $T$, including $u$ and $v$.

\begin{prop}\label{prop:maxASPTprops}
Let $(T,\vv)$ be a maximal ASPT of the form $\cT_{\Theta,\ph_0}$.
\hfill\begin{enumerate}[label = (\alph*)]
\item If $\vpath(\vv(a),\vv(b))$ contains exactly one $\sigma$-fixed vertex, then exactly one of $a$ and $b$ lies in $[1,n]$.
\item For any $i_1,i_2,j_1,j_2\in[1,n]$, all vertices in $\vpath(\vv(i_1),\vv(i_2))\cap\vpath(\vv(\ol{j_1}),\vv(\ol{j_2}))$ are $\sigma$-fixed.
\end{enumerate}    
\end{prop}
\begin{proof}
If $(T,\vv)$ has type (I), then $T$ has a distinguished edge $e_0$ corresponding to the diagonal $\delta_0$. The deletion of $e_0$ from $T$ produces two components which are exchanged by $\sigma$. In view of the realization $(T,\vv)=\cT_{\Theta,\ph_0}$, all leaves $\vv(i)$ with $i\in[1,n]$ lie in one of these component, while all $\vv(\ol i)$ with $i\in[1,n]$ lie in the other. Hence, both (a) and  (b) are trivial in this case.

Suppose that $(T,\vv)$ has type (II) and $\Theta$ contains $m$ diagonals perpendicular to $\delta_0$. We denote the $\sigma$-fixed vertices of $T$ by $v_1,\dots,v_{m+1}$ with $v_k$ adjacent to $v_{k+1}$. These vertices correspond to cells of $\Theta$ adjacent to diagonals perpendicular to $\delta_0$. The deletion of all $\sigma$-fixed vertices produces $2m+2$ components, with each $v_k$ adjacent to two of these components: $T_k$ and $\sigma(T_k)$. In view of the realization $(T,\vv)=\cT_{\Theta,\ph_0}$, we may assume that every $T_k$ only contains leaves of the form $\vv(i)$, $i\in[1,n]$.  
We see that $\vpath(\vv(a),\vv(b))$ contains exactly one $\sigma$-fixed vertex if and only if $\vv(a)\in T_k$ and $\vv(b)\in\sigma(T_k)$ for some $k$ (or vice versa). This proves part (a). 

Now, in the setting of part (b), let $\vv(i_1)\in T_{k_1}$ and $\vv(i_2)\in T_{k_2}$. If $k_1=k_2$, then $\vpath(\vv(i_1),\vv(i_2))$ is contained within $T_{k_1}$. Otherwise, $\vpath(\vv(i_1),\vv(i_2))$ starts with several vertices of $T_{k_1}$, then passes through at least two $\sigma$-fixed vertices and ends with several vertices of $T_{k_2}$. A similar description applies to $\vpath(\vv(\ol j_1),\vv(\ol j_2))$, but, instead, featuring subtrees of the form $\sigma(T_k)$. Part (b) follows.
\end{proof}


Consider the polynomial ring $Q_{T,\vv}$ in $k(T,\vv)=2n-1$ variables, one for every $\sigma$-orbit in the edge set of $T$. For an edge $e$ we denote the variable corresponding to the orbit of $e$ by $t_e$, so that $t_e=t_{\sigma(e)}$. We define a homomorphism $\psi_{T,\vv}\colon R\to Q_{T,\vv}$ by setting
\begin{equation}\label{eq:xidef}
\psi_{T,\vv}(x_{a,b})=c_{T,\vv}(a,b)\prod_{e\in\Path(\vv(a),\vv(b))} t_e,
\end{equation}
where $c_{T,\vv}(a,b)=2$ if $\vpath(\vv(a),\vv(b))$ contains exactly one $\sigma$-fixed vertex, $c_{T,\vv}(a,b)=\mathrm i$ (the imaginary unit) if both $a,b\in[1,n]$ or both $a,b\notin[1,n]$, and $c_{T,\vv}(a,b)=1$ in all other cases. The coefficients $c_{T,\vv}(a,b)$ are well-defined in view of Proposition~\ref{prop:maxASPTprops}(a).

\begin{lemma}\label{lem:initinker}
Let $(T,\vv)$ be a maximal ASPT of the form $\cT_{\Theta,\ph_0}$. If $w\in C_{T,\vv}$, then  $\init_w I$ is contained in $\ker\psi_{T,\vv}$ and, in particular, $\init_w I$ is monomial-free.   
\end{lemma}

As with Lemma~\ref{lem:keylemma}, the proof of Lemma~\ref{lem:initinker} is given separately in Section~\ref{sec:initinkerproof} due to its length and technicality. There is, however, one useful corollary of the proof that we state now (and prove in Section~\ref{sec:initinkerproof}).
\begin{cor}\label{cor:gbasis}
For $(T,\vv)$ as in Lemma~\ref{lem:initinker}, the expressions $r_{a,b,c,d}$ form a Gr\"obner basis of $I$ with respect to any $w\in C_{T,\vv}$.
\end{cor}

\subsection{The proof and bases of $I$}\label{sec:mainproof}

\begin{proof}[Proof of Theorem~\ref{thm:main}]
We have already seen that $w(T,\vv,l)\in|\Trop I|$ for any ASWPT $(T,\vv,l)$, verifying claim (C1).

Also, the statement of Corollary~\ref{cor:gbasis} is easily extended to an arbitrary ASPT $(T,\vv,l)$. Indeed, the set of weights $w$ with respect to which the expressions $r_{a,b,c,d}$ form a Gr\"obner basis is closed (e.g., by lower semicontinuity of graded dimension of $\al \init_w r_{a,b,c,d}\ar_{a\prec b\prec c\prec d}$). Hence, it suffices to to consider the case of maximal $(T,\vv)$. In this case we may find $\tau\in W$ such that $(T,\vv\circ\tau^{-1},l)$ is of the form considered in Corollary~\ref{cor:gbasis}. However, 
\begin{equation*}\label{eq:htauofr}
h_\tau(r_{a,b,c,d})=\pm r_{\tau(a),\tau(b),\tau(c),\tau(d)}  
\end{equation*} 
since both sides contain the same three monomials and we can obtain a quadratic monomial as their linear combination contradicting the primeness of $I$, unless the two sides differ by a scalar factor. Therefore, 
\begin{equation}\label{eq:tauinitform}
h_\tau(\init_{w(T,\vv,l)} r_{a,b,c,d})=\pm\init_{w(T,\vv\circ\tau^{-1},l)} r_{\tau(a),\tau(b),\tau(c),\tau(d)}.    
\end{equation}
The expressions in the right-hand side of~\eqref{eq:tauinitform} generate $\init_{w(T,\vv\circ\tau^{-1},l)} I$, which coincides with $h_\tau(\init_{w(T,\vv,l)} I)$ by Proposition~\ref{prop:htauinI}. Hence, the forms $\init_{w(T,\vv,l)} r_{a,b,c,d}$ generate $\init_{w(T,\vv,l)} I$. 

We proceed to verify (C2).
First we check that $\init_{w(T,\vv,l_1)} I=\init_{w(T,\vv,l_2)} I$ for ASWPTs $(T,\vv,l_1)$ and $(T,\vv,l_2)$. Indeed, for $a\prec b\prec c\prec d$, let $T_{a,b,c,d}$ denote the minimal subtree of $T$ containing the leaves $\vv(a)$, $\vv(b)$, $\vv(c)$, $\vv(d)$. One sees that the initial form $\init_{w(T,\vv,l)}r_{a,b,c,d}$ is determined by $T_{a,b,c,d}$ (for details, see proof of Lemma~\ref{lem:initforms}). Therefore, this initial form is independent of the weighting $l$. Together with the above Gr\"obner basis statement, this shows that $\init_{w(T,\vv,l)} I$ is also independent of the weighting. 

Conversely, we check that if $(T_1,\vv_1,l_1)$ and $(T_2,\vv_2,l_2)$ are ASWPTs with $(T_1,\vv_1)\neq(T_2,\vv_2)$, then $\init_{w(T_1,\vv_1,l_1)} I\neq\init_{w(T_2,\vv_2,l_2)} I$. It suffices to consider the case when $w(T_2,\vv_2,l_2)$ lies in the interior of a facet of $C_{T_1,\vv_1}$. Indeed, since the initial ideal is constant within each $C_{T,\vv}$, these cones form a refinement of the tropical fan. If this refinement were proper, there would be two cones that define the same initial ideal and share a facet, which would then also correspond to the same initial ideal.

Thus, $(T_2,\vv_2)$ is obtained from $(T_1,\vv_1)$ by a single symmetric contraction. Suppose this is the symmetric contraction by the non-leaf edge $e$ between vertices $v_1$ and $v_2$. Since $v_1$ and $v_2$ have degree at least 3, we may find leaves $\vv_1(a)$, $\vv_1(b)$ such that $\vpath(\vv_1(a),\vv_1(b))$ contains $v_1$ but not $v_2$, and leaves $\vv_1(c)$, $\vv_1(d)$ such that $\vpath(\vv_1(c),\vv_1(d))$ contains $v_2$ but not $v_1$. Then, 
\[\init_{w(T_1,\vv_1,l_1)} r_{a,b,c,d}=\pm x_{a,c}x_{b,d}\pm x_{a,d}x_{b,c},\] 
while $\init_{w(T_2,\vv_2,l_2)} r_{a,b,c,d}=r_{a,b,c,d}$. A monomial-free ideal cannot contain both of these initial forms and we conclude that $\init_{w(T_1,\vv_2,l_1)} I\neq\init_{w(T_2,\vv_2,l_2)} I$.

It remains to verify (C3). For this last step, we use the following notation. For $w\in\bR^D$ let $\wh w\colon N^2\to\bR$ be defined by $\wh w(a,b)=w_{a,b}$ for $a\neq b$, while all $\wh w(a,a)=0$. 

We check that if $w\in|\Trop I|$, then $w$ necessarily has the form $w(T,\vv,l)$ for an ASWPT $(T,\vv,l)$. We may assume that $\wh w$ is a metric on $N$. Indeed, both $\Trop I$ and the space of ASPTs are invariant under translations by the all-ones vectors $(1,\dots,1)$. However, if we denote $w_C=w+C(1,\dots,1)$ for $C\in\bR$, then $\wh{w_C}(a,b)=\wh w(a,b)+C$ for all $a\neq b$, hence $\wh{w_C}$ is a metric for large enough $C$.

Since every $\init_w r_{a,b,c,d}$ is not a monomial, the function $\wh w$ satisfies the \textit{four-point condition}: among the three quantities $\wh w(a,b)+\wh w(c,d)$, $\wh w(a,c)+\wh w(b,d)$ and $\wh w(a,d)+\wh w(b,c)$ at least two are maximal. Since $\wh w$ is also a metric, by the results of~\cite{Bun1974} there exists a metric (i.e., positively-weighted) phylogenetic tree $(T,\vv,l)$ such that $d_{T,\vv,l}(a,b)=\wh w(a,b)$ for all $a,b\in N$. We check that $(T,\vv,l)$ is an ASWPT by verifying that the values $d_{T,\vv,l}(a,b)=\wh w(a,b)$ satisfy the two properties in Lemma~\ref{lem:keylemma}. Property (i) is immediate from our definition of $\wh w$. Furthermore, for any $1\le i<j<k\le n$ we have
\begin{multline}\label{eq:sijk}
s_{i,j,k}=x_{i,j}x_{i,\ol j}x_{k,\ol k}+x_{i,k}x_{i,\ol k}x_{j,\ol j}+x_{j,k}x_{j,\ol k}x_{i,\ol i}-2x_{i,k}x_{i,\ol j}x_{j,\ol k}\\
= x_{j,\ol k}r_{i,j,k,\ol i}-x_{i,\ol k}r_{i,j,k,\ol j}+x_{i,\ol j}r_{i,j,k,\ol 
k} \in I
\end{multline}
and, therefore, $\init_w s_{i,j,k}$ is not a monomial. However, the gradings of the four terms in $s_{i,j,k}$ with respect to $w$ are equal to the four quantities in property (ii), which means that this property holds for $(T,\vv,l)$. Hence, $(T,\vv,l)$ is an ASWPT and $w=w(T,\vv,l)$.
\end{proof}

We have two immediate consequences of this proof, the second of which emphasizes the importance of the expressions $s_{i,j,k}$ defined by~\eqref{eq:sijk}.

\begin{cor}\label{cor:grobnerbasis}
The expressions $r_{a,b,c,d}$ with $a\prec b\prec c\prec d$ form a Gr\"obner basis of $I$ with respect to any $w\in|\Trop I|$.    
\end{cor}

\begin{cor}\label{cor:tropicalbasis}
The expressions $r_{a,b,c,d}$ with $a\prec b\prec c\prec d$, together with the expressions $s_{i,j,k}$ for all $1\le i<j<k\le n$, form a tropical basis of $I$.
\end{cor}

\begin{remark}
The inclusion of the expressions $s_{i,j,k}$ into the tropical basis is necessary, in contrast to the type A case where the quadratic generators already form a tropical basis, see~\cite{SpeStu2004}. Specifically, if we consider the phylogenetic tree $(T,\vv,l)$ in Figure~\ref{fig:nonASPT} and define $w\in\bR^D$ by $w_{a,b}=d_{T,\vv,l}(a,b)$, then $w\notin|\Trop I|$ despite all $\init_w r_{a,b,c,d}$ being binomials.
\end{remark}

\subsection{Toric degenerations}

We have seen that the initial ideal corresponding to a maximal cone of $\Trop I$ is contained in a toric ideal. We extend this observation to a classification of the initial toric ideals of $I$, i.e., toric degenerations of $X$. Our result is as follows.

\begin{theorem}\label{thm:toricdegens}
For an ASPT $(T,\vv)$ and $w\in C_{T,\vv}$, the ideal $\init_w I$ is toric if and only if $(T,\vv)$ is maximal of type I.
\end{theorem}

Let $\Theta_0$ be the subdivision consisting of all diagonals that contain the endpoint of $\delta_0$ that lies between the sides $\ph_0(1)$ and $\ph_0(\ol 1)$. Set $(T_0,\vv_0)=\cT_{\Theta_0,\ph_0}$, it is the ASPT shown in Figure~\ref{fig:T0v0}. Recall the map $\pi\colon R\to U$ and the toric ideal $I_\tor=\ker\pi$ considered in Subsection~\ref{sec:monideal}. We define a homomorphism $\alpha\colon Q_{T_0,\vv_0}\to U$. Denote the first $n$ horizontal edges in Figure~\ref{fig:T0v0} by $e_1, f_2, \dots, f_{n-1}, f_n$ from left to right, and the first $n-1$ vertical edges by $e_2,\dots,e_n$. Any remaining edge is symmetric to one of these $2n-1$, hence $Q_{T_0,\vv_0}$ is generated by the $t_{e_i}$ and $t_{f_i}$. We set
\[
\alpha(t_{e_i})=
\begin{cases}
\sqrt{-\mathrm i}\,v_1u_1 &\text{if } i=1,\\
\sqrt{-\mathrm i}\,v_i &\text{if } i>1,   
\end{cases}\quad
\alpha(t_{f_i})=
\begin{cases}
u_i &\text{if } i<n,\\
\mathrm iu_n^2 &\text{if } i=n,
\end{cases}
\]
where $\sqrt{-\mathrm i}=\frac{1-\mathrm i}{\sqrt2}$.

\begin{figure}[h!tbp]

\begin{tikzpicture}[x=15mm, y=10mm]
\node[circle, fill=black, inner sep=2] (a) at (-1,0) {};
\node[circle, fill=black, inner sep=2] (b) at (0,0) {};
\node[circle, fill=black, inner sep=2] (v2) at (0,-1) {};
\node[circle, fill=black, inner sep=2] (c) at (2.5,0) {};
\node[circle, fill=black, inner sep=2] (vn) at (2.5,-1) {};
\node[circle, fill=black, inner sep=2] (d) at (3.5,0) {};
\node[circle, fill=black, inner sep=2] (voln) at (3.5,-1) {};
\node[circle, fill=black, inner sep=2] (e) at (6,0) {};
\node[circle, fill=black, inner sep=2] (vol2) at (6,-1) {};
\node[circle, fill=black, inner sep=2] (f) at (7,0) {};

\node at (-1.4,0) {$\vv_0(1)$};
\node at (-.4,-1) {$\vv_0(2)$};
\node at (2.1,-1) {$\vv_0(n)$};
\node at (3.9,-1) {$\vv_0(\ol n)$};
\node at (6.4,-1) {$\vv_0(\ol 2)$};
\node at (7.4,0) {$\vv_0(\ol1)$};

\coordinate (x) at ($(b)!.2!(c)$);
\coordinate (z) at ($(b)!.75!(c)$);
\coordinate (y) at ($(d)!.25!(e)$);
\coordinate (w) at ($(d)!.8!(e)$);

\draw (a) -- (b) node[pos=.5,yshift=5] {$e_1$}
      (b) -- (v2) node[pos=.5,xshift=6.5] {$e_2$}
      (vn) -- (c) node[pos=.5,xshift=6.5] {$e_n$}
      (c) -- (d) node[pos=.5,yshift=7] {$f_n$}
      (d) -- (voln) 
      (vol2) -- (e) 
      (e) -- (f) 
      (b) -- (x) node[pos=.5,yshift=7] {$f_2$}
      (z) -- (c) node[pos=.45,yshift=6] {$f_{n-1}$}
      (d) -- (y) 
      (w) -- (e); 
\draw[dashed] 
      (x) -- (z)
      (y) -- (w);
\end{tikzpicture} 

\caption{The tree $(T_0,\vv_0)=\cT_{\Theta_0,\ph_0}$.}\label{fig:T0v0}

\vspace{-2mm}

\end{figure}

It is checked directly that $\pi=\alpha\circ\psi_{T_0,\vv_0}$. Furthermore, algebraic independence of the $\alpha(t_{e_i})$ and $\alpha(t_{f_i})$ implies that $\alpha$ is an embedding, and we obtain $\ker\psi_{T_0,\vv_0}=\ker\pi=I_\tor$. Lemma~\ref{lem:ItorI} now implies that $\ker\psi_{T_0,\vv_0}$ is an initial ideal of $I$. However, by Lemma~\ref{lem:initinker}, $\ker\psi_{T_0,\vv_0}$ contains $\init_w I$ for any $w\in C_{T_0,\vv_0}$. Hence, $\init_w I=\ker\psi_{T_0,\vv_0}$ for any $w\in C_{T_0,\vv_0}$.

To extend the last equality to other ASPTs, we use the notion of flips. Let $\Theta$ be a maximal axially symmetric triangulation of type I. For a diagonal $\delta\neq\delta_0$ in $\Theta$, consider the two triangular cells of $\Theta$ adjacent to $\delta$ and let $\delta^*\neq\delta$ denote the other diagonal of the quadrilateral formed by these two triangles. Let $\sigma(\delta)$ denote the reflection of $\delta$ across $\delta_0$. We define
\[\Theta^\delta=(\Theta\bs\{\delta,\sigma(\delta)\})\cup\{\delta^*,\sigma(\delta)^*\}\] 
and say that $\Theta^\delta$ is obtained from $\Theta$ by a \textit{symmetric flip} of $\delta$. 
It is clear that any two maximal axially symmetric triangulations of type I can be obtained from each other by a sequence of symmetric flips.

\begin{proof}[Proof of Theorem~\ref{thm:toricdegens}]
For the ``only if'' part, suppose that $\init_w I$ is toric. Then the cone $C_{T,\vv}$ is necessarily maximal in $\Trop I$ and $(T,\vv)$ is a maximal ASPT. 
If $(T,\vv)$ has type II, we may find $a,b\in N$ such that $\Path(\vv(a),\vv(b))$ contains at least one $\sigma$-fixed edge, where $\sigma$ denotes the symmetry of $(T,\vv)$. Note that $\Path(\vv(a),\vv(\ol b))$ will contain the same $\sigma$-fixed edge, while $\Path(\vv(a),\vv(\ol a))$ and $\Path(\vv(b),\vv(\ol b))$ contain no $\sigma$-fixed edges. We deduce that $\init_w I$ contains the following reducible polynomial:
\[\init_w r_{a,b,\ol a,\ol b} = \init_w (x_{a,b}^2+x_{a,\ol b}^2-x_{a,\ol a}x_{b,\ol b}) = x_{a,b}^2+x_{a,\ol b}^2,\]
which contradicts the quadratic ideal $\init_w I$ being toric (and therefore prime).

We now show that $\init_w I$ is toric if $(T,\vv)$ is maximal of type I. By Corollary~\ref{cor:fansWinv} and \eqref{eq:taurespectsw}, we may assume that $(T,\vv)=\cT_{\Theta,\ph_0}$ for a maximal axially symmetric subdivision $\Theta$ of type I. We show that in this case $\init_w I=\ker\psi_{T,\vv}$.
In view of Lemma~\ref{lem:initinker}, it suffices to check that $\grdim(\ker\psi_{T,\vv})=\grdim I$. Since we have already seen that $\ker\psi_{T_0,\vv_0}$ is an initial ideal of $I$, we have reduced to the following claim:
\begin{equation}\label{eq:toricgrdims}
\grdim(\ker\psi_{T,\vv})=\grdim(\ker\psi_{T^*,\vv^*}),
\end{equation}
where $(T^*,\vv^*)=\cT_{\Theta^\delta,\ph_0}$ for some diagonal $\delta\neq\delta_0$ in $\Theta$.

Let $e$ denote the edge of $T$ corresponding to $\delta$ and let $e^*$ denote the edge of $T^*$ corresponding to $\delta^*$. Let $(T',\vv')$ be the ASPT obtained from $(T,\vv)$ by the symmetric contraction of $e$, it is also obtained from $(T^*,\vv^*)$ by symmetric contraction of $e^*$. We will write $\sigma$ to denote the symmetry of both $(T,\vv)$ and $(T^*,\vv^*)$. Furthermore, three sets are naturally in mutual bijection: the edges of $T$ other than $e$ and $\sigma(e)$, the edges of $T^*$ other than $e^*$ and $\sigma(e^*)$, and the edges of $T'$. For convenience of notation, we identify an edge lying in one of these three sets with the respective elements of the other sets. In particular, the polynomial ring $Q_{T',\vv'}$ (defined similarly to $Q_{T,\vv}$) becomes a subring of both $Q_{T,\vv}$ and $Q_{T^*,\vv^*}$. 

Consider the map $\psi_{T',\vv'}\colon R\to Q_{T',\vv'}$ taking $x_{a,b}$ to $\prod_{f\in\Path(\vv'(a),\vv'(b))} t_f$. If  $M'\in Q_{T',\vv'}$ is a monomial such that $M't_e^k$ lies in the image $\psi_{T,\vv}(R)$ for some $k$, then $M'$ lies in the image of $\psi_{T',\vv'}$. The same applies if $M't_{e^*}^l$ lies in $\psi_{T^*,\vv^*}(R)$ for some $l$. 

To verify~\eqref{eq:toricgrdims}, we establish a bijection between the monomials in $\psi_{T,\vv}(R)$ and those in $\psi_{T^*,\vv^*}(R)$ that respects the grading induced by the standard grading on $R$. Note that for a monomial $M$ contained in the image of $\psi_{T,v}$, $\psi_{T^*,\vv^*}$, or $\psi_{T',\vv'}$, its grading induced by the standard grading on $R$ is half of the total degree of $M$ in those $t_f$ for which $f$ is a leaf. Thus, taking the observation in the previous paragraph into account, it suffices to check that for every monomial $M'\in \psi_{T',\vv'}(R)$, the number of monomials of the form $M't_e^k$ in $\psi_{T,\vv}(R)$ equals the number of monomials of the form $M't_{e^*}^l$ in $\psi_{T^*,\vv^*}(R)$.  

Consider $M't_e^k\in \psi_{T,\vv}(R)$ and $M't_{e^*}^l\in\psi_{T^*,\vv^*}(R)$. Let $e_1$, $e_2$, $e_3$, $e_4$ be the edges incident to $e$ in $T$ (and to $e^*$ in $T^*$); we assume that $e_1$ and $e_2$ are incident in $T$ but $e_1$ and $e_3$ are incident in $T^*$. Since $\sigma(e)\neq e$, all $e_i$ lie in distinct $\sigma$-orbits. Let $k_i$ denote the power in which $M'$ contains $t_{e_i}$. We have two cases. 

Let $\sigma(e_i)\neq e_i$ for all $i$. Note that $\Path(\vv(a),\vv(b))$ contains $e$ if and only if it contains exactly one of $e_1$ and $e_2$ and exactly one of $e_3$ and $e_4$, and likewise for $\sigma(e)$ and the $\sigma(e_i)$. This provides two bounds: $k\ge\max(|k_1-k_2|,|k_3-k_4|)$ and $k\le \min(k_1+k_2,k_3+k_4)$. It also implies that $k$ is of the same parity as $k_1+k_2$ and $k_3+k_4$. The two bounds together with the parity condition define the set of all possible values of $k$. Similar restrictions, but with $k_2$ and $k_3$ exchanged, hold for $l$, and we observe that
\[\min(k_1+k_2,k_3+k_4)-\max(|k_1-k_2|,|k_3-k_4|)=\min(k_1+k_3,k_2+k_4)-\max(|k_1-k_3|,|k_2-k_4|)\]
for any real numbers $k_1, k_2, k_3, k_4$. 
Indeed, if $s\in S_4$ is such that ${k_{s(1)}\le k_{s(2)}\le k_{s(3)}\le k_{s(4)}}$, then both sides equal $2k_{s(1)}+\min(0,d_1-d_2)$, where $d_1=k_{s(2)}-k_{s(1)}$ and $d_2=k_{s(4)}-k_{s(3)}$.

Let $e_1=\sigma(e_1)$. Then, $\Theta$ and $\Theta^\delta$ have type (I) and $e_1$ is the edge corresponding to $\delta_0$. In this case, the argument is similar but $k_1$ is replaced with $2k_1$ in all conditions.
\end{proof}


\begin{remark}
Theorem~\ref{thm:toricdegens} implies that $\init_w I$ is prime whenever $w\in C_{T,\vv}$ for some $(T,\vv)$ that is obtained from a type I maximal ASPT by a sequence of symmetric contractions. Meanwhile, the argument in the first paragraph of the above proof can be applied to show that $\init_w I$ is not prime for all other $w$. As discussed in the beginning of Subsection~\ref{sec:signedtrops}, ASPTs obtained from type I maximal ASPTs by symmetric contractions are precisely the \textit{centrally symmetric phylogenetic trees}. Thus, the tropical cone $C_{T,\vv}$ is \textit{prime} if and only if $(T,\vv)$ is centrally symmetric.
\end{remark}
\vspace{-2.5mm}

\section{The cluster configuration space}

\subsection{Definition, tropicalization, and full rank cluster algebras}\label{sec:TropM}

We recall the definition of the cluster configuration space and its realization in $u$-coordinates, largely following~\cite[\S\S~3--4]{AHHL2021}. We then obtain a realization of its tropicalization in $u$-coordinates, as well as suitable Gr\"obner and tropical bases. 
We conclude the subsection by applying this construction to an extension of our main result to all full rank geometric type cluster algebras of finite type C.

The lineality space of $\Trop I$ is $L=C_{T,\vv}$, where $T$ has a single non-leaf vertex and $\vv$ is arbitrary. The vector space $L$ has a basis $w_1,\dots,w_n$, where 
\[(w_i)_{a,b}=|\{a,b\}\cap\{i,\ol i\}|\]
for $(a,b)\in D$. The ideal $I$ is homogeneous with respect to the $\bZ^n$-grading $(w_1,\dots,w_n)$, which we denote by $\grad$. Note that $w_i=w(T,\vv,l_i)$ for $(T,\vv)$ as above and $l_i$ that takes value 1 on the edges incident to $\vv(i)$ and $\vv(\ol i)$, and equals 0 otherwise.

Equivalently, $I$ is preserved by the action of the torus $(\bC^*)^n$ on $R$ given by 
\begin{equation}\label{eq:torus}
(t_1,\dots,t_n)x_{a,b}=t_{|a|}t_{|b|}x_{a,b}.
\end{equation}
This induces an action of $(\bC^*)^n$ on the cluster algebra $\cA$ and on the cluster variety $X$. We note that this action has the kernel $\{(1,\dots,1),(-1,\dots,-1)\}$ and, therefore, both $\cA$ and $X$ enjoy a faithful action of $(\bC^*)^n/(\bZ/2)$.

We also consider an involution $\eta$ of $R$ given by 
\begin{align*}
\eta(x_{i,j}) & = -x_{i,j},\quad 1\le i<j\le n,\\
\eta(x_{i,\ol j}) & = x_{i,\ol j},\quad 1\le i\le j\le n.
\end{align*}
This involution scales each $r_{a,b,c,d}$ by a factor of $\pm1$ and, therefore, preserves $I$. This allows us to view $\eta$ as an automorphism of $\cA$ and $X$. Note that $\eta$ commutes with the above torus action. Combining both actions we obtain a group of automorphisms
\[H\simeq (\bZ/2)\times(\bC^*)^n/(\bZ/2).\]

\begin{lemma}\label{lem:H}
The group $H$ contains all automorphisms of $R$ that multiply each $x_{i,j}$ by a scalar and preserve $I$.    
\end{lemma}
\begin{proof}
Let $\zeta$ be an automorphism of $R$ that multiplies $x_{a,b}$ by $k_{a,b}\in\bC^*$ and preserves $I$. Note that $\zeta$ must also multiply each $r_{a,b,c,d}$ by a scalar $k_{a,b,c,d}$. For $i\in[1,n-1]$ we have $k_{i,i+1,\ol i,\ol{i+1}}=k_{i,\ol{i+1}}^2=k_{i,\ol i} k_{i+1,\ol{i+1}}$. This allows us to choose $t_1$ so that $t_1^2=k_{1,\ol 1}$, then choose $t_2$ so that $t_2^2=k_{2,\ol 2}$ and $t_1t_2=k_{1,\ol 2}$, then choose $t_3$ so that $t_3^2=k_{3,\ol 3}$ and $t_2t_3=k_{2,\ol 3}$, etc. We claim that the obtained tuple $(t_1,\dots,t_n)$ satisfies $t_it_j=k_{i,\ol j}$ for all $1\le i\le j\le n$. We proceed by induction on $j-i$, the cases $j=i$ and $j-i=1$ holding by construction. For $j-i\ge 2$, we have $k_{i,j\ol{i+1},\ol j}=k_{i,\ol{i+1}}k_{j,\ol j}=k_{i,\ol j}k_{i+1,\ol j}$. We also have $k_{i,\ol{i+1}}k_{j,\ol j}=t_it_{i+1}t_j^2$ and, by induction hypothesis, $k_{i+1,\ol j}=t_{i+1}t_j$. Combining the last three equalities, we obtain $k_{i,\ol j}=t_it_j$. 

Since $k_{1,2,\ol 1,\ol 2}=k_{1,2}^2=k_{1,\ol 2}^2$, we have $k_{1,2}=\pm t_1t_2$. Suppose that $k_{1,2}=t_1t_2$. Then,  for $1\le i<j\le n$, we deduce $k_{i,j}=t_it_j$ from $k_{1,2,\ol i, \ol j}=k_{1,2}k_{i,j}=k_{1,\ol i}k_{2,\ol j}$. This implies that $\zeta$ acts as $(t_1,\dots,t_n)$. If $k_{1,2}=-t_1t_2$, then the same argument shows that $k_{i,j}=-t_it_j$ for all $1\le i<j\le n$, which means that $\zeta$ acts as $\eta\circ(t_1,\dots,t_n)$.
\end{proof}

By definition, $X$ is equipped with an embedding $X\subset\bC^D$, and we consider the very affine part $\mr X\subset X$ where none of the coordinates vanish. We denote by $\mr R$ the ring of Laurent polynomials $\bC[x_{a,b}^{\pm1}]_{(a,b)\in D}$, by $\mr I$ the ideal in $\mr R$ generated by $I\subset R\subset \mr R$, and set $\mr\cA=\mr R/\mr I$. In these notations, we have $\mr X=\Spec\mr\cA$. The action of $H$ naturally extends to $\mr R$, $\mr\cA$ and $\mr X$.

We consider the GIT quotient $\mr X//H=\Spec(\mr\cA^H)$, which we denote by $\cM$. 
In \cite[\S~4]{AHHL2021}, the \textit{cluster configuration space} of any finite type $\mathrm D$ is defined as the GIT quotient of the very affine part of a full rank cluster variety of type $\mathrm D$ by the group of those automorphisms that scale all cluster and frozen variables. Since the cluster algebra $\cA$ has full rank (see~\cite[Example~5.2.11]{IltNCTref2025}), Lemma~\ref{lem:H} implies that $\cM$ is precisely the \textit{cluster configuration space} of type $\rC_{n-1}$.

\begin{remark}
The action of the first component of the exact sequence~\eqref{eq:exactseq} is contained in the action of $H$. Thus, the action of $\wt W$ on $X$ descends to a $W$-action on $\cM$.
\end{remark}

Since $\mr\cA^H=\mr R^H/\mr J$, where $\mr J=\mr I\cap\mr R^H$, a coordinate system on $\cM$ is defined by choosing a generating set of $\mr R^H$. Note that $\mr R^H$ is spanned by Laurent monomials of $\grad$-grading zero whose total degree in the variables $x_{i,j}, 1\le i<j\le n$ is even. Let $\Lambda\subset\bZ^D$ be the lattice formed by the exponent vectors of such monomials. Note that $\dim\Lambda=n^2-n$.

Let $D^*$ consist of all $(a,b)\in D$ other than the $(n-1)$ pairs $(i,i+1)$ with $i\in[1,n-1]$ and the pair $(1,\ol n)$.\footnote{
In cluster algebra terminology, the generators $\Delta_{a,b}$ with $(a,b)\in D^*$ are the cluster variables of $\cA$, while the $\Delta_{a,b}$ with $(a,b)\in D\bs D^*$ are the frozen variables.}
For $a\in N\bs\{\ol n\}$, let $a^\pp$ denote the element covering $a$ in the order $\prec$, while $\ol n^\pp=1$.
For every $(a,b)\in D^*$, we set
\begin{equation}\label{eq:yab}
y_{a,b}=\frac{x_{a^\pp,b}x_{a,b^\pp}}{x_{a,b}x_{a^\pp,b^\pp}}\in\mr R.    
\end{equation}
Evidently, $y_{a,b}\in\mr R^H$ and can, therefore, be viewed as a function on $\cM$. These functions are precisely the functions $f_\gamma$ defined in~\cite[\S~4.2]{AHHL2021}, which generate $\bC[\cM]=\mr R^H$ by~\cite[Theorem~4.2]{AHHL2021}. Since $|D^*|=\dim\Lambda$, this implies that the exponent vectors of the $y_{a,b}$ form a $\bZ$-basis in $\Lambda$ (which can also be checked directly). Thus, we have identified $\mr R^H$ with the ring of Laurent polynomials $\bC[y_{a,b}^{\pm1}]_{(a,b)\in D^*}$.


This fixes an embedding of $\cM$ into $(\bC^*)^{D^*}$, defining the \textit{$u$-coordinates} on $\cM$. In turn, this embedding realizes the \textit{tropical cluster configuration space} $\Trop\cM=\Trop\mr J$ as a fan in $\bR^{D^*}$. 
Consider the linear surjection $\rho\colon\bR^D\to\bR^{D^*}$ with kernel $L$ defined by
\begin{equation}\label{eq:fanprojection}
\rho(w)_{a,b}=w_{a^\pp,b}+w_{a,b^\pp}-w_{a,b}-w_{a^\pp,b^\pp}.    
\end{equation}
For an ASPT $(T,\vv)$, set $C^*_{T,\vv}=\rho(C_{T,\vv})$, this is a simplicial cone of dimension $k(T,\vv)-n$. We deduce~\cite[Conjecture 7.5]{CoxMakh2025}:

\begin{theorem}\label{thm:TropM}
The tropical cluster configuration space $\Trop\cM$ is formed by the cones $C^*_{T,\vv}$ with $(T,\vv)$ ranging over all ASPTs.
\end{theorem}
\begin{proof}
We are to check that $\init_w I$ is monomial-free for $w\in\bR^D$ if and only if $\init_{\rho(w)}\mr J$ is also monomial-free. This is essentially seen by comparing~\eqref{eq:yab} and~\eqref{eq:fanprojection}. Indeed, a Laurent monomial $M\in\bC[y_{a,b}^{\pm1}]_{(a,b)\in D^*}$ may also be viewed as a Laurent monomial in $\mr R$ by~\eqref{eq:yab}. Formula~\eqref{eq:fanprojection} is chosen so that the grading of $M$ with respect to $w$ (as of a monomial in the $x_{a,b}$) is equal to its grading with respect to $\rho(w)$ as of a monomial in the $y_{a,b}$. Hence, $\init_{\rho(w)} p=\init_w p$ for any $p\in\mr R^H$.

Suppose that $\init_{\rho(w)}\mr J$ is not monomial-free, i.e., $\init_{\rho(w)} \mr p$ is a Laurent monomial for some $\mr p\in\mr J$. Then, $\init_w \mr p$ is also a Laurent monomial, both in the $y_{a,b}$ and the $x_{a,b}$. Since $\mr p\in\mr I$, we may find a monomial $M$ in the $x_{a,b}$ such that $M\mr p\in I$. However, $\init_w(M\mr p)=M\init_w \mr p$ is a monomial, meaning that $\init_w I$ is also not monomial-free.

Conversely, if $\init_w I$ is not monomial-free, then $\init_w p$ is a monomial for some $p\in I$. Since $I$ is $H$-invariant, it is spanned by \textit{$H$-homogeneous} elements: those on which any element of $H$ acts by a scalar. We may thus assume that $p$ is $H$-homogeneous and, therefore, $p/M\in\mr R^H$ for any monomial $M$ occurring in $p$. Hence, $\init_{\rho(w)}(p/M)=(\init_w p)/M$  will be a Laurent monomial both in the $x_{a,b}$ and the $y_{a,b}$.
\end{proof}

Here, we remark that the paper~\cite{CoxMakh2025} deals with the ideal $J=\mr J\cap\bC[y_{a,b}]_{(a,b)\in D^*}$ in the polynomial ring and with the affine variety $\wt\cM=\Spec(\bC[y_{a,b}]_{(a,b)\in D^*}/J)$ known as the \textit{binary geometry}. However, it is clear that $\Trop\cM=\Trop\mr J=\Trop J=\Trop\wt\cM$.

In order to describe the Gr\"obner and tropical bases, we recall a set of generators of $\mr J$, termed the \textit{$u$-equations} in~\cite{AHHLT2023,AHHL2021}. For $a\prec b\prec c\prec d$, set
\[r^*_{a,b,c,d}=x_{a,c}^{-1}x_{b,d}^{-1}r_{a,b,c,d}\in\mr R^H.\]
The $r^*_{a,b,c,d}$ generate $\mr J$ and each $r^*_{a,b,c,d}$ has the form $M_1+M_2-1$, where $M_1$ and $M_2$ are Laurent monomials in the $y_{a,b}$. Moreover, each $r^*_{a,b,c,d}$ is, in fact, a polynomial in the $y_{a,b}$, meaning that $M_1$ and $M_2$ are monomials rather than just Laurent monomials. To recall the explicit combinatorial description of these monomials found in~\cite[\S 10.1.2]{AHHL2021}, we introduce several notations. 
If $(a,b)\notin D^*$ such that $a\neq b$, $a^\pp\neq b$ and $a\neq b^\pp$, we set $y_{a,b}=y_{c,d}$, where $(c,d)$ is the unique element of $\{(a,b),(b,a),(\ol a,\ol b),(\ol b,\ol a)\}\cap D^*$. For disjoint $A,B\subset N$ such that $a^\pp\neq b$ and $a\neq b^\pp$ for any $a\in A$ and $b\in B$, set $M_{A,B}=\prod_{a\in A,b\in B}y_{a,b}$. Finally, for $a,b\in N$ we define $[a,b)$ as $\{c\mid a\preceq c\prec b\}$ if $a\preceq b$, and as $\{c\mid c\succeq a\text{ or }c\prec b\}$ otherwise.  One may check directly that 
\begin{equation}\label{eq:extendedueq}
r^*_{a,b,c,d}=M_{[a,b),[c,d)}+M_{[b,c),[d,a)}-1.    
\end{equation}

The coordinate ring $\mr\cA^H$ of $\cM$ is generated by the elements $u_{a,b}^{\pm1}$, where $u_{a,b}$ is the quotient-map image of $y_{a,b}$. Substituting $y_{a,b}$ with $u_{a,b}$ in each $r^*_{a,b,c,d}$, we obtain a set of polynomial relations satisfied by the $u_{a,b}$, the \textit{extended $u$-equations}. Among them are the \textit{primitive $u$-equations} enumerated by $D^*$, those obtained from
\[r^*_{a,b}:=r^*_{a,a++,b,b++}=y_{a,b}+M_{[a^\pp,b),[b^\pp,a)}-1.\]
The special role of the primitive $u$-equations is that the elements $r^*_{a,b}$, $(a,b)\in D^*$ already generate both the ideal $\mr J$ and the ideal $J$, see \cite[Theorems 4.2 and 5.1]{AHHL2021}. 

Directly from Corollary~\ref{cor:grobnerbasis} one deduces
\begin{cor}
The expressions $r^*_{a,b,c,d}$ with $a\prec b\prec c\prec d$ form a Gr\"obner basis of both $\mr J$ and $J$ with respect to any $w\in|\Trop \cM|$.    
\end{cor}

Next, for $1\le i<j<k\le n$ we consider the element
\[s^*_{i,j,k}=x_{i,k}^{-1}x_{i,\ol j}^{-1}x_{j,\ol k}^{-1}s_{i,j,k}=M_{[j,k),[\ol k,i)}+M_{[i,j),[\ol j,\ol k)}+M_{[i,j),[k,\ol i)}-2.\]
From Corollary~\ref{cor:tropicalbasis} we obtain
\begin{cor}
The expressions $r^*_{a,b,c,d}$ with $a\prec b\prec c\prec d$, together with the expressions $s^*_{i,j,k}$ for all $1\le i<j<k\le n$, form a tropical basis of both $\mr J$ and $J$.
\end{cor}

\begin{example}
Let $n=3$. We have 6 generators: $y_{1,3}$, $y_{1,\ol 1}$, $y_{1,\ol 2}$, $y_{2,\ol 2}$, $y_{2,\ol 3}$, $y_{3,\ol 3}$. There are 9 different relations of the form $r^*_{a,b,c,d}$, three of which are
\begin{gather*}
r^*_{1,2,3,\ol1}=r^*_{1,3}=y_{1,3}+y_{1,\ol 2}y_{2,\ol 2}y_{2,\ol 3}-1,\quad r^*_{1,2,\ol 1,\ol 2}=r^*_{1,\ol 1}=y_{1,\ol 1}+y_{2,\ol 2}y_{2,\ol 3}^2y_{3,\ol 3}-1,\\
r^*_{1,2,3,\ol 2}=y_{1,3}y_{1,\ol 1}+y_{2,\ol 2}y_{2,\ol 3}-1.    
\end{gather*}
The other 6 are obtained from these by iterations of the circular shift $y_{a,b}\mapsto y_{a^\pp,b^\pp}$. Meanwhile, the only element of the form $s^*_{i,j,k}$ is 
\[s^*_{1,2,3}=y_{2,\ol 3}+y_{1,\ol 2}+y_{1,3}-2.\]
\end{example}

Now, consider an arbitrary full rank geometric type cluster algebra of type $\rC_{n-1}$. That is, the algebra $\cA(\tilde B)$ given by an $(n-1+m)\times(n-1)$ extended exchange matrix $\tilde B$ of type $\rC_{n-1}$ and rank $n-1$, see~\cite[\S~2]{AHHL2021}. This algebra is generated by $n(n-1)$ cluster variables indexed by $D^*$, together with $m$ frozen variables and their inverses. Hence, $\cA(\tilde B)$ is the quotient of $\bC[x_1^{\pm1},\dots,x_m^{\pm1}][x_{a,b}]_{(a,b)\in D^*}$ by an ideal $I(\tilde B)$. This realizes the tropicalization of the cluster variety $X(\tilde B)=\Spec\cA(\tilde B)$ as the fan $\Trop I(\tilde B)$ in $\bR^{|D^*|+m}$. As shown in~\cite{LamSpe2022}, the group $H(\tilde B)$ of all automorphisms of $\cA(\tilde B)$ that act by scaling the cluster and frozen variables is an $m$-dimensional abelian algebraic group. This implies that the lineality space of $\Trop X(\tilde B)$ has dimension $m$. 

Let $\mr\cA(\tilde B)$ be the localization of $\cA(\tilde B)$ in all cluster variables and $\mr X(\tilde B)=\Spec\mr\cA(\tilde B)$ be the very affine part of $X(\tilde B)$. 
By~\cite[Theorem 4.2]{AHHL2021} the GIT quotient $\mr X(\tilde B)//H(\tilde B)=\Spec\mr\cA(\tilde B)^{H(\tilde B)}$ is isomorphic to $\cM$ for any such $\tilde B$. 
However, $\mr\cA(\tilde B)^{H(\tilde B)}$ is the quotient of $\cR=\bC[x_i^{\pm1},x_{a,b}^{\pm1}]^{H(\tilde B)}$ by its intersection with the defining ideal of $\mr X(\tilde B)$. Furthermore, $\cR$ is a ring of Laurent polynomials generated by $|D^*|$ monomials in the $x_i$ and $x_{a,b}$. Thus, we obtain another embedding of $\cM$ into $(\bC^*)^{|D^*|}$ and, similar to Theorem~\ref{thm:TropM}, a linear surjection of $\Trop X(\tilde B)$ onto $\Trop\cM$. Here, $\Trop\cM$ is taken with respect to the new embedding, however, any two very affine embeddings of $\cM$ provide linearly equivalent tropicalizations, see~\cite[Proposition 6.4.4]{MacStu2015}.

\begin{cor}\label{cor:fullrank}
The quotient of $\Trop X(\tilde B)$ by its $m$-dimensional lineality space is linearly equivalent to the fan $\Trop\cM$ described by Theorem~\ref{thm:TropM}.
\end{cor}

\subsection{Sign patterns}\label{sec:signpatterns}

The real part of the cluster configuration space $\cM\subset(\bC^*)^{D^*}$ is the real very affine variety $\cM(\bR)=\cM\cap(\bR^*)^{D^*}$. In~\cite[Proposition 7.6]{AHHL2021} it is shown that the number of connected components in $\cM(\bR)$ is equal to $2^{n-2}(n+1)(n-1)!$. In \cite[\S~11.3]{AHHL2021} it is conjectured that the number of sign patterns occurring in $\cM$ is also $2^{n-2}(n+1)(n-1)!$. The goal of this subsection is to prove this conjecture. Since the number of connected components is no less than the number of occurring sign patterns, it suffices to find $2^{n-2}(n+1)(n-1)!$ distinct occurring sign patterns.

Let $\alpha=(\alpha_1:\alpha_2)$, $\beta=(\beta_1:\beta_2)$, $\gamma=(\gamma_1:\gamma_2)$, $\delta=(\delta_1:\delta_2)$ be four pairwise distinct points in $\bC\bP^1$. Recall that their cross-ratio is defined as
\begin{equation}\label{eq:crossratio}
(\alpha,\beta;\gamma,\delta)=\frac{
\begin{vsmallmatrix}\alpha_1&\gamma_1\\\alpha_2&\gamma_2\end{vsmallmatrix} \begin{vsmallmatrix}\beta_1&\delta_1\\\beta_2&\delta_2\end{vsmallmatrix}}
{\begin{vsmallmatrix}\alpha_1&\delta_1\\\alpha_2&\delta_2\end{vsmallmatrix}
\begin{vsmallmatrix}\beta_1&\gamma_1\\\beta_2&\gamma_2\end{vsmallmatrix}}\in\bC^*.    
\end{equation}
If all four points are real, then the cross-ratio is also real with its sign determined as follows: it is negative if and only if $\alpha$ and $\beta$ lie in distinct connected components of $\bR\bP^1\bs\{\gamma,\delta\}$. Visualizing $\bR\bP^1$ as a circle, one may say that $(\alpha,\beta;\gamma,\delta)$ is negative if and only if the chord connecting $\alpha$ and $\beta$ crosses the chord connecting $\gamma$ and $\delta$. 

Next, \textit{dihedral orderings} are orderings of $N$ considered up to circular shifts and reversals. There are $(2n-1)!/2$ dihedral orderings in total. With a labeling $\ph$ of the polygon $P$ one naturally associates the dihedral ordering $\la(\ph)$ obtained by reading the labels in $\ph$ clockwise or counter-clockwise, starting from any side. Two labelings provide the same dihedral ordering if and only if they can be identified by the natural action of $D_{4n}$ on $P$.

Let $(\alpha_a)_{a\in N}$ be a tuple of pairwise distinct points in $\bR\bP^1$ indexed by $N$. Since $\bR\bP^1$ is a circle, such a tuple can be viewed as a labeling of $2n$ points in the circle by $N$. This defines a dihedral ordering, we will call it the dihedral ordering \textit{defined by} the tuple $(\alpha_a)_{a\in N}$. The above discussion shows that this dihedral ordering determines the sign of the cross-ratio of any four points in the tuple. The converse is also true; in fact, it suffices to consider only $n(2n-3)$ out of the ${2n}\choose4$ cross-ratios.
\begin{prop}\label{prop:dihedralcoords}
Let $(\alpha_a)_{a\in N}$ and $(\beta_a)_{a\in N}$ be two tuples of pairwise distinct points in $\bR\bP^1$. The dihedral orderings defined by these tuples coincide if and only if for every pair $a\prec b$ such that $a^\pp\neq b$ and $a\neq b^\pp$, the cross-ratios $(\alpha_a,\alpha_{a^\pp};\alpha_{b^\pp},\alpha_b)$ and $(\beta_a,\beta_{a^\pp};\beta_{b^\pp},\beta_b)$ have the same sign.
\end{prop}
\begin{proof}
The ``only if'' part was already discussed above. For the converse, suppose that the dihedral orderings are distinct. This means that there are $a_0,b_0\in N$ that are adjacent in the dihedral ordering defined by $(\alpha_a)_{a\in N}$ but not in that defined by $(\beta_a)_{a\in N}$. Consequently, we have $c_0,d_0\in N$ for which $(\alpha_{a_0},\alpha_{b_0};\alpha_{c_0},\alpha_{d_0})>0$ but $(\beta_{a_0},\beta_{b_0};\beta_{c_0},\beta_{d_0})<0$. It remains to recall that by~\cite[Lemma 2.2]{Bro2009}, $(\alpha_{a_0},\alpha_{b_0};\alpha_{c_0},\alpha_{d_0})$ can, up to sign, be expressed as a product of the $(\alpha_a,\alpha_{a^\pp};\alpha_{b^\pp},\alpha_b)$ and their inverses. Since the same product with $\alpha$ replaced by $\beta$ equals $(\beta_{a_0},\beta_{b_0};\beta_{c_0},\beta_{d_0})$, we conclude that $(\alpha_a,\alpha_{a^\pp};\alpha_{b^\pp},\alpha_b)$ and $(\beta_a,\beta_{a^\pp};\beta_{b^\pp},\beta_b)$ have different signs for at least one pair $a,b$.
\end{proof}


Next, a labeling $\ph$ of $P$ is \textit{centrally symmetric} if the sides $\ph(a)$ and $\ph(\ol a)$ are opposite to each other for every $a\in N$. A \textit{centrally symmetric dihedral ordering} (or \textit{CSDO}) is a dihedral ordering of the form $\la(\ph)$ for a centrally symmetric labeling $\ph$ of $P$. There are $2^{n-2}(n-1)!$ distinct CSDOs. Every CSDO $\la$ determines a sign pattern $\nu(\la)\in\{1,-1\}^{D^*}$: consider a tuple $(\alpha_a)_{a\in N}$ of points in $\bP^1$ that defines $\la$ and set
\[\nu(\la)_{a,b}=\sgn(\alpha_a,\alpha_{a^\pp};\alpha_{b^\pp},\alpha_b)=\sgn(\alpha_{\ol a},\alpha_{\ol a^\pp};\alpha_{\ol b^\pp},\alpha_{\ol b}).\]
By Proposition~\ref{prop:dihedralcoords}, for CSDOs $\la_1\neq \la_2$ we have $\nu(\la_1)\neq\nu(\la_2)$. 

Now, consider pairwise distinct points $(\alpha_i)_{i\in [1,n]}$ in $\bR\bP^1$ with $\alpha_i=(\alpha_{1,i}:\alpha_{2,i})$. In addition, assume that $\alpha_i\neq(\alpha_{2,j}:-\alpha_{1,j})$ for $i,j\in[1,n]$. For $i\in[1,n]$, set $\alpha_{1,\ol i}=-\alpha_{2,i}$, $\alpha_{2,\ol i}=\alpha_{1,i}$ and $\alpha_{\ol i}=(\alpha_{1,\ol i}:\alpha_{2,\ol i})$. By Proposition~\ref{prop:AinS1}, the point 
\begin{equation}\label{eq:csdopreimage}
\left(\begin{vsmallmatrix}\alpha_{1,a}&\alpha_{1,b}\\\alpha_{2,a}&\alpha_{2,b}\end{vsmallmatrix}\right)_{(a,b)\in D}\in (\bC^*)^D
\end{equation}
lies in $\mr X$. By comparing~\eqref{eq:yab} and~\eqref{eq:crossratio}, we see that the point 
\[((\alpha_a,\alpha_{a^\pp};\alpha_{b^\pp},\alpha_b))_{(a,b)\in D^*}\in(\bR^*)^{D^*}\] 
lies in $\cM(\bR)$.
 
Evidently, the dihedral ordering defined by the tuple $(\alpha_a)_{a\in N}$ is a CSDO and any CSDO can be obtained from a tuple of this form. We obtain the following.
\begin{prop}\label{prop:CSpatterns}
Each of the $2^{n-2}(n-1)!$ pairwise distinct sign patterns of the form $\nu(\la)$ with $\la$ a CSDO occurs in $\cM$.
\end{prop}

To construct the remaining sign patterns, we consider the second realization of $\cA$ provided by~\cite[Proposition 12.13]{FomZel2003}.
\begin{prop}\label{prop:AinS2}
There is an injective homomorphism $\iota'$ from $\cA$ to $S$ defined on the generators as follows:
\begin{align*}
\iota'(\Delta_{i,j}) & = z_{1,i}z_{2,j}-z_{1,j}z_{2,i},\quad 1\le i<j\le n,\\
\iota'(\Delta_{i,\ol j}) & = \mathrm i(z_{1,i}z_{2,j}+z_{1,j}z_{2,i}),\quad 1\le i\le j\le n.
\end{align*}
\end{prop}

Consider the matrix $Z'$ with columns labeled by $N$ and rows labeled by $\{1,2\}$ such that $Z'_{1,i}=z_{1,i}$ and $Z'_{2,i}=z_{2,i}$, while $Z'_{1,\ol i}=-\mathrm iz_{1,i}$ and $Z'_{2,\ol i}=\mathrm iz_{2,i}$, for all $i\in[1,n]$.
The embedding $\iota'$ takes $\Delta_{a,b}$ to the minor in $Z'$ spanned by columns $a$ and $b$.

An \textit{axially symmetric dihedral ordering} (or \textit{ASDO}) is a dihedral ordering of the form $\la(\ph)$ for an axially symmetric labeling $\ph$ of $P$. There are $2^{n-2}n!$ distinct ASDOs. Every ASDO $\la$ also determines a sign pattern $\nu(\la)\in\{1,-1\}^{D^*}$: if $\la$ is defined by the tuple $(\alpha_a)_{a\in N}$ of points in $\bP^1$, then
\begin{equation*}\label{eq:nuasdo}
\nu(\la)_{a,b}=\sgn(\alpha_a,\alpha_{a^\pp};\alpha_{b^\pp},\alpha_b)=\sgn(\alpha_{\ol a},\alpha_{\ol a^\pp};\alpha_{\ol b^\pp},\alpha_{\ol b}).    
\end{equation*}
Again, for ASDOs $\la_1\neq \la_2$ we have $\nu(\la_1)\neq\nu(\la_2)$. 

Consider pairwise distinct points $(\alpha_i)_{i\in [1,n]}$ in $\bR\bP^1$ with $\alpha_i=(\alpha_{1,i}:\alpha_{2,i})$ such that $\alpha_i\neq(-\alpha_{1,j}:\alpha_{2,j})$ for $i,j\in[1,n]$. For $i\in[1,n]$, set $\alpha_{1,\ol i}=-\mathrm i\alpha_{1,i}$, $\alpha_{2,\ol i}=\mathrm i\alpha_{2,i}$ and $\alpha_{\ol i}=(\alpha_{1,\ol i}:\alpha_{2,\ol i})\in\bR\bP^1$. By Proposition~\ref{prop:AinS2}, we have 
\[\left(\begin{vsmallmatrix}\alpha_{1,a}&\alpha_{1,b}\\\alpha_{2,a}&\alpha_{2,b}\end{vsmallmatrix}\right)_{(a,b)\in D}\in \mr X\]
and 
\[((\alpha_a,\alpha_{a^\pp};\alpha_{b^\pp},\alpha_b))_{(a,b)\in D^*}\in\cM(\bR).\]
 The dihedral ordering defined by $(\alpha_a)_{a\in N}$ is an ASDO and any ASDO can be obtained from a tuple of this form. We obtain the following.
 
\begin{prop}\label{prop:ASpatterns}
Each of the $2^{n-2}n!$ pairwise distinct sign patterns of the form $\nu(\la)$ with $\la$ an ASDO occurs in $\cM$.
\end{prop}

Now, observe that no dihedral ordering is both a CSDO and an ASDO. Combining Propositions~\ref{prop:CSpatterns} and~\ref{prop:ASpatterns} with the mentioned result in~\cite{AHHL2021} that counts connected components of $\cM(\bR)$, we obtain a description of all occurring sign patterns.
\begin{theorem}\label{thm:signpatterns}
There are exactly $2^{n-2}(n+1)(n-1)!$ sign patterns that occur in $\cM$: those of the form $\nu(\la)$ with $\la$ a CSDO and those of the form $\nu(\mu)$ with $\mu$ an ASDO.
\end{theorem}

\subsection{Signed tropicalizations}\label{sec:signedtrops}

Our next goal is to describe the signed tropicalizations of $\cM$ corresponding to all of the occurring sign patterns, proving~\cite[Conjecture 7.6]{CoxMakh2025}. We recall the notion of centrally symmetric phylogenetic trees introduced in~\cite{CoxMakh2025}.

\begin{defn}\label{def:CSPT}
A subdivision $\Theta$ of $P$ is \textit{centrally symmetric} if with every diagonal $\delta$ it also contains the central reflection of $\delta$ in the center of $P$. A \textit{centrally symmetric phylogenetic tree} (or \textit{CSPT}) is a phylogenetic tree of the form $\cT_{\Theta,\ph}$ with $\Theta$ and $\ph$ centrally symmetric.
\end{defn}

Every CSPT is also an ASPT. Informally, that is since one may ``twist'' one of the two mutually symmetric halves of a CSPT to obtain an ASPT; for a precise argument see~\cite[Proposition 6.3]{CoxMakh2025}. In particular, if the subdivision $\Theta$ in Definition~\ref{def:CSPT} is a centrally symmetric triangulation, then $\cT_{\Theta,\ph}$ is a maximal ASPT of type (I).
Therefore, one may define CSPTs as those APSTs that can be obtained from a maximal ASPT of type (I) by a sequence of symmetric contractions. Consequently, those cones $C_{T,\vv}$ for which $(T,\vv)$ is a CSPT form a subfan in $\Trop X$ of pure dimension $2n-1$. The respective cones $C^*_{T,\vv}$ form a subfan in $\Trop\cM$.

\begin{defn}
A phylogenetic tree $(T,\vv)$ is \textit{compatible} with a dihedral ordering $\la$ if there is a subdivision $\Theta$ and a labeling $\ph$ of $P$ such that $(T,\vv)=\cT_{\Theta,\ph}$ and $\la=\la(\ph)$.
\end{defn}

Recall that the $(n-1)$-dimensional cyclohedron $\cW_n$ is a polytope whose faces are enumerated by centrally symmetric subdivisions of $P$, with inclusion of faces corresponding to coarsening of subdivisions. The $(n-1)$-dimensional associahedron $\cK_{n+1}$ is a polytope whose faces are enumerated by subdivisions of a convex $(n+2)$-gon. By~\cite[Proposition~5.2]{CoxMakh2025}, the faces of $\cK_{n+1}$ are also enumerated by axially symmetric subdivisions of $P$, with a similar characterization of face inclusions.

We may now state the main result of this subsection as follows.

\begin{theorem}\label{thm:signedtrops}
\hfill
\begin{enumerate}[label=(\alph*)]
\item For a CSDO $\la$, the signed tropicalization $\Trop_{\nu(\la)}\cM$ consists of those cones $C^*_{T,\vv}$ for which $(T,\vv)$ is a CSPT that is compatible with $\la$. The fan $\Trop_{\nu(\la)}\cM$ is combinatorially equivalent to the dual fan of $\cW_n$.
\item For an ASDO $\mu$, the signed tropicalization $\Trop_{\nu(\mu)}\cM$ consists of those cones $C^*_{T,\vv}$ for which $(T,\vv)$ is an ASPT that is compatible with $\mu$. The fan $\Trop_{\nu(\mu)}\cM$ is combinatorially equivalent to the dual fan of $\cK_{n+1}$.
\end{enumerate}
\end{theorem}

Before presenting the proof, we make the following observation.
\begin{lemma}\label{lem:phsimthetasim}
\hfill
\begin{enumerate}[label=(\alph*)]
\item If $\cT_{\Theta,\ph}$ is a CSPT and the labeling $\ph$ is centrally symmetric, then the subdivision $\Theta$ is also centrally symmetric.
\item If $\cT_{\Theta,\ph}$ is an ASPT and the labeling $\ph$ is axially symmetric, then the subdivision $\Theta$ is also axially symmetric.
\end{enumerate}
\end{lemma}
\begin{proof}
For part (a), suppose that the labeling $\ph$ is centrally symmetric and the subdivision $\Theta$ is not. Let $\delta\in\Theta$ be such that the central reflection of $\delta$ is not in $\Theta$. Let $\ph(a_1),\dots,\ph(a_k)$ be the sides of $P$ lying on one side of $\delta$, where $k\le n$. Let $\cT_{\Theta,\ph}=(T,\vv)$. The edge corresponding to $\delta$ in $T$ has the property that $\vv(a_1),\dots,\vv(a_k)$ lie on one side of this edge, while all other $\vv(a)$ lie on the other. However, $T$ contains no edge that would similarly separate $\vv(\ol{a_1}),\dots,\vv(\ol{a_k})$ from the other leaves, which implies that $(T,\vv)$ is not a CSPT.
The argument for part (b) is similar.
\end{proof}

\begin{proof}[Proof of Theorem~\ref{thm:signedtrops}]
By Corollary~\ref{cor:fansWinv} the action of the group $W$ on $\bR^D$ preserves $\Trop I$. In particular, it preserves the lineality space $L$ and, therefore, descends to a $W$-action on $\bR^{D^*}$ that preserves $\Trop\cM$. 
The group $W$ also acts naturally on the set of dihedral orderings, respecting both central symmetry and axial symmetry. On one hand, for $\la$ either a CSDO or an ASDO and $\tau\in W$, we have $\Trop_{\nu(\tau\la)}\cM=\tau(\Trop_{\nu(\la)}\cM)$. On the other, a CSPT (resp.\ ASPT) $(T,\vv)$ is compatible with the CSDO (resp.\ ASDO) $\la$  if and only if $(T,\vv\circ\tau^{-1})$ is compatible with $\tau\la$. This means that both parts of the theorem can be verified for a single dihedral ordering.

We start with part (a), specifically, we consider the CSDO $\la_0$ obtained by ordering $N$ according to $\prec$. The corresponding sign pattern is $\nu(\la_0)=(1,\dots,1)$ and $\Trop_{\nu(\la_0)}\cM$ is the positive tropicalization $\Trop_{>0}\cM$. In view of~\eqref{eq:yab}, for any $w\in\bR^D$, the ideal $\init_{\rho(w)} \mr J$ contains no elements of $\bR_{>0}[y_{a,b}]_{(a,b)\in D^*}$ if and only if $\init_{w} I$ contains no elements of $\bR_{>0}[x_{a,b}]_{(a,b)\in D}$. First, we show that if $w\in C_{T,\vv}$ and $(T,\vv)$ is a CSPT compatible with $\la_0$, then $w$ lies in $\Trop_{>0} I$. We may assume that $(T,\vv)$ is maximal. Then, for every $a\prec b\prec c\prec d$, the initial form $\init_w r_{a,b,c,d}$ is a binomial that contains the term $-x_{a,c}x_{b,d}$ (due to compatibility with $\la_0$). These binomials generate $\init_w I$, hence $(1,\dots,1)\in\bC^D$ lies in the zero set of $\init_w I$. This implies that $\init_w I\cap\bR_{>0}[x_{a,b}]_{(a,b)\in D}=\varnothing$. 

For the converse, we consider an ASPT $(T,\vv)$ that is not compatible with $\la_0$. This is equivalent to the existence of elements $a\prec b\prec c\prec d$ and an edge $e$ in $T$, such that $\vv(a)$ and $\vv(c)$ lie in one component of $T\bs e$ while $\vv(b)$ and $\vv(d)$ lie in the other. For such elements $a,b,c,d$, edge $e$ and $w\in C_{T,\vv}$, we must have $\init_w r_{a,b,c,d}=x_{a,b}x_{c,d}+x_{a,d}x_{b,c}$. Consequently, $w$ does not lie in $\Trop_{>0} I$.

It remains to note that by Lemma~\ref{lem:phsimthetasim}(a), CSPTs that are compatible with $\la_0$ are enumerated by centrally symmetric subdivisions of $P$.
Hence, so are the cones of $\Trop_{>0}\cM$, with inclusion of cones corresponding to refinement of subdivisions. We deduce that $\Trop_{>0}\cM$ is combinatorially equivalent to the dual fan of $\cW_n$. 

For part (b), we consider the ASDO $\mu_0=\la(\ph_0)$, which is also obtained by ordering $N$ according to $\dotprec$. The sign pattern $\nu(\mu_0)$ satisfies $\nu(\mu_0)_{n,\ol n}=-1$ while all other $\nu(\mu_0)_{a,b}=1$. Let $(T,\vv)$ be a maximal ASPT compatible with $\mu_0$, choose $w\in C_{T,\vv}$. We are in the setting of Lemma~\ref{lem:initinker}, in particular, the point $(c_{T,\vv}(a,b))_{a,b}\in\bC^D$ lies in the zero set of $\init_w I$. Consequently, the point $(p_{a,b})_{(a,b)\in D^*}\in(\bC^*)^{D^*}$ with
\[p_{a,b}=\frac{c_{T,\vv}(a,b^\pp)c_{T,\vv}(a^\pp,b)}{c_{T,\vv}(a,b)c_{T,\vv}(a^\pp,b^\pp)}\]
lies in the zero set of $\init_{\rho(w)}\mr J$. The definition of $c_{T,\vv}(a,b)$ shows that all $p_{a,b}$ are real and $p_{a,b}>0$ if and only if $(a,b)\neq(n,\ol n)$. We conclude that $\rho(w)$ lies in $\Trop_{\nu(\mu_0)}\cM$.

It remains to show that $\rho(w)\notin|\Trop_{\nu(\mu_0)}\cM|$ for $w\in C_{T,\vv}$ such that $(T,\vv)$ is not compatible with $\mu_0$. This means that there are elements $a\dotprec b\dotprec c\dotprec d$ and an edge $e$ in $T$ such that $\vv(a)$ and $\vv(c)$ lie in one component of $T\bs e$ while $\vv(b)$ and $\vv(d)$ lie in the other. Considering $(\ol a,\ol b,\ol c,\ol d)$ instead if necessary, we may assume that $a,b\in[1,n]$. Thus, 
\[\init_w r_{a,b,c,d}=x_{a,b}x_{c,d}\pm x_{a,d}x_{b,c},\]
where the second coefficient is negative if and only if $a\prec b\prec d\prec c$, i.e., $c\notin[1,n]$. First, suppose that $c\in[1,n]$. In this case, $r^*_{a,b,c,d}=x_{a,c}^{-1}x_{b,d}^{-1}r_{a,b,c,d}$. In view of~\eqref{eq:extendedueq}, we have
\[\init_{\rho(w)}r^*_{a,b,c,d}=x_{a,c}^{-1}x_{b,d}^{-1}\init_w r_{a,b,c,d}=M_{[a,b),[c,d)}+M_{[b,c),[d,a)}\in\init_{\rho(w)}\mr J.\] 
Note that if $d\in[1,n]$, then $n,\ol n\in[d,a)$. If $d\notin[1,n]$, then $n\in[c,d)$ while $\ol n\in[d,a)$. Both of these conditions imply that $y_{n,\ol n}$ does not appear in $M_{[a,b),[c,d)}+M_{[b,c),[d,a)}$. Hence, the automorphism $\epsilon_{\nu(\mu_0)}$ which substitutes $y_{n,\ol n}$ with $-y_{n,\ol n}$ (Definition~\ref{def:signedtrop}) fixes $\init_{\rho(w)}r^*_{a,b,c,d}=M_{[a,b),[c,d)}+M_{[b,c),[d,a)}$. In particular, $\epsilon_{\nu(\mu_0)}(\init_{\rho(w)}r^*_{a,b,c,d})$ has positive coefficients, which implies that $\rho(w)\notin|\Trop_{\nu(\mu_0)}\cM|$. 

Suppose that $c\notin[1,n]$. Then, $r^*_{a,b,d,c}=x_{a,d}^{-1}x_{b,c}^{-1}r_{a,b,d,c}$ and
\[\init_{\rho(w)}r^*_{a,b,d,c}=x_{a,d}^{-1}x_{b,c}^{-1}\init_w r_{a,b,d,c}=M_{[b,d),[c,a)}-1\in\init_{\rho(w)}\mr J.\] 
In this case, $n\in[b,d)$ and $\ol n\in[c,a)$, hence $y_{n,\ol n}$ appears in degree 1 in $M_{[b,d),[c,a)}$. Therefore, $\epsilon_{\nu(\mu_0)}(-\init_{\rho(w)}r^*_{a,b,c,d})=M_{[b,d),[c,a)}+1$, which again implies $\rho(w)\notin|\Trop_{\nu(\mu_0)}\cM|$.

Finally, by Lemma~\ref{lem:phsimthetasim}(b), ASPTs that are compatible with $\mu_0$ are enumerated by axially symmetric subdivisions of $P$. This implies that $\Trop_{\nu(\mu_0)}\cM$ is combinatorially equivalent to the dual fan of $\cK_{n+1}$. 
\end{proof}

\begin{remark}
In the above proof, we have established that, modulo its lineality space, $\Trop_{>0} X$ is combinatorially equivalent to the dual fan of $\cW_n$. This is a special case of~\cite[Conjecture 8.1]{SpeWil2005}, which was proved in~\cite{AHHL2021,JaLoSt2021}. 
\end{remark}

We can now apply the results in this section to describe sign patterns and signed tropicalizations of the cluster variety $X$. For a dihedral ordering $\la$, let $\Omega_\la$ denote the subfan of $\Trop X$ formed by all $C_{T,\vv}$ such that $(T,\vv)$ is compatible with $\la$. If $\la$ is a CSDO, then, by Theorem~\ref{thm:signedtrops}(a), $\Omega_\la$ is the preimage of $\Trop_{\nu(\la)}\cM$ under the linear surjection $\rho:\Trop X\to\Trop M$ defined by~\eqref{eq:fanprojection}. In particular, $\Omega_\la$ is combinatorially equivalent to the product of the dual fan of $\cW_n$ with $\bR^n$.

\needspace{2\baselineskip}
\begin{cor}\label{cor:Xsigns}
\hfill
\begin{enumerate}[label=(\alph*)]
\item There are exactly $2^{2n-1}(n-1)!$ sign patterns that occur in $X$. For every such sign pattern $\nu$, the set $X_\nu$ (see~\eqref{eq:Xnu}) is connected.
\item For every sign pattern $\nu$ occurring in $X$, there is a CSDO $\la$ such that $\Trop_\nu X=\Omega_\la$. For a CSDO $\la$, this equality holds for exactly $2^{n+1}$ sign patterns $\nu$.
\end{enumerate}
\end{cor}

\begin{proof}
With a matrix $m\in\bM=\mathrm{Mat}_{2,n}(\bC)$, we associate the $\{1,2\}\times N$ matrix
\[B(m)=
\begin{pmatrix}
m_{1,1} & \dots & m_{1,n} & -m_{2,1} & \dots & -m_{2,n}\\
m_{2,1} & \dots & m_{2,n} & m_{1,1} & \dots & m_{1,n} 
\end{pmatrix},\]
whose columns are ordered according to $\prec$. Consider the map $f\colon\bM\to\bC^D$ given by 
\[f(m)=\left(
\begin{vmatrix}
B(m)_{1,a} & B(m)_{1,b}\\
B(m)_{2,a} & B(m)_{2,b}
\end{vmatrix}
\right)_{(a,b)\in D}.\]
Proposition~\ref{prop:AinS1} states that $f(\bM)$ is a dense subset of $X$.
Note that if $o\in\operatorname{SO}_2(\bC)$, then $f(om)=f(m)$ for all $m$. Moreover, as explained in~\cite[Example 12.12]{FomZel2003}, $X$ is the GIT quotient $\bM//\operatorname{SO}_2(\bC)$ with $f$ the quotient map. Thus, $f(\bM)=X$.

Consider $p\in \mr X$ and $m\in f^{-1}(p)$. Since $p_{1,\ol 1}=m_{1,1}^2+m_{2,1}^2\neq 0$, we may find $o\in\operatorname{SO}_2(\bC)$ such that $o\left(\begin{smallmatrix}m_{1,1}\\m_{2,1}\end{smallmatrix}\right)=\left(\begin{smallmatrix}0\\c\end{smallmatrix}\right)$ for some $c\in\bC$. Hence, there exists $m'\in f^{-1}(p)$ with $m'_{1,1}=0$. 

For $m\in f^{-1}(\mr X)$, the columns of $B(m)$ define a tuple of pairwise distinct points $\alpha(m)\in(\bC\bP^1)^N$.
Now, choose $p\in\mr X(\bR)$ and $m\in f^{-1}(p)$ with $m_{1,1}=0$. Since $p_{1,\ol 1}=m_{2,1}^2$, we see that $m_{2,1}$ is either real or imaginary (and nonzero). Since $p_{1,i}=-m_{2,1}m_{1,i}$ and $p_{1,\ol i}=m_{2,1}m_{2,i}$ for $i\in[1,n]$, we deduce that all elements of $m$ other than $m_{1,1}$ differ from $m_{2,1}$ by a nonzero real factor. Since $f(m)=f(-m)$, we assume that $m_{1,2}$ is either in $\bR_{>0}$ or in $\mathrm i\bR_{>0}$. Note that $\alpha(m)_a\in\bR\bP^1$ for all $a\in N$, and the tuple $\alpha(m)$ defines a dihedral ordering $\la$ which is a CSDO.

We consider a second point $p'\in\mr X(\bR)$ and, again, choose $m'\in f^{-1}(p')$ with $m'_{1,1}=0$ and $m'_{1,2}\in \bR_{>0}\cup\mathrm i\bR_{>0}$. We define the CSDO $\lambda'$ similarly. First, suppose that all elements of $m$ and $m'$ are real, $\sgn m_{2,i}=\sgn m'_{2,i}$ for all $i\in[1,n]$ and, moreover, $\lambda=\lambda'$. We show that, in this case, $p$ and $p'$ lie in the same connected component of $\mr X(\bR)$. 

The conditions $m_{1,2},m'_{1,2}>0$ and $\sgn m_{2,2}=\sgn m'_{2,2}$ ensure that $\alpha(m)_2$ and $\alpha(m')_2$ lie in the same connected component of $\bR\bP^1\bs\{0,\infty\}$. In view of $\la=\la'$, the same applies to $\alpha(m)_a$ and $\alpha(m')_a$ for any $a\in N\bs\{1,\ol 1\}$. Thus, we may find a continuous path $(\alpha(m_s))_{0\le s\le1}$ in $(\bC\bP^1)^N$  such that:
\begin{itemize}
\item $\alpha(m_0)=\alpha(m)$,
\item the elements of $\alpha(m_s)$ are pairwise distinct for every $s$, 
\item $\alpha(m_s)_1=(0:1)$ for all $s$,
\item $\alpha(m_1)=\alpha(m')$.
\end{itemize}
In fact, we may assume that $(m_s)_{2,i}=m_{2,i}$ for all $s$ and $i$, while the elements $(m_s)_{1,i}, i>1$ vary continuously within $\bR$ without changing sign. Thus, $(m_s)_{0\le s\le 1}$ is a continuous path in the subspace of real matrices $\bM(\bR)$ with $m_0=m$. Applying $f$ to the latter path, we obtain a continuous path in $\mr X(\bR)$ between $f(m)=p$ and $f(m_1)$.

Since $\sgn{(m_s)_{2,i}}=\sgn m_{2,i}=\sgn m'_{2,i}$ for all $s$ and $i$, we can obtain $m'$ from $m_1$ by scaling each column by a positive real number. Hence, there is a path between $f(m_1)$ and $f(m')=p'$ in $\mr X(\bR)$ and, subsequently, between $p$ and $p'$.

A similar argument shows that $p$ and $p'$ lie in the same connected component if all elements of $m$ and $m'$ are imaginary, $\sgn(m_{2,i}/\mathrm i)=\sgn(m'_{2,i}/\mathrm i)$ for all $i\in[1,n]$ and $\lambda=\lambda'$. In each of the two cases, we have $2^n$ choices of signs in the second row and $2^{n-2}(n-1)!$ choices of the CSDO $\la=\la'$. Thus, $\mr X(\bR)$ has no more than $2^{2n-1}(n-1)!$ connected components and at most as many occurring sign patterns.

Denote the quotient map $\mr X\to\cM$ by $\om$, in coordinates it is given by~\eqref{eq:yab}. In~\eqref{eq:csdopreimage}, given a CSDO $\la$, we defined a point $p\in\mr X$ mapped to $\cM_{\nu(\la)}$ by $\om$. In this construction one may assume that all $\alpha_{1,i}$ and $\alpha_{2,i}$ are real, which entails $p\in\mr X(\bR)$. 
Hence, given a CSDO $\la$, we may consider a point $p\in \om^{-1}(\cM_{\nu(\la)})\cap\mr X(\bR)$. Consider the subgroup $G\subset H$ generated by $\{-1,1\}^n\subset(\bC^*)^n$, $(\mathrm i,\dots,\mathrm i)\in(\bC^*)^n$ and $\eta$. Since $(1,\dots,1)$ and $(-1,\dots,-1)$ act trivially, $|G|=2^{n+1}$. Note that for a non-identity element $g\in G$, the point $gp$ is obtained from $p$ by reversing the signs of some nonempty set of coordinates and, furthermore, $\om(gp)=p$. Since $H$ acts freely, we have $|Gp|=2^{n+1}$ and we conclude that at least $2^{n+1}$ distinct sign patterns occur in $\om^{-1}(\cM_{\nu(\la)})$. Furthermore, the sign pattern of $\om(p)$ is determined by that of $p$, hence no sign pattern can occur in both $\om^{-1}(\cM_{\nu(\la)})$ and $\om^{-1}(\cM_{\nu(\la')})$ for $\la\neq\la'$. Thus, there are at least $2^{2n-1}(n-1)!$ distinct sign patterns in $\mr X(\bR)$ and at least as many connected components. Combining with the previously obtained upper bounds, we deduce part (a).

The proof of part (b) is similar to that of Theorem~\ref{thm:TropM}. Consider a CSDO $\la$ and let $\tilde\nu$ be one of the $2^{n+1}$ sign patterns occurring in $\om^{-1}(\cM_{\nu(\la)})$. We are to check that 
\[\rho(\Trop_{\tilde\nu} X)=\Trop_{\nu(\la)}\cM,\] 
i.e., that $\init_w\epsilon_{\tilde\nu}(I)$ contains no elements of $\bR_{>0}[x_{a,b}]_{(a,b)\in D}$ if and only if $\init_{\rho(w)}\epsilon_{\nu(\la)}(\mr J)$ contains no elements of $\bR_{>0}[y^{\pm1}_{a,b}]_{(a,b)\in D^*}$. 
Note that $\nu(\la)_{a,b}=\frac{\tilde\nu_{a^\pp,b}\tilde\nu_{a,b^\pp}}{\tilde\nu_{a,b}\tilde\nu_{a^\pp,b^\pp}}$ and, hence, the restriction of $\epsilon_{\tilde\nu}$ to $\mr R^H$ coincides with $\epsilon_{\nu(\la)}$. If $\init_{\rho(w)}\epsilon_{\nu(\la)}(\mr p)\in\bR_{>0}[y^{\pm1}_{a,b}]_{(a,b)\in D^*}$ for $\mr p\in\mr J$ and a monomial $M\in R$ is such that $M\mr p\in I$, then 
\[\init_w\epsilon_{\tilde\nu}(M\mr p)=\epsilon_{\tilde\nu}(M)\init_{\rho(w)}\epsilon_{\nu(\la)}(\mr p)\in\pm\bR_{>0}[x_{a,b}]_{(a,b)\in D}.\] 
Conversely, if $\init_w\epsilon_{\tilde\nu}(p)\in\bR_{>0}[x_{a,b}]_{(a,b)\in D}$ for $H$-homogeneous $p\in I$, then, for a monomial $M$ occurring in $p$, we have $p/M\in\mr J$ and 
\[\init_{\rho(w)}\epsilon_{\nu(\la)}(p/M)=\epsilon_{\tilde\nu}(M)^{-1}\init_w\epsilon_{\tilde\nu}(p)\in\pm\bR_{>0}[y^{\pm1}_{a,b}]_{(a,b)\in D^*}.\qedhere\] 
\end{proof}

\section{Proof of Lemma~\ref{lem:keylemma}}\label{sec:keylemmaproof}

The proof will be by induction and, to that end, we consider phylogentic trees labeled by subsets of $N$. For nonempty $K\subset N$, a \textit{$K$-labeled phylogenetic tree} $(T,\vv)$ consists of an abstract tree $T$ with $|K|$ leaves and no vertices of degree 2, and a bijection $\vv$ from $K$ to the set of leaf vertices of $T$. 

Let $P_k$ denote a regular $k$-gon. Consider $K\subset N$ with $|K|\ge3$, a subdivision of $P_{|K|}$ formed by the set of diagonals $\Theta$, and a bijection $\ph$ from $K$ to the set of sides of $P_{|K|}$ (a \textit{$K$-labeling}). As before, this data defines a $K$-labeled phylogenetic tree $\cT_{\Theta,\ph}$.

For even $k\ge 4$, we distinguish one longest diagonal of $P_k$, denoted by $\delta_{0,k}$. A subdivision of $P_k$ is axially symmetric if with every diagonal it also contains its mirror image with respect to $\delta_{0,k}$. A subset $K\subset N$ is symmetric if with every $a$ it also contains $\ol a$. A $K$-labeling $\ph$ of $P_{|K|}$ is axially symmetric if $K$ is symmetric and for every $a\in K$ the sides $\ph(a)$ and $\ph(\ol a)$ are symmetric to each other with respect to $\delta_{0,|K|}$. 

\begin{defn}
A $K$-labeled phylogenetic tree $(T,\vv)$ is a \textit{$K$-labeled ASPT} if $K$ is symmetric and either 
\begin{itemize}
\item $|K|=2$ or 
\item $(T,\vv)=\cT_{\Theta,\ph}$ for an axially symmetric subdivision $\Theta$ and an axially symmetric labeling $\ph$ of $P_{|K|}$.  
\end{itemize}
\end{defn}

A $K$-labeled ASPT $(T,\vv)$ is again equipped with a unique involution of $T$ that swaps $\vv(a)$ and $\vv(\ol a)$: its symmetry.

We also consider \textit{$K$-labeled weighted phylogenetic trees}: triples $(T,\vv,l)$ with $(T,\vv)$ a $K$-labeled phylogenetic tree and $l$ a real-valued function on the edges of $T$ that is positive on non-leaf edges. As before, for a $K$-labeled weighted phylogenetic tree $(T,\vv,l)$ we obtain a function $d_{T,\vv,l}\colon K^2\to\bR$. A $K$-labeled weighted phylogenetic tree $(T,\vv,l)$ is a \textit{$K$-labeled ASWPT} if $(T,\vv)$ is a $K$-labeled ASPT and $l$ is invariant under the symmetry.


\begin{proof}[Proof of Lemma~\ref{lem:keylemma}]
Consider a nonempty symmetric $K\subset N$. We prove by induction on $|K|$ that a $K$-labeled weighted phylogenetic tree $(T,\vv,l)$, such that (i) holds for every $a,b\in K$ and (ii) holds for every $i,j,k\in K\cap[1,n]$, is a $K$-labeled ASWPT. 

The base $|K|=2$ is trivial.
We proceed to the induction step. Suppose that $|K|\ge 4$ and we have verified our claim for all symmetric $K'$ with $|K'|<|K|$.
If $T$ has only one non-leaf vertex, property (i) implies that $l(e_a)+l(e_b)=l(e_{\ol a})+l(e_{\ol b})$ for any $a,b\in K$. This, in turn, means that $l(e_a)=l(e_{\ol a})$ for all $a\in K$, which implies that $(T,\vv,l)$ is a $K$-labeled ASWPT.

We assume that $T$ has at least two non-leaf vertices. Our first goal is to reduce to the case when there is a pair $a_0,b_0\in K$ with $|a_0|\neq|b_0|$ such that $\vv(a_0)$ and $\vv(b_0)$ have a common neighbor. 
In view of our assumption, there exist two distinct vertices $u$ and $v$, each of which has only one non-leaf neighbor and at least two leaf neighbors. We suppose that the only leaves adjacent to $u$ are $\vv(a)$ and $\vv(\ol a)$ for some $a\in K$, and the only leaves adjacent to $v$ are $\vv(b)$ and $\vv(\ol b)$; otherwise, we have found the desired pair. For $c\in K$, denote the edge incident to $\vv(c)$ by $e_c$. Let $d_{u,v}$ be the (weighted) distance between $u$ and $v$. By property (i), we have 
\[l(e_a)+d_{u,v}+l(e_b)=l(e_{\ol{a}})+d_{u,v}+l(e_{\ol b})\quad\text{and}\quad l(e_a)+d_{u,v}+l(e_{\ol b})=l(e_{\ol{a}})+d_{u,v}+l(e_{b}).\]
This implies that $l(e_a)=l(e_{\ol a})$ and $l(e_b)=l(e_{\ol b})$. Consider the set $V=\vpath(u,v)$ of vertices lying on the path between $u$ and $v$. First, suppose that every $v'\in V$ only has neighbors that are either leaves or also lie in $V$. This possibility is depicted in Figure~\ref{fig:uvtwocases}A. In this case, $T$ is a caterpillar tree consisting entirely of $V$ and the leaves adjacent to vertices in $V$. Furthermore, for every $c\in K$ the leaves $\vv(c)$ and $\vv(\ol c)$ are adjacent to the same vertex in $V$, otherwise we would have a contradiction with property (i): $\mathbf{d}(c,a)\neq \mathbf{d}(\ol c,\ol a)$, where we denote $\mathbf{d}=d_{T,\vv,l}$ for brevity. To show that $(T,\vv,l)$ is a $K$-labeled ASWPT, it remains to check that $l(e_c)=l(e_{\ol c})$ for every $c\in K$. This is implied by the equalities $\mathbf{d}(c,a)=\mathbf{d}(\ol c,\ol a)$, $\mathbf{d}(c,\ol a)=\mathbf{d}(\ol c,a)$ and $l(e_a)=l(e_{\ol a})$.

\begin{figure*}[h!tbp]
\begin{subfigure}[t]{.45\textwidth}
\centering
\begin{tikzpicture}[x=7mm, y=7mm]
\node[circle, fill=black, inner sep=2] (u) at (0,0) {};
\node[circle, fill=black, inner sep=2] (v) at (5,0) {};
\node[circle, fill=black, inner sep=2] (a) at (-.7,.7) {};
\node[circle, fill=black, inner sep=2] (ola) at (-.7,-.7) {};
\node[circle, fill=black, inner sep=2] (b) at (5.7,.7) {};
\node[circle, fill=black, inner sep=2] (olb) at (5.7,-.7) {};
\node[circle, fill=black, inner sep=2] (v1) at (1.5,0) {};
\node[circle, fill=black, inner sep=2] (v2) at (3.5,0) {};
\node[circle, fill=black, inner sep=2] (c) at (1.5,1) {};
\node[circle, fill=black, inner sep=2] (olc) at (1.5,-1) {};
\node[circle, fill=black, inner sep=2] (d) at (3.5,1) {};
\node[circle, fill=black, inner sep=2] (old) at (3.5,-1) {};
\node[circle, fill=black, inner sep=2] (e) at (4.2,.7) {};
\node[circle, fill=black, inner sep=2] (ole) at (4.2,-.7) {};

\node at (-1.4,.7) {$\vv(a)$};
\node at (-1.4,-.7) {$\vv(\ol a)$};
\node at (0.1,0.4) {$u$};
\node at (4.9,0.4) {$v$};
\node at (6.4,.7) {$\vv(b)$};
\node at (6.4,-.7) {$\vv(\ol b)$};
\node at (0.9,1.3) {$\vv(c)$};
\node at (0.9,-1.3) {$\vv(\ol c)$};

\draw (a) -- (u) -- (v1) -- (c)
      (ola) -- (u)
      (v1) -- (olc)
      (b) -- (v) -- (v2) -- (d)
      (olb) -- (v)
      (e) -- (v2) --(ole)
      (v2) -- (old);
\draw[dashed] (v1) -- (v2);
\end{tikzpicture}
\caption{All vertices outside of $V$ are leaves.}
\end{subfigure}
~
\begin{subfigure}[t]{.45\textwidth}
\centering
\begin{tikzpicture}[x=8mm, y=8mm]
\node[circle, fill=black, inner sep=2] (u) at (0,0) {};
\node[circle, fill=black, inner sep=2] (v) at (3.5,0) {};
\node[circle, fill=black, inner sep=2] (a) at (-.7,.7) {};
\node[circle, fill=black, inner sep=2] (ola) at (-.7,-.7) {};
\node[circle, fill=black, inner sep=2] (b) at (4.2,.7) {};
\node[circle, fill=black, inner sep=2] (olb) at (4.2,-.7) {};
\node[circle, fill=black, inner sep=2] (v1) at (1.5,0) {};
\node[circle, fill=black, inner sep=2] (v2) at (1.5,1) {};
\node[circle, fill=black, inner sep=2] (v3) at (3,1.5) {};
\node[circle, fill=black, inner sep=2] (c) at (2.3,2.2) {};
\node[circle, fill=black, inner sep=2] (d) at (3.7,2.2) {};

\node at (-1.3,.7) {$\vv(a)$};
\node at (-1.3,-.7) {$\vv(\ol a)$};
\node at (0.1,.3) {$u$};
\node at (3.4,-.3) {$v$};
\node at (4.8,.7) {$\vv(b)$};
\node at (4.8,-.7) {$\vv(\ol b)$};
\node at (1.5,-.4) {$v_1$};
\node at (1.1,1.1) {$v_2$};
\node at (3.1,1.1) {$v_3$};
\node at (1.7,2.2) {$\vv(c)$};
\node at (4.35,2.2) {$\vv(d)$};

\draw (a) -- (u) -- (ola)
      (b) -- (v) -- (olb)
      (v1) -- (v2)
      (v3) -- (c)
      (v3) -- (d);
\draw[dashed] (u) -- (v1)
              (v1) -- (v)
              (v2) to[in=-140, out=40] (v3);
\end{tikzpicture}
\caption{Some $v_1\in V$ has a non-leaf neighbor $v_2\notin V$. Here $v_2$ and $v_3$ may coincide.}
\end{subfigure}
\caption{}
\label{fig:uvtwocases}
\vspace{-1mm}
\end{figure*}

Now, suppose that some vertex $v_1\in V$ has a non-leaf neighbor $v_2\notin V$, this possibility is shown in Figure~\ref{fig:uvtwocases}B. If we delete the edge between $v_1$ and $v_2$, the component containing $v_2$ will contain two leaves $\vv(c)$ and $\vv(d)$ that have a common neighbor $v_3$. If $c\neq\ol d$ then $c,d$ is the desired pair. If $c=\ol d$, then, similarly to the above, property (i) provides $l(e_c)=l(e_{\ol c})$. Denote by $d_a$ the distance between $v_1$ and $u$, by $d_b$ the distance between $v_1$ and $v$ and by $d_c$ the distance between $v_1$ and $v_3$; all three distances are positive. Set $\{i,j,k\}=\{|a|,|b|,|c|\}$ and consider the four quantities in property (ii). The first three quantities are, in some order, equal to $2(d_a+d_b+l(e_a)+l(e_b)+l(e_c))$, $2(d_a+d_c+l(e_a)+l(e_b)+l(e_c))$ and $2(d_b+d_c+l(e_a)+l(e_b)+l(e_c))$. Meanwhile, the fourth quantity is equal to $2(d_a+d_b+d_c+l(e_a)+l(e_b)+l(e_c))$ and is strictly greater than the other three, contradicting property (ii). 

Thus, we may assume that $T$ has leaves $\vv(a_0)$ and $\vv(b_0)$ with $|a_0|\neq|b_0|$ that have a common neighbor $v_0$. We claim that $\vv(\ol{a_0})$ and $\vv(\ol{b_0})$ also have a common neighbor, which we denote $\ol v_0$. Indeed, two arbitrary leaves $\vv(a')$ and $\vv(b')$ have a common neighbor if and only if the difference $\mathbf{d}(a',c')-\mathbf{d}(b',c')$ is the same for all $c'\notin\{a',b'\}$. If this condition is satisfied when $a'=a_0$ and $b'=b_0$, it must also be satisfied when $a'=\ol{a_0}$ and $b'=\ol{b_0}$ by property (i). In addition, by the argument already used above, property (i) implies that $l(e_{a_0})=l(e_{\ol{a_0}})$ and $l(e_{b_0})=l(e_{\ol{b_0}})$.

Furthermore, $v_0$ has degree 3 if and only if $\ol v_0$ has degree 3. Indeed, $v_0$ has degree 4 or more if and only if there is a pair $c,d\in K\bs\{a_0,b_0\}$ such that $v_0\in\vpath(\vv(c),\vv(d))$, i.e.,
\[\mathbf{d}(c,d)=\mathbf{d}(c,a_0)+\mathbf{d}(d,a_0)-2l(e_{a_0}).\]
Meanwhile, $\ol v_0$ has degree 4 or more if and only if there are $c,d\in K\bs\{\ol{a_0},\ol{b_0}\}$ such that 
\[\mathbf{d}(\ol c,\ol d)=\mathbf{d}(\ol c,\ol{a_0})+\mathbf{d}(\ol d,\ol{a_0})-2l(e_{\ol{a_0}}).\]
By property (i), the two conditions are equivalent.

Set $K'=K\bs\{a_0,\ol{a_0}\}$. Let $T'$ be the subtree of $T$ obtained by deleting $\vv(a_0)$ and $\vv(\ol{a_0})$, while $\vv'$ and $l'$ are the restrictions of $\vv$ and $l$ to, respectively, $K\bs\{a,\ol a\}$ and the edge set of $T'$. If $v_0$ and $\ol v_0$ have degree 4 or more in $T$ (in particular, if $v_0=\ol v_0$), then $(T',\vv',l')$ is a $K'$-labeled weighted phylogenetic tree. However, we have $d_{T',\vv',l'}(c,d)=\mathbf{d}(c,d)$ for any $c,d\in K'$. Therefore, properties (i) and (ii) hold for $(T',\vv',l')$ and, by the induction hypothesis, $(T',\vv',l')$ is a $K'$-labeled ASWPT. If $\sigma$ denotes its symmetry, then $\sigma(\vv'(b_0))=\vv'(\ol{b_0})$, hence $\sigma(v_0)=\ol v_0$. Now, $(T,\vv,l)$ is obtained from $(T',\vv',l')$ by adding two leaves, one adjacent to $v_0$, the other to $\ol v_0$, labeling them by $a_0$ and $\ol{a_0}$, and assigning the same weight $l(e_{a_0})$ to the two new edges. Therefore, $(T,\vv,l)$ is a $K$-labeled ASWPT.

Suppose that $v_0$ and $\ol v_0$ have degree 3 in $T$. Then, in $T'$ these two vertices have degree 2. If $|K|=4$, the equalities $l(e_{a_0})=l(e_{\ol{a_0}})$ and $l(e_{b_0})=l(e_{\ol{b_0}})$ already ensure that $(T,\vv,l)$ is a $K$-labeled ASWPT. We assume $|K|\ge6$, in particular, $v_0$ and $\ol v_0$ are non-adjacent. We consider the tree $T''$ obtained from $T'$ by replacing each vertex of degree 2 together with its incident edges by a single edge (i.e., \textit{smoothing} the tree). The leaves of $T''$ are in natural bijection with the leaves of $T'$, we define $\vv''$ by composing $\vv'$ with this bijection. For every edge $e$ of $T''$, other than the two edges created by smoothing, we set $l''(e)$ equal to the weight of the respective edge in $T'$ and $T$. Finally, if, say, smoothing replaced edges $e_1$ and $e_2$ with $e$, we set $l''(e)=l'(e_1)+l'(e_2)=l(e_1)+l(e_2)$. We obtain a $K'$-labeled phylogenetic tree $(T'',\vv'',l'')$. For any $c,d\in K'$ we have $d_{T'',\vv'',l''}(c,d)=\mathbf{d}(c,d)$ and, by the induction hypothesis, $(T'',\vv'',l'')$ is a $K'$-labeled ASWPT. To obtain $(T,\vv,l)$ from $(T'',\vv'',l'')$, one subdivides two mutually symmetric leaf edges into pairs of edges of the same two weights, and then attaches leaves $\vv(a_0)$ and $\vv(\ol{a_0})$ to the two created vertices by edges of weight $l(e_{a_0})$. Hence, $(T,\vv,l)$ is again a $K$-labeled ASWPT. 
\end{proof}

To conclude our discussion of Lemma~\ref{lem:keylemma}, we give the alternative criterion mentioned in Remark~\ref{rem:keylemma}; it is an immediate consequence of the lemma. Recall that a weighted phylogenetic tree is uniquely determined by the pairwise distances between its leaves.

\begin{cor}\label{cor:altcriterion}
Consider a weighted phylogenetic tree $(T,\vv,l)$. For a triple $i,j,k\in[1,n]$, let $\mathbf T_{i,j,k}$ denote the unique $\{i,j,k,\ol i,\ol j,\ol k\}$-labeled weighted phylogenetic tree, such that 
$d_{\mathbf T_{i,j,k}}(a,b)=d_{T,\vv,l}(a,b)$
for any $a,b\in\{i,j,k,\ol i,\ol j,\ol k\}$. Then, $(T,\vv,l)$ is axially symmetric if and only if every $\mathbf T_{i,j,k}$ is axially symmetric.
\end{cor}

\section{Proof of Lemma~\ref{lem:initinker}}\label{sec:initinkerproof}

We first prove an auxiliary statement and then deduce Lemma~\ref{lem:initinker} and Corollary~\ref{cor:gbasis}.

\begin{lemma}\label{lem:initforms}
Let $(T,\vv)$ be a maximal ASPT equal to $\cT_{\Theta,\ph_0}$ for an axially symmetric subdivision $\Theta$. Choose $w\in C_{T,\vv}$ and a quadruple  $a\dotprec b\dotprec c\dotprec d$ in $N$. The initial form $\init_{w}r_{a,b,c,d}$ contains the monomial $x_{a,c}x_{b,d}$ with a nonzero coefficient and, furthermore, lies in $\ker\psi_{T,\vv}$. 
\end{lemma}

\begin{proof}
Let $T_{a,b,c,d}$ denote the minimal subtree of $T$ containing the vertices $\vv(a)$, $\vv(b)$, $\vv(c)$, $\vv(d)$. This subtree has four leaves and either 
\begin{enumerate}[label=(\Alph*)]
\item two vertices of degree 3 or
\item one of vertex of degree 4,
\end{enumerate}
with all remaining vertices having degree 2. The structure of $T_{a,b,c,d}$ determines $\init_{w}r_{a,b,c,d}$: if (B) holds, then $\init_{w}r_{a,b,c,d}=r_{a,b,c,d}$; if (A) holds, $\init_{w}r_{a,b,c,d}$ is a binomial. More explicitly, for a partition $\{a,b,c,d\}=\{a',b'\}\sqcup\{b',d'\}$, the term $\pm x_{a',b'}x_{c',d'}$ is contained in $\init_{w}r_{a,b,c,d}$ if and only if $\vpath(\vv(a'),\vv(b'))$ and $\vpath(\vv(c'),\vv(d'))$ intersect. This is always true in case (B) and true for two out of three partitions in case (A). The realization $(T,\vv)=\cT_{\Theta,\ph_0}$ ensures that $\vpath(\vv(a),\vv(c))$ and $\vpath(\vv(b),\vv(d))$ intersect, confirming the first claim. We prove that $\psi_{T,\vv}(\init_{w}r_{a,b,c,d})=0$. We assume that $a,b\in[1,n]$, otherwise we can replace $(a,b,c,d)$ with $(\ol d,\ol c,\ol b,\ol a)$.

First, suppose that $T_{a,b,c,d}$ has a vertex of degree 4, i.e., $\init_{w}r_{a,b,c,d}=r_{a,b,c,d}$. 
In this case, $(T,\vv)$ must have type (II) and the vertex $v$ of degree 4 must be fixed by the symmetry $\sigma$. As seen in the proof of Proposition~\ref{prop:maxASPTprops}, two of the neighbors of $v$ are also $\sigma$-fixed, while the remaining two are exchanged by $\sigma$. Furthermore, one of $\vpath(\vv(a),\vv(c))$ and $\vpath(\vv(b),\vv(d))$ contains both $\sigma$-fixed neighbors, while the other contains the remaining two neighbors. We assume that $\vpath(\vv(a),\vv(c))$ contains the $\sigma$-fixed neighbors of $v$, the other case is similar.
Since $\vpath(\vv(b),\vv(d))$ contains exactly one $\sigma$-fixed vertex, $d\notin[1,n]$ by Proposition~\ref{prop:maxASPTprops}(a). 
If $c\in[1,n]$, then $a\prec b\prec c\prec d$. We have $c_{T,\vv}(a,b)=c_{T,\vv}(a,c)=c_{T,\vv}(b,c)=\mathrm i$, $c_{T,\vv}(a,d)=c_{T,\vv}(c,d)=1$ and $c_{T,\vv}(b,d)=2$. Consequently,
\[\psi_{T,\vv}(r_{a,b,c,d})=(c_{T,\vv}(a,b)c_{T,\vv}(c,d)+c_{T,\vv}(a,d)c_{T,\vv}(b,c)-c_{T,\vv}(a,c)c_{T,\vv}(b,d))\prod_{e\text{ in }T_{a,b,c,d}}t_e=0.\]
If $c\notin[1,n]$, then $a\prec b\prec d\prec c$. In this case, $c_{T,\vv}(a,b)=c_{T,\vv}(c,d)=\mathrm i$, $c_{T,\vv}(a,c)=c_{T,\vv}(a,d)=c_{T,\vv}(b,c)=1$ and $c_{T,\vv}(b,d)=2$. Hence,
\[\psi_{T,\vv}(r_{a,b,c,d})=(c_{T,\vv}(a,b)c_{T,\vv}(c,d)+c_{T,\vv}(a,c)c_{T,\vv}(b,d)-c_{T,\vv}(a,d)c_{T,\vv}(b,c))\prod_{e\text{ in }T_{a,b,c,d}}t_e=0.\]


Now, suppose that $T_{a,b,c,d}$ contains two vertices of degree 3. This means that there are precisely two partitions 
$\{a,b,c,d\}=\{a',b'\}\sqcup\{c',d'\}$
for which $\vpath(\vv(a'),\vv(b'))$ and $\vpath(\vv(c'),\vv(d'))$ intersect. One of these partitions is $\{a,c\}\sqcup\{b,d\}$, denote the other by $\{a_1,b_1\}\sqcup\{c_1,d_1\}$. We may assume that $a_1=a$, so that $b_1=b$ or $b_1=d$. We have 
\[\init_{w} r_{a,b,c,d}=\pm x_{a,c}x_{b,d}\pm x_{a_1,b_1}x_{c_1,d_1}.\]

Observe that $\psi_{T,\vv}(x_{a,c}x_{b,d})$ and $\psi_{T,\vv}(x_{a_1,b_1}x_{c_1,d_1})$ differ only by a scalar factor---we check that the total coefficient is zero. If $d\in[1,n]$, then 
\begin{equation}\label{eq:initbinomial}
a\prec b\prec c\prec d\quad\text{and}\quad\init_{w} r_{a,b,c,d}=x_{a_1,b_1}x_{c_1,d_1}-x_{a,c}x_{b,d}.    
\end{equation}
In this case, $c_{T,\vv}(a_1,b_1)=c_{T,\vv}(c_1,d_1)=c_{T,\vv}(a,c)=c_{T,\vv}(b,d)=\mathrm i$, and the total coefficient is indeed zero.

If $c\in[1,n]$ but $d\notin[1,n]$, then~\eqref{eq:initbinomial} applies again. We assume that $b_1=b$, the case $b_1=d$ is similar. We have $c_{T,\vv}(a,c)=c_{T,\vv}(a,b)=\mathrm i$, while $c_{T,\vv}(b,d)$ and $c_{T,\vv}(c,d)$ can equal either 1 or 2. We are to check that $c_{T,\vv}(b,d)=c_{T,\vv}(c,d)$. 
If $\vpath(\vv(b),\vv(c))$ contains no $\sigma$-fixed vertices (Figure~\ref{fig:dnotin1n}A), then a $\sigma$-fixed vertex lies in $\vpath(\vv(b),\vv(d))$ if and only if it lies in $\vpath(\vv(c),\vv(d))$, implying $c_{T,\vv}(b,d)=c_{T,\vv}(c,d)$, as desired. 

If, however, $\vpath(\vv(b),\vv(c))$ contains a $\sigma$-fixed vertex, then $(T,\vv)$ has type (II). For $i\in\{a,b,c,d\}$ denote the unique $\sigma$-fixed vertex in $\vpath(\vv(i),\vv(\ol i))$ by $v_i$, see Figure~\ref{fig:dnotin1n}B. Since $\vpath(\vv(a),\vv(d))\cap\vpath(\vv(b),\vv(c))=\varnothing$ and the former path contains $v_d$, we have $v_d\notin\{v_b,v_c\}$. Consequently, both $\vpath(\vv(b),\vv(d))$ and $\vpath(\vv(c),\vv(d))$ contain at least two $\sigma$-fixed vertices, providing $c_{T,\vv}(b,d)=c_{T,\vv}(c,d)=1$.

\begin{figure*}[h!tbp]
\begin{subfigure}[t]{.45\textwidth}
\centering
\begin{tikzpicture}[x=10mm, y=10mm]
\node[circle, fill=black, inner sep=2] (v1) at (0,0) {};
\node[circle, fill=black, inner sep=2] (d) at (1.5,0) {};
\node[circle, fill=black, inner sep=2, red] (v2) at (-1,-.5) {};
\node[circle, fill=black, inner sep=2, red] (b) at (-2,0) {};
\node[circle, fill=black, inner sep=2, red] (c) at (-2,-1) {};
\node[circle, fill=black, inner sep=2] (a) at (-1,1.5) {};

\node at (2,0) {$\vv(d)$};
\node at (-1.5,1.5) {$\vv(a)$};
\node at (-2.5,0) {$\vv(b)$};
\node at (-2.5,-1) {$\vv(c)$};

\draw[dashed] (v1) to[in=-60, out=90] (a)
              (v1) to[in=-160, out=-20] (d)
              (v1) to[in=40, out=-140] (v2);
\draw[dashed, red] (v2) to[in=0, out=150] (b)
                   (v2) to[in=0, out=-150] (c);
\end{tikzpicture}
\caption{There are no $\sigma$-fixed vertices in the path between $\vv(b)$ and $\vv(c)$, highlighted in red.}
\end{subfigure}
~
\begin{subfigure}[t]{.45\textwidth}
\centering
\begin{tikzpicture}[x=10mm, y=10mm]
\node[circle, fill=black, inner sep=2] (vb) at (0,0) {};
\node[circle, fill=black, inner sep=2] (vc) at (0,-.75) {};
\node[circle, fill=black, inner sep=2] (d) at (1.5,.5) {};
\node[circle, fill=black, inner sep=2] (vd) at (0,1) {};
\node[circle, fill=black, inner sep=2] (b) at (-2,.5) {};
\node[circle, fill=black, inner sep=2] (c) at (-2,-1.25) {};
\node[circle, fill=black, inner sep=2] (a) at (-1.5,1.5) {};

\node at (2,.5) {$\vv(d)$};
\node at (-2,1.5) {$\vv(a)$};
\node at (-2.5,.5) {$\vv(b)$};
\node at (-2.5,-1.25) {$\vv(c)$};
\node at (.4,0) {$v_b$};
\node at (.4,-.75) {$v_c$};
\node at (-.4,1) {$v_d$};

\draw[dashed] (vb) -- (vc)
              (vc) -- (vd)
              (vb) to[in=-20, out=135] (b)
              (vc) to[in=20, out=-135] (c)
              (vd) to[in=120, out=0] (d)
              (vd) to[in=20, out=90] (a);
\end{tikzpicture}
\caption{In the contrary case, note that $v_a$ may lie anywhere above $v_b$ but its position does not affect our argument.}
\end{subfigure}
\caption{The case $a,b,c\in[1,n]$, $d\notin[1,n]$, assuming $b_1=b$.}
\label{fig:dnotin1n}
\vspace{-1mm}
\end{figure*}

Finally, if $c,d\notin[1,n]$, then $a\prec b\prec d\prec c$. Let $v_1$ and $v_2$ denote vertices of degree three in $T_{a,b,c,d}$. If $b_1=b$ (Figure~\ref{fig:сnotin1n}A), then
\[\init_{w} r_{a,b,c,d}=x_{a,c}x_{b,d}+x_{a,b}x_{c,d}\]
and $c_{T,\vv}(a,b)=c_{T,\vv}(c,d)=\mathrm i$.
Both $\vpath(\vv(a),\vv(b))$ and $\vpath(\vv(c),\vv(d))$ contain each of $v_1$ and $v_2$. By Proposition~\ref{prop:maxASPTprops}(b), both $v_1$ and $v_2$ are $\sigma$-fixed. Since $v_1$ and $v_2$ also lie in $\vpath(\vv(a),\vv(c))$ and $\vpath(\vv(b),\vv(d))$, we conclude that $c_{T,\vv}(a,c)=c_{T,\vv}(b,d)=1$, which again provides a zero total.

\begin{figure*}[h!tbp]
\begin{subfigure}[t]{.45\textwidth}
\centering
\begin{tikzpicture}[x=10mm, y=10mm]
\node[circle, fill=black, inner sep=2] (v1) at (0,1) {};
\node[circle, fill=black, inner sep=2] (v2) at (0,0) {};
\node[circle, fill=black, inner sep=2] (a) at (-1.5,1.5) {};
\node[circle, fill=black, inner sep=2] (b) at (-1.5,-.5) {};
\node[circle, fill=black, inner sep=2] (c) at (1.5,0) {};
\node[circle, fill=black, inner sep=2] (d) at (1.5,1) {};

\node at (-.4,1) {$v_1$};
\node at (-.4,0.1) {$v_2$};
\node at (-2,1.5) {$\vv(a)$};
\node at (-2,-.5) {$\vv(b)$};
\node at (2,0) {$\vv(c)$};
\node at (2,1) {$\vv(d)$};

\draw[dashed] (v1) to[in=-30, out=90] (a)
              (v1) to[in=-160, out=20] (d)
              (v1) -- (v2)
              (v2) to[in=0, out=-160] (b)
              (v2) to[in=-150, out=-30] (c);
\end{tikzpicture}
\caption{If $b_1=b$, then $v_1$ and $v_2$ are $\sigma$-fixed.}
\end{subfigure}
~
\begin{subfigure}[t]{.45\textwidth}
\centering
\begin{tikzpicture}[x=10mm, y=10mm]
\node[circle, fill=black, inner sep=2] (v1) at (0,0) {};
\node[circle, fill=black, inner sep=2] (v2) at (1,.5) {};
\node[circle, fill=black, inner sep=2] (a) at (-1.5,.5) {};
\node[circle, fill=black, inner sep=2] (b) at (-1.5,-.5) {};
\node[circle, fill=black, inner sep=2] (c) at (2.5,1) {};
\node[circle, fill=black, inner sep=2] (d) at (2.5,0) {};

\node at (.1,-.3) {$v_1$};
\node at (.8,.8) {$v_2$};
\node at (-2,.5) {$\vv(a)$};
\node at (-2,-.5) {$\vv(b)$};
\node at (3,1) {$\vv(d)$};
\node at (3,0) {$\vv(c)$};

\draw[dashed] (v1) to[in=-30, out=90] (a)
              (v1) to[in=10, out=-140] (b)
              (v1) to[in=-160, out=20] (v2)
              (v2) to[in=180, out=60] (c)
              (v2) to[in=-160, out=-80] (d);

\node at (0,-1) {};              
\end{tikzpicture}
\caption{If $b_1=d$, then $T_{a,b,c,d}$ can contain zero, one or multiple $\sigma$-fixed vertices.}
\end{subfigure}
\vspace{-1mm}
\caption{The case $a,b\in[1,n]$, $c,d\notin[1,n]$.}
\label{fig:сnotin1n}
\end{figure*}

If, however, $b_1=d$ (Figure~\ref{fig:сnotin1n}B), then 
\[\init_{w} r_{a,b,c,d}=x_{a,c}x_{b,d}-x_{a,d}x_{b,c}.\]
If $T_{a,b,c,d}$ contains no $\sigma$-fixed vertices, $c_{T,\vv}(a,c)=c_{T,\vv}(b,d)=c_{T,\vv}(a,d)=c_{T,\vv}(b,c)=1$. 
Let $T_{a,b,c,d}$ contain $\sigma$-fixed vertices. Then, $(T,\vv)$ has type II and the $\sigma$-fixed vertices in $T_{a,b,c,d}$ form a path subgraph that intersects all of $\vpath(\vv(a),\vv(c))$, $\vpath(\vv(a),\vv(d))$, $\vpath(\vv(b),\vv(c))$ and $\vpath(\vv(b),\vv(d))$. In particular, at least one $\sigma$-fixed vertex must lie in $\vpath(v_1,v_2)$.
If $T_{a,b,c,d}$ contains a single $\sigma$-fixed vertex, then all four coefficients are equal to 2. 
Suppose that $T_{a,b,c,d}$ contains multiple $\sigma$-fixed vertices. If at least two of them are in $\vpath(v_1,v_2)$, then all four coefficients are 1. If exactly one $\sigma$-fixed vertex lies in $\vpath(v_1,v_2)$, then it coincides with $v_1$ or $v_2$. If it is $v_1$, then $c_{T,\vv}(a,c)=c_{T,\vv}(a,d)$ and $c_{T,\vv}(b,c)=c_{T,\vv}(b,d)$. If it is $v_2$, then $c_{T,\vv}(a,c)=c_{T,\vv}(b,c)$ and $c_{T,\vv}(a,d)=c_{T,\vv}(b,d)$.  
\end{proof}

\begin{proof}[Proof of Lemma~\ref{lem:initinker}]
We recall the weight $w'\in\bR^D$ considered in the proof of Lemma~\ref{lem:ImonItor}. It can be defined by $w'_{a,b}=\ln(k+1)$ where $k$ is the number of elements strictly between $a$ and $b$ in the order $\dotprec $. For a quadruple $a\dotprec b\dotprec c\dotprec d$, we have $\init_{w'}r_{a,b,c,d}=\pm x_{a,c}x_{b,d}$.

Let $J\subset\init_w I$ be the ideal generated by the initial forms $\init_{w}r_{a,b,c,d}$. By Lemma~\ref{lem:initforms},
\[\init_{w'}(\init_{w}r_{a,b,c,d})=\init_{w'}r_{a,b,c,d}=\pm x_{a,c}x_{b,d}.\]
Hence, $\init_{w'} J$ contains the ideal $I'_\mon$ generated by all $x_{a,c}x_{b,d}$. Thus, $\grdim J\ge\grdim I'_\mon$, where the latter is equal to $\grdim I$ by Lemma~\ref{lem:ImonI}. However, $\grdim J\le\grdim I$ by construction. We conclude that $J=\init_{w} I$ and, in view of Lemma~\ref{lem:initforms}, $\init_w I\subset\ker\psi_{T,\vv}$.
\end{proof}

We have also shown that $\init_w I$ is generated by the $\init_w r_{a,b,c,d}$, obtaining Corollary~\ref{cor:gbasis}.

\bibliographystyle{plainurl}
\bibliography{refs.bib}

\end{document}